\documentclass[11pt,letterpaper]{amsart}

\usepackage{amsmath,amssymb,amsthm,verbatim, csquotes,mathrsfs,stmaryrd}
\usepackage{hyperref}
\usepackage{xcolor}
\usepackage{esint}
\usepackage{mathtools}
\usepackage{graphicx}
\usepackage{float}
\usepackage{caption}
\mathtoolsset{showonlyrefs}
\usepackage{geometry}

\usepackage{bm}
\usepackage{pdfpages}
\usepackage[shortlabels]{enumitem}

\newtheorem{theorem}{Theorem}[section]\newtheorem{lemma}[theorem]{Lemma}\newtheorem{proposition}[theorem]{Proposition}\newtheorem{definition}[theorem]{Definition}\newtheorem{corollary}[theorem]{Corollary}\newtheorem{remark}[theorem]{Remark}

\makeatletter
\def\l@subsection{\@tocline{2}{0pt}{1.5em}{5pc}{}}
\makeatother

\def\om{\omega}
\def\Om{\Omega}
\def\p{\partial}

\def\de{\delta}
\def\De{\Delta}

\def\S{{\Sigma}}
\def\<{\langle}
\def\>{\rangle}

\def\na{\nabla}

\def\spt{{\rm spt}}

\providecommand{\abs}[1]{\lvert#1\rvert}
\providecommand{\Abs}[1]{\left\lvert#1\right\rvert}

\providecommand{\norm}[1]{\lVert#1\rVert}

\newcommand{\mbE}{\mathbb{E}}
\newcommand{\mbF}{\mathbb{F}}

\newcommand{\mbN}{\mathbb{N}}

\newcommand{\mbR}{\mathbb{R}}
\newcommand{\mbS}{\mathbb{S}}

\newcommand{\mbZ}{\mathbb{Z}}
\newcommand{\mcA}{\mathcal{A}}
\newcommand{\mcB}{\mathcal{B}}

\newcommand{\mcH}{\mathcal{H}}

\newcommand{\mcQ}{\mathcal{Q}}

\newcommand{\mfC}{\mathbf{C}}

\newcommand{\mfK}{\mathbf{K}}

\newcommand{\rd}{{\rm d}}

\newcommand{\bseta}{\boldsymbol\eta}

\newcommand{\ra}{\rightarrow}

\newcommand{\eq}[1]{\begin{equation}\begin{alignedat}{2} #1 \end{alignedat}\end{equation}}

\numberwithin{equation} {section}

\begin{document}

\title[Branched stable minimal hypersurfaces]
{Regularity of branched stable minimal immersed hypersurfaces}
\date{\today}

\author[Wang]{Gaoming Wang}
\address[G.W]{Beijing Institute of Mathematical Sciences and Applications\\Huairou District\\101408\\Beijing\\China}
\email{wanggaoming@bimsa.cn}

\author[Zhang]{Xuwen Zhang}
	\address[X.Z]{Mathematisches Institut\\Universit\"at Freiburg\\Ernst-Zermelo-Str.1\\79104\\Freiburg\\ Germany}
\email{xuwen.zhang@math.uni-freiburg.de}

\begin{abstract}
We establish a sharp bound on the Hausdorff dimension of the non-branch singular set
of branched stable minimal immersed hypersurfaces whose singular sets have locally
finite $\mathcal H^{n-2}$-measure: the non-branch singular set is empty when $n=2$,
discrete when $n=3$, and has Hausdorff dimension at most $n-3$ when $n\geq4$.
We also construct a non-flat stable minimal cone in $\mathbb R^4$ arising from a
branched minimal immersion whose vertex is a non-branch singularity. Taking products
with Euclidean factors yields examples whose non-branch singular sets have Hausdorff
dimension exactly $n-3$, showing that our regularity bound is sharp in every dimension
$n\geq3$.
The main ingredients in our proof are a generalized Schoen inequality and a
corresponding branched sheeting theorem near stationary classical cones and unions
of hyperplanes.

\end{abstract}

\maketitle
\tableofcontents

\section{Introduction}

Minimal hypersurfaces are the critical points of the area functional, while stable minimal hypersurfaces are those for which the second variation is nonnegative.
For a two-sided minimal immersed hypersurface $M=\iota(\S)$, where $\iota\colon \S\to U\subset\mathbb{R}^{n+1}$, with global unit normal $\nu$ and second fundamental form $A$, stability is equivalent to
\eq{\label{eq:stability-intro}
  \int_\S\lvert A\rvert^{2}\varphi^{2}\,\rd\mcH^n
  \leq
  \int_\S\lvert\nabla\varphi\rvert^{2}\,\rd\mcH^n
  \quad\text{for every }\varphi\in C_c^1(\S).
}

The classical regularity and compactness theory for stable minimal hypersurfaces is by now well understood.
Schoen-Simon-Yau \cite{Schoen1975curvature} obtained local curvature estimates for stable minimal immersed hypersurfaces in dimensions $n\leq5$, and Bellettini \cite{Bellettini25} completed the Euclidean estimate in the remaining case $n=6$.
Schoen--Simon \cite{SS81} developed the corresponding sheeting, regularity, and compactness theory for stable minimal embedded hypersurfaces and obtained the optimal codimension-$7$ conclusion under the assumption that the singular set has locally finite $\mcH^{n-2}$-measure.
Wickramasekera \cite{Wickramasekera14} removed the need to impose a size condition on the singular set in advance, replacing it with the structural exclusion of classical singularities, and proved the optimal regularity and compactness conclusions for stable codimension-one integral varifolds under this structural hypothesis.

To establish the regularity and compactness theory for stable minimal immersed hypersurfaces, the main new difficulty is the possible formation of branch points.
Unlike transverse intersections of smooth sheets, branch points represent a genuine degeneration of the immersion and need not admit a local decomposition into smooth single-valued sheets.
Simon-Wickramasekera \cite{SimonWickramasekera2007,SimonWickramasekera2016} constructed stable branched minimal immersions and developed a frequency-function analysis of their branch sets.
Wickramasekera \cite{Wickramasekera08} established two-valued $C^{1,\alpha}$ regularity for stable minimal immersed hypersurfaces under the assumption that multiplicity is at most two.
These results show both that codimension-two branching is genuine and that a compactness theory allowing it must use multi-valued graphical descriptions.
For related work concerning the local structure of stable codimension-one stationary integral varifolds near higher-multiplicity planes and classical cones, see e.g. \cite{KrummelWickramasekera2021,MinterWickramasekera2024,Minter2024QHalf,EdelenMinter2024,Minter2025FiveHalves,Minter2025Campanato,BeckerKahnMinterWickramasekera2025}.

Recently, Bellettini \cite{Bellettini25} developed an intrinsic PDE method based on weak Caccioppoli inequalities and De Giorgi iteration, to show tilt estimate for stable minimal immersed hypersurfaces, in the spirit of Schoen-Simon \cite{SS81}.
His estimate gives a Lipschitz multi-valued graphical description near a single hyperplane with multiplicity, when the singular set has vanishing $2$-capacity. In particular,
if the singular set is $\mcH^{n-2}$-measure negligible, this improves to smooth single-valued sheeting.
Together with Hong and Li \cite{HLW2024deltaStable}, the first author obtained compactness for stable minimal immersions under a stronger a priori Hausdorff-dimension bound on the initial singular set.
Minter-Xiao \cite{MinterXiao2026} recently obtained the optimal non-branched regularity and compactness theorem: if the initially prescribed non-immersed singular set is $\mathcal H^{n-2}$-negligible, then the final singular set has Hausdorff dimension at most $n-7$, and the corresponding class is compact under local mass bounds.
Their conclusion is the immersed non-branched analogue of the optimal embedded theory.

It is natural to expect that two-sided stable minimal immersed hypersurfaces with locally finite $\mathcal H^{n-2}$ singular set should be closed under varifold limits subject to local mass bounds.
Bellettini \cite[p.~6]{Bellettini25} identified the analysis near classical cones,
beyond the setting of a hyperplane with multiplicity, as the natural missing step
toward the compactness conjecture. Our branched sheeting theorem addresses this step
by providing the required local analysis near stationary classical cones and unions
of hyperplanes.

\subsection{Main result}

To state our main result, we first introduce the following notations.

\begin{definition}\label{def:intro-varifold-class}
\normalfont
Let $n\geq2, \Lambda\in(0,\infty)$. Define $\mathscr V(\Lambda)$ to be the class of integral $n$-varifolds $V$ in $B_2(0)$ such that:
\begin{enumerate}
    \item [(i)] $V=\lvert(M,\beta)\rvert$ is the induced varifold of $M$ with multiplicity $\beta$, where $M$ is a two-sided, properly immersed, stable minimal hypersurface, and the singular set of $M$, denoted by ${\rm Sing}M=(\overline M\setminus M)\cap B_2(0)$, satisfies $\mathcal H^{n-2}({\rm Sing}M)<\infty$. $\beta$ is a positive integer-valued function on $M$ that is constant on each connected component;
    \item [(ii)] The mass bound holds: $ \frac{\lVert V\rVert(B_2(0))}{2^n\omega_n}\leq\Lambda$.
\end{enumerate}
We denote by $\overline{\mathscr V}(\Lambda)$ the closure of ${\mathscr V}(\Lambda)$ in the varifold topology.
\end{definition}

The branch point singularity considered in this paper is defined as follows.

\begin{definition}
\normalfont
Let $V$ be an integral $n$-varifold in $B_2(0)$.
We call a point $X\in{\rm spt}\lVert V\rVert\cap B_2(0)$ {\em regular point}, if there exists $\rho_X>0$ such that $V\llcorner B_{\rho_X}(X)=\sum_{i=1}^{N}q_i\lvert\Sigma_i\rvert\llcorner B_{\rho_X}(X)$, where $q_i\in\mathbb Z_{>0}$ and the $\Sigma_i$ are smooth, properly embedded hypersurfaces without interior boundary in $B_{\rho_X}(X)$, not necessarily mutually disjoint.
The set of regular points is denoted as ${\rm Reg}V$.
The {\em singular set} of $V$ is then
${\rm Sing}V
  \coloneqq\left({\rm spt}\lVert V\rVert\cap B_2(0)\right)\setminus{\rm Reg}V$.

A point $X\in{\rm Sing}V$ is a \emph{branch point} if some tangent cone is a finite integer sum of distinct hyperplanes, namely:
\eq{\label{eq:precompactness-branch-cone}
  \mathbf C\in{\rm VarTan}(V,X),
  \quad
  \mathbf C=\sum_{i=1}^{N}q_i\lvert P_i\rvert,
  \quad
  q_i\in\mathbb Z_{>0}.
}
The set of branch points is denoted by ${\rm Sing}_b V$. The non-branch singular set is then denoted by
${\rm Sing}_e V={\rm Sing}V\setminus{\rm Sing}_b V.$
\end{definition}

Our main result is the following sharp regularity and precompactness theorem for
$\overline{\mathscr V}(\Lambda)$.
\begin{theorem}\label{thm:l3-precompactness}
Let $n\geq2, \Lambda<\infty$.
For $j\in\mbN$, let $V_j=\lvert(M_j,\beta_j)\rvert\in\mathscr V(\Lambda)$, and assume that $0\in\overline M_j$.
Then, after passing to a subsequence,
$V_j\ra V$ in $B_2(0)$ as varifolds,
where $V\in\overline{\mathscr V}(\Lambda)$ is a nonzero stationary integral $n$-varifold, with $\dim_{\mathcal H}\bigl({\rm Sing}_e V\cap B_1(0)\bigr)\leq n-3$.
More precisely, ${\rm Sing}_e V\cap B_1(0)=\emptyset$ when $n=2$, ${\rm Sing}_e V\cap B_1(0)$ is discrete when $n=3$, and for every $n\geq3$, $\mathcal H^{n-3+\gamma}\bigl({\rm Sing}_e V\cap B_1(0)\bigr)=0$ for every $\gamma>0$.
\end{theorem}

\begin{remark}
\normalfont
In the proof of Theorem \ref{thm:l3-precompactness}, we will show that when $n\geq3$, ${\rm Sing}_eV\subset\mathcal{S}_{n-3}$, where $\mathcal{S}_{n-3}$ is the {\em$(n-3)$-stratum} of the singular set of $V$, hence by the result of Naber-Valtorta \cite{NV20}, ${\rm Sing}_eV$ is in fact countably $(n-3)$-rectifiable.
\end{remark}

\begin{remark}
\normalfont
Theorem \ref{thm:l3-precompactness} remains valid if the condition
$\mathcal H^{n-2}({\rm Sing}M)<\infty$ in Definition
\ref{def:intro-varifold-class} is replaced by the weaker assumption that
${\rm Sing}M$ has locally vanishing $2$-capacity relative to $M$.
Indeed, the local finiteness of $\mathcal H^{n-2}({\rm Sing}M)$ is used only
through the standard cutoff construction, which gives precisely
this property. These cutoffs extend the stability inequality and the subsequent
testing arguments across ${\rm Sing}M$, and once they are available, the sheeting
and compactness proofs work.
\end{remark}

The dimension estimate in Theorem \ref{thm:l3-precompactness} is sharp, in view of the following example.

\begin{theorem}\label{thm:stable-cone-sharpness}
There is a branched minimal immersion of a closed orientable surface
$F\colon M\ra\mathbb S^3$,
with finite nonempty branch set such that the cone $\mathbf C\subset\mathbb R^4$ over $F$ is non-flat, stationary, and stable.
Moreover, $0\in{\rm Sing}_e\mathbf C$.
For every $n\geq3$, the product cone
$\mathbf C\times\mathbb R^{n-3}\subset\mathbb R^{n+1}$
is stable and satisfies $\{0\}\times\mathbb R^{n-3}
  \subset
  {\rm Sing}_e
  \bigl(\mathbf C\times\mathbb R^{n-3}\bigr)$.
\end{theorem}

A key step towards Theorem \ref{thm:l3-precompactness} is the establishment of a
sheeting theorem when $M$ is close to a stationary classical cone or a hyperplane
cone whose spine has dimension at least $n-2$.
We call it a {\em Branched sheeting theorem}, for its statement it is useful to introduce the following two classes of cones.

\begin{definition}\label{defn:hyperplane-and-classical-cones}
\normalfont
Let $n\geq2$, let $\mfC$ be an $n$-dimensional cone in $\mbR^{n+1}$.
\begin{itemize}
    \item We call $\mfC$ a {\em hyperplane cone} if $\mathbf C=\sum_{i=1}^{J}q_i\lvert P_i\rvert$, where $J\geq1$, $q_i\in\mbZ_{>0}$, and $P_i$ are distinct hyperplanes in $\mbR^{n+1}$;
    \item We call $\mfC$ a {\em classical cone} if $\mathbf C=\sum_{i=1}^{J}q_i\lvert H_i\rvert$, where $J\geq3$, $q_i\in\mathbb Z_{>0}$, and $H_i$ are distinct half-hyperplanes in $\mbR^{n+1}$ with a common boundary given by a $(n-1)$-dimensional linear subspace.
Moreover, a classical cone is called {\em paired}, if it is in fact a hyperplane cone, and called {\em unpaired} otherwise.
\end{itemize}
For either type of cone, let $P_1,\dots,P_J$ be its distinct supporting hyperplanes (i.e. for classical cone these are the full extensions of the $H_i$), then choose a unit normal $p_i$ to each $P_i$, and set
$L(\mathbf C)\coloneqq{\rm span}\{p_1,\dots,p_J\}$.
We call $\ell$ the \emph{normal rank} of $\mfC$, defined as $\ell\coloneqq\dim L(\mathbf C)$.
\end{definition}
Note that any classical cone has $\ell=2$.
For any hyperplane cone $\mfC$, its spine is given by $ \mathcal S(\mathbf C)=\bigcap_{i=1}^{J}P_i=L(\mathbf C)^\perp$, with $\dim\mathcal S(\mathbf C)=n+1-\ell$.
In the case $J=1$, $\mfC$ is an integer multiplicity hyperplane, and a corresponding sheeting theorem is proved by Bellettini \cite{Bellettini25}, while if $J>1$ and ${\rm dim}\mathcal S(\mfC)\geq n-2$, then $\ell=2$ or $3$.
In this case we can prove the following generalized Schoen differential inequality (recall \cite{Schoen77Thesis} and \cite[(2.7)]{SS81}):

\begin{theorem}\label{thm:SS-ineq}
Let $n\geq2$, let $p_1,\dots,p_m\in\mathbb S^n$ be distinct, and suppose that
\eq{
\ell
=\dim{\rm span}\{p_1,\dots,p_m\}\in\{2,3\}.
}
There are a nonnegative function $G_{\mathbf p}$ on $\mbS^n$ which is smooth away from $\{p_1,\dots,p_m\}$, and a constant $C_0=C_0(\ell,n,\mathbf p)>0$ with the following properties:

The zero set of $G_{\mathbf p}$ is exactly $\{p_1,\dots,p_m\}$, and $G_{\mathbf p}$ is comparable to the spherical distance from $\{p_1,\cdots,p_m\}$.
For every properly immersed, two-sided minimal hypersurface $M\hookrightarrow\mathbb R^{n+1}$ with unit normal $\nu$, the function $g_{\mathbf p}=G_{\mathbf p}\circ\nu$ satisfies pointwisely
\eq{\label{eq:generalized-SS-intro}
  \lvert A\rvert^2g_{\mathbf p}^2
  +g_{\mathbf p}\Delta g_{\mathbf p}
  \geq C_0\lvert A\rvert^2
  \quad\text{on }\{g_{\mathbf p}>0\}.
}
\end{theorem}

For a cone $\mathbf C$ in Definition \ref{defn:hyperplane-and-classical-cones}, let $\mathcal N(\mathbf C)=\{\pm p_1,\dots,\pm p_J\}$ be the set of all unit normals to its supporting full hyperplanes.
For a two-sided immersion $M$ with unit normal $\nu$, define
\eq{\label{defn:g_mfC}
  E_{\mathbf C,R}(M)
  \coloneqq R^{-n}\int_{M\cap B_R(0)}g_{\mathbf C}^{2}\,d\mathcal H^n,
}
where $g_{\mfC}$ is obtained by applying Theorem \ref{thm:SS-ineq} to $\mathcal{N}(\mfC)$, and the corresponding function $G_{\mathbf C}$ is even on $\mbS^n$.

With these notations, we can now state our branched sheeting theorem (the multi-valued notation used below are recorded in Section \ref{sec:multivalued-functions}):

\begin{theorem}\label{thm:epsilon-regularity-sheeting}
Let $n\geq2, \Lambda<\infty$.
Let $\mathbf C$ be either a stationary classical cone or a hyperplane cone whose spine has dimension at least $n-2$, in the sense of Definition \ref{defn:hyperplane-and-classical-cones}.
Let $P_1,\dots,P_J$ be the distinct supporting full hyperplanes of $\mathbf C$.
There are positive constants
$\varepsilon=\varepsilon(n,\Lambda,\mathbf C)$, $C=C(n,\Lambda,\mathbf C)$,
with the following property:

Let $M\hookrightarrow B_R(0)$ be a properly immersed, two-sided, stable minimal hypersurface satisfying $\mathcal H^{n-2}\bigl({\rm Sing}M\cap B_R(0)\bigr)<\infty$ and $\frac{\mathcal H^n(M\cap B_R(0))}{\omega_nR^n}\leq\Lambda$.
If
\begin{equation}\label{eq:sheeting-smallness}
  E_{\mathbf C,R}(M)
  +
  \mathcal D\left((\eta_R)_\#\lvert M\rvert,
  \mathbf C\llcorner B_1(0)\right)
  <\varepsilon,
\end{equation}
where $\bseta_R(X)=R^{-1}X$, and $\mathcal{D}$ is the varifold distance.
Then $\mathbf C$ must be a hyperplane cone (i.e. if $\mfC$ is given as a classical cone, then it must be paired).
Write $\mathbf C=\sum_{i=1}^{J}q_i\lvert P_i\rvert$, $q_i\in\mathbb Z_{>0}$.
For each $i=1,\dots,J$, there is a Lipschitz $q_i$-valued function
$u_i\colon P_i\cap B_{R/2}(0)
  \ra
  \mathcal A_{q_i}(P_i^\perp)$,
such that, as integer varifolds,
\eq{\label{eq:graph-decompose-Lip}
  \lvert M\rvert\llcorner B_{R/2}(0)
  =
  \sum_{i=1}^{J}\mathbf{v}(u_i)\llcorner B_{R/2}(0).
}
Moreover, each $\mathbf{v}(u_i)$ is stationary in $B_{R/2}(0)$.
After choosing an orientation of $P_i^\perp$, the function $u_i$ has ordered Lipschitz representation $u_i=\sum_{a=1}^{q_i}[\![u_{i,a}]\!]$, with $u_{i,1}\leq\cdots\leq u_{i,q_i}$.
Away from its branch set, $\mathbf v(u_i)$ locally admits a possibly different labeling by smooth single-valued solutions of the minimal surface equation.
Finally, we have the estimate
\begin{equation}\label{eq:sheeting-estimate}
  {\rm Lip}(u_i)
  \leq
  C E_{\mathbf C,R}(M)^{1/2}.
\end{equation}
\end{theorem}

\begin{remark}
\normalfont
Suppose that $q_i=1$ for every $i$, so that each $u_i$ is single-valued.
The hypothesis that ${\rm Sing}M$ has locally finite $\mathcal H^{n-2}$-measure implies that the projection of the singular set of the graph of $u_i$ to $P_i$ has locally finite $\mathcal H^{n-2}$-measure.
By the removability of singularity of minimal surface equation (cf. \cite{Simon77}), we see that $u_i$ is smooth and each points in $\overline{M}\cap B_{R/2}(0)$ is a regular point in the immersed sense.

\end{remark}

When some $q_i>1$, Theorem \ref{thm:epsilon-regularity-sheeting} gives a stationary Lipschitz multi-valued description, but no fine structure or measure control of the
flat branch set of $u_i$. After the graphical reduction in Theorem \ref{thm:epsilon-regularity-sheeting}, the only remaining local ingredient for the compactness conjecture discussed by Bellettini \cite{Bellettini25} is a structure theorem for the branch sets of limiting stationary Lipschitz multi-valued graphs. More precisely, if every such branch set were countably $(n-2)$-rectifiable with locally finite $\mathcal H^{n-2}$ measure, then, together with the estimate for ${\rm Sing}_e V$ in Theorem \ref{thm:l3-precompactness}, the conjecture would follow. Related structure theorems for branch sets of multi-valued minimal graphs, under additional hypotheses or multiplicity restrictions, were obtained in
\cite{MinterWickramasekera2024,KrummelWickramasekera2021,
BeckerKahnMinterWickramasekera2025,KrummelMinterWickramasekera2026}, but the required arbitrary-multiplicity statement remains open.

\subsection{Strategy of the proof}

The key step in our regularity theory is to establish the branched sheeting theorem (Theorem \ref{thm:epsilon-regularity-sheeting}), and in view of Schoen-Simon \cite{SS81}, Bellettini \cite{Bellettini25}, the main point of the proof is to find a suitable tilt function which captures the information of the non-planar cone $\mfC$, namely, $g_{\bf p}$ is smooth and positive on $\left\{X\in M: g_{\bf p}(X)>0\right\}$, $g_{\bf p}\equiv0$ if $M=\mfC$, such that the intrinsic PDE \eqref{eq:generalized-SS-intro} (aka. Schoen differential inequality) holds.

We prove that when the normal rank of the cone $\mfC$ in Definition \ref{defn:hyperplane-and-classical-cones} satisfies $\ell\leq3$, then such a tilt function exists and takes the form $g_{\mfC}=G_{\mfC}\circ\nu$, where
\eq{\label{defn:G_mfC}
G_{\mfC}(y)
  =
  \sqrt{1-\lvert\pi_{L(\mfC)}(y)\rvert^2
    +k\prod^J_{j=1}\phi\left(1-\left<y,p_j\right>\right)},
  \quad\forall y\in\mathbb S^n,
}
for some suitably chosen number $k\in(0,1]$ and smooth function $\phi:[0,2]\ra[0,\infty)$ with $\phi^{-1}(0)=0$, depending only on $n,{\bf C}$.
This is partly motivated by our recent work \cite{WZ26} on the regularity of stable capillary minimal hypersurfaces in the half-space $\{x_1>0\}$.
A reason that we believe the tilt function of the form \eqref{defn:G_mfC} works comes from the following observation: if we reflect the capillary cone $\mfC$ in \cite[Example 1.2]{WZ26} across the hyperplane $\{x_1=0\}$, then we obtain a hyperplane cone $\widehat\mfC$ consisting of two distinct hyperplanes.
The estimates \cite[Proposition 3.2, Lemma 3.6]{WZ26} then yield a Schoen differential inequality for $\widehat\mfC$.
By virtue of this intrinsic PDE, one can use Bellettini's method via De Giorgi iteration to prove a corresponding Branched sheeting theorem when the stable minimal immersed hypersurface is close to $\widehat\mfC$, which then provides a short proof of Wickramasekera's multiplicity-$2$ sheeting theorem \cite[Theorem 1.4]{Wickramasekera08}.

The proof of Theorem \ref{thm:SS-ineq} is technical, mainly because the cone $\mfC$ could still be very complicated even when $\ell\leq3$. For example, think of a classical cone in the sense of Definition \ref{defn:hyperplane-and-classical-cones} with $J=100$.
We postpone more explanation and discussion of why and how the tilt function of the form \eqref{defn:G_mfC} works to Section \ref{sec:generalized-SS}.
Once the generalized Schoen differential inequality \eqref{eq:generalized-SS-intro} is established, we can then use Bellettini's method \cite{Bellettini25} to prove the Branched sheeting theorem (Theorem \ref{thm:epsilon-regularity-sheeting}), as already discussed above.
This local regularity theorem, together with the nowadays standard Federer dimension reduction argument, then leads to our main regularity and precompactness theorem for the branched stable minimal immersed hypersurfaces (Theorem \ref{thm:l3-precompactness}).

The sharpness of the Hausdorff dimension bounds in Theorem \ref{thm:l3-precompactness} can be seen by constructing 
examples based on the Kapouleas-Wiygul gluing construction \cite{KapouleasWiygul2022} and Lawson's polar map \cite{Lawson1970}.
The main intuition here is that Lawson's polar map produces classical branched minimal surfaces, say $F:M\ra\mbS^3$, and the geometry of the {\em polar mapped surface} is controlled by the {\em initial surface}, while by Simons cone stability criterion \cite{Simons68}, the cone with link given by $F(M)$ is stable minimal in $\mbR^4$ if a corresponding quadratic form (see \eqref{eq:stable-cone-stability-quadratic-form} below) is non-negative definite.
This can be satisfied by considering the initial surface using Kapouleas-Wiygul gluing construction \cite{KapouleasWiygul2022}, which desingularizes the two orthogonally intersecting Clifford tori using Scherk necks and produces arbitrary high genus.
After removing the Scherk necks, this minimal surface decomposes into four cylinders, which converge to the flat cylinder $(0,\frac{\pi}{2})\times(\mbR/2\pi\mbZ)$ when the genus goes to $\infty$, and at the same time the Scherk necks collapse, thus imposing a Dirichlet boundary condition on the flat cylinder.
We can then prove the required spectral lower bound (Proposition \ref{thm:stable-cone-spectral-lower-bound}) and show that the cone is indeed stable minimal.
Finally, the vertex of this cone is neither a regular point nor a branched point, thanks to the fact that Clifford tori are non-flat, which gives us the desired example in Theorem \ref{thm:stable-cone-sharpness}.

\subsection{Organization of the paper}
In Section \ref{sec:preliminaries}, we collect some preliminaries from Geometric Measure Theory.
In Section \ref{sec:generalized-SS}, we construct tilt functions for cones with normal link $\ell=2,3$ and prove the generalized Schoen differential inequality (Theorem \ref{thm:SS-ineq}).
Section \ref{Sec:4} is devoted to the proof of the branched sheeting theorem (Theorem \ref{thm:epsilon-regularity-sheeting}).
Section \ref{sec:regularity-precompactness} proves the main regularity and precompactness theorem (Theorem \ref{thm:l3-precompactness}).
In Section \ref{sec:stable-cone-example}, we construct branched stable minimal cones (Theorem \ref{thm:stable-cone-sharpness}) and show that the dimension bound for the non-branch singular set is sharp.

\

\noindent{\em Acknowledgements.}
The authors would like to thank Neshan Wickramasekera for his comments and interest in this work. We also thank Paul Minter and Zhengyi Xiao for their interest and Mario Schulz for the discussions regarding the initial surface used in Section \ref{sec:stable-cone-example}.

\

\noindent{\em On the use of AI.}

{\em How we found the test function for Branched stable minimal immersions}:
Starting from a capillary prototype \cite{WZ26}, we combined hand calculations with AI-assisted programming, principally through GitHub Copilot and Cursor’s Composer 2, to generate and test numerical experiments rapidly. These experiments guided the search but did not replace the proof: the resulting test function was ultimately justified by rigorous analysis. In this \href{https://gaomw.com/notes/test-function-discovery/}{note}, we explain in detail the mathematical origin of this test function, and also describe how intelligent assistance contributed to the research process.

ChatGPT assisted the authors in identifying that minimal surfaces obtained by desingularizing two orthogonally intersecting Clifford tori be used as initial
surfaces for Lawson's polar map in the construction of the examples, and were also used for selected algebraic checks and proofreading. The authors independently verified all arguments and take full responsibility for the completeness and correctness of the proofs.
\section{Preliminaries}\label{sec:preliminaries}

We adopt the following basic notations throughout the paper.
\begin{itemize}
    \item We work with the Euclidean space $\mbR^{n+1}$, with Euclidean scalar product denoted by $\langle\cdot,\cdot\rangle$, and the corresponding Levi-Civita connection denoted by $D$.
When considering the topology of $\mbR^{n+1}$, we denote by $\overline{E}$ the topological closure of a set $E\subset\mbR^{n+1}$.
We denote by $e_i$ ($i=1,\cdots,n+1$) the $i$-th coordinates basis of $\mbR^{n+1}$;
    \item $B_r(X)$ is the open ball in $\mbR^{n+1}$, centered at $X$ with radius $r>0$. We denote by $\overline{B}_r(X)$ the closed ball, understood similarly;
    \item $\mathcal H^k$ denotes $k$-dimensional Hausdorff measure, and $\omega_k=\mathcal H^k(B_1^k(0))$ denotes the volume of the $k$-dimensional unit ball.
\item If $L\subset\mathbb R^{n+1}$ is a linear subspace, then $\pi_L$ denotes the orthogonal projection onto $L$, and $L^\perp$ denotes its orthogonal complement.
    \item If $E$ is a Euclidean vector space, we identify symmetric bilinear forms on $E$ with self-adjoint endomorphisms using the Euclidean metric, and write
    \eq{
    {\rm Sym}_0(E)
    =
    \left\{
    S\in{\rm End}(E):
    S^*=S,\ {\rm tr}S=0
    \right\}.
    }
    For a subspace $E_1\subset E$, we also write
    \eq{
    {\rm Sym}_0(E,E_1)
    =
    \left\{
    S\in{\rm Sym}_0(E):
    {\rm Im}S\subset E_1
    \right\}.
    }
    \item On an immersed hypersurface $M\subset\mbR^{n+1}$, we let $\na,{\rm div},\De$ denote the Levi-Civita connection, divergence, and Laplacian induced by the immersion into $\mbR^{n+1}$. For any vector $e\in\mbR^{n+1}$, we write $e^\top=e-\langle e,\nu\rangle\nu$ for its tangential component along $M$. Let $A$ denote the second fundamental form of $M$ in $\mbR^{n+1}$, defined by $A(\tau,\xi)=\langle D_{\tau}\xi,\nu\rangle$.
\end{itemize}

\subsection{Stable minimal hypersurfaces}
Let $M$ be a properly immersed, two-sided, smooth, stable minimal hypersurface of $B_2(0)\subset\mbR^{n+1}$, then $M$ satisfies
the stability inequality \eqref{eq:stability-intro} (cf. \cite[\textsection 9]{Simon83}).
By a standard argument (cf. \cite[pp. 6]{Bellettini25}), we also have
\eq{\label{ineq:SS81-(1.17)}
\int_M\abs{A}^2\varphi^2\rd\mcH^n
\leq\int_M\abs{\na\varphi}^2\rd\mcH^n,
}
for any $\varphi\in C^1_c(M\cap B_2(0))$.
If $\mcH^{n-2}({\rm Sing}M)<\infty$ and $\frac{\mcH^n\left(M\cap B_2(0)\right)}{2^n\om_n}\leq\Lambda$, then by a standard approximation argument (cf. \cite{SS81,Wickramasekera08}), the stability inequality \eqref{ineq:SS81-(1.17)} extends to hold for any Lipschitz function $\varphi$ with compact support in $B_2(0)$.

\subsection{Varifolds}
We use the notation and terminology in \cite{Simon83}. Recall that an $n$-rectifiable varifold $V$ in $U$ is a positive Radon measure on the trivial Grassmannian bundle $U\times G(n,n+1)$ of the form
\eq{
V(\phi(X,P))
=\int_{R_V}\phi(X,T_XR_V)\beta_{V}(X)\rd\mcH^n(X),\quad\forall\phi\in C^0_c(U\times G(n,n+1)),
}
where $R_V$ is an $n$-rectifiable set in $U$, $\beta_V$ is a non-negative $\mcH^n\llcorner R_V$-measurable function. The weight measure of $V$ is defined as $\norm{V}\coloneqq\pi_\ast V$, where $\pi:U\times G(n,n+1)\ra U$ is the canonical projection, and $\pi_\ast(\cdot)$ denotes the {\em push-forward} of measure through $\pi$.
$V$ is called {\em integral} if in addition, $\beta_V\in\mbN$ at $\norm{V}$-a.e. If $M=\iota(\S)$ and $\iota\colon \S\to U$ is a smooth, proper immersion and $\beta$ is a positive integer-valued function on $M$, we denote by $\lvert(M,\beta)\rvert$ the induced integral varifold:
\eq{\label{eq:immersed-varifold-notation}
  \lvert(M,\beta)\rvert(\varphi)
  =
  \int_\S
  \beta(x)\,
  \varphi\bigl(\iota(x),\rd\iota_x(T_x\S)\bigr)\,\rd\mu_\S(x)
}
for every $\varphi\in C_c(U\times G(n,n+1))$, where $\rd\mu_\S$ is the measure induced by the immersion.
When $\beta\equiv1$, we write simply $\lvert M\rvert$. If $S$ is a $k$-dimensional Lipschitz submanifold of $U$, and $\beta\in\mbZ_{>0}$, we write $\abs{(S,\beta)}=\beta\mcH^k\llcorner S\otimes T_XS$ for the multiplicity-$\beta$ varifold induced by $S$. In the case $\beta\equiv1$, we simply write $\abs{S}$ for the multiplicity-$1$ varifold induced by $S$.

Following \cite[Definition 42.3]{Simon83}, we denote ${\rm VarTan}(V,X)$ to be the set of \textit{varifold tangents} of $V$ at  $X\in{\rm spt}\norm{V}$. By the compactness of Radon measures, ${\rm VarTan}(V,X)$ is compact and non-empty provided that the upper density $\Theta^{\ast n}(\norm{V},X)\coloneqq\limsup_{r\searrow0}\frac{\norm{V}(B_r(X))}{\om_nr^n}$ is finite.
Moreover, there exists a non-zero element in ${\rm VarTan}(V,X)$ if and only if $\Theta^{\ast n}(\mu_V,X)>0$.

\subsection{Multi-valued functions}\label{sec:multivalued-functions}

We record here some notations in \cite{DS11}, see also \cite{MinterWickramasekera2024}.

Let $E$ be a finite-dimensional Euclidean vector space and let $Q\in\mathbb Z_{>0}$.
The space of {\em $Q$-points} in $E$ is (cf. \cite[Definition 0.1]{DS11})
\eq{\label{eq:Aq-definition}
  \mathcal A_Q(E)
  =
  \left\{
    \sum_{i=1}^Q[\![v_i]\!]:v_i\in E\text{ for every }i=1,\cdots,Q
  \right\},
}
where $[\![v]\!]$ is the Dirac mass in $v\in E$.
Clearly, the points $v_i$ do not have to be distinct, e.g. $Q[\![v]\!]$ is an element of $\mcA_Q(E)$ when all $Q$ points coincide at $v$.
For every $T_1=\sum_{i=1}^Q[\![v_i]\!]$ and
  $T_2=\sum_{i=1}^Q[\![w_a]\!]$,
define (cf. \cite[Definition 0.2]{DS11})
\eq{\label{eq:Aq-metric}
  \mathcal G(T_1,T_2)
  =
  \min_{\sigma\in\mathscr{P}_Q}
  \left(
    \sum_{i=1}^Q
    \lvert v_i-w_{\sigma(i)}\rvert^2
  \right)^{1/2},
}
where $\mathscr{P}_Q$ is the group of permutations of $\{1,\dots,Q\}$.
Continuity and Lipschitz continuity of maps with values in $\mathcal A_Q(E)$ are understood with respect to $\mathcal G$.
For a map $u\colon\Omega\to\mathcal A_Q(E)$, where $\Omega$ is a subset of a Euclidean space, we write
${\rm Lip}(u)
  =
  \sup_{\substack{x,y\in\Omega\\x\neq y}}
  \frac{\mathcal G\bigl(u(x),u(y)\bigr)}{\lvert x-y\rvert}$.

Now let $P\subset\mathbb R^{n+1}$ be an $n$-dimensional linear hyperplane, let $\Omega\subset P$ be open.
Since $P^\perp$ is $1$-dim, after choosing either orientation of $P^\perp$, every Lipschitz map $u\colon\Omega\to\mathcal A_Q(P^\perp)$ has ordered Lipschitz representations (cf. \cite[pp. 871--872]{MinterWickramasekera2024}) $u=\sum_{i=1}^Q[\![u_i]\!]$.
Using the map $F_i(x)=x+u_i(x)$, we define the associated {\em integral graph varifold} by
\eq{\label{eq:q-graph-varifold}
  \mathbf{v}(u)
  =
  \sum_{i=1}^Q(F_i)_\#\lvert\Omega\rvert,
}
which is equivalent to the definition in \cite[pp. 872]{MinterWickramasekera2024}.

\section{A generalized Schoen differential inequality}\label{sec:generalized-SS}

\subsection{Construction of the tilt function}

Fix $p_1,\ldots,p_m\in\mbS^n$, put an $m$-tuple ${\bf p}=(p_1,\ldots,p_m)$, and set
\eq{
L
={\rm span}\{p_1,\ldots,p_m\},\quad\text{with }
\ell
=\dim L\in\{2,3\}.
}
Let $k>0$ and let $\phi:[0,\infty)\to[0,\infty)$ be smooth with
$\phi^{-1}(0)=0$. We define
\eq{\label{defn:P}
P(y)
=\prod_{i=1}^m\phi(1-\langle y,p_i\rangle),
\quad\forall y\in\mbS^n.
}
Set the length of the projection of $y$ onto $L^\perp$ as
\eq{\label{defn:s_L}
s_L(y)
=\abs{\pi_{L^\perp}(y)},
}
and define the general tilt function by
\eq{\label{eq:G-p-definition}
G(y)
=\left(s_L^2(y)+kP(y)\right)^{1/2}.
}
We note that these functions depend on ${\bf p}$, and throughout this section we suppress that subscript and simply write $P=P_{\bf p}$, $G=G_{\bf p}$.
Clearly, $P>0$ away from $\{p_1,\ldots,p_m\}$ and vanishes precisely at these points.  The choices of $\phi$ and $k$ will be made below.

For a positive smooth function $f$ on $\mbS^n$, define the quadratic form
\eq{\label{defn:mcQ-1}
\mcQ_f
=f^2g_{\mbS^n}+f\na_{\mbS^n}^2f,
}
or alternatively
\eq{\label{defn:mcQ-2}
\mcQ_f
=f^2g_{\mbS^n}+\frac12\na_{\mbS^n}^2(f^2)
-\frac1{4f^2}\rd(f^2)\otimes\rd(f^2).
}

Our starting point is the following useful lemma.

\begin{lemma}\label{lem:minimal-composition}
Let $M\hookrightarrow\mbR^{n+1}$ be a two-sided minimal immersed hypersurface with unit normal $\nu$, and put $g=G\circ\nu$.
Then
\eq{\label{eq:A2g2-gDelta-g}
\abs{A}^2g^2+g\De g
={\rm tr}(A\mcQ_GA),\quad\text{whenever }g>0.
}
\end{lemma}

\begin{proof}
At any $X\in M$ with $g(X)>0$, choose an orthonormal frame $\{\tau_i\}_{i=1}^n$ diagonalizing $A$, with $\rd\nu(\tau_i)=-\lambda_i\tau_i$.  Since the Gauss map of a minimal
hypersurface is harmonic, we have
\eq{
\abs{A}^2g^2+g\De g
&=\sum_{i=1}^n\lambda_i^2
\left(G^2g_{\mbS^n}(\tau_i,\tau_i)
+G\na_{\mbS^n}^2G(\tau_i,\tau_i)\right)\\
&=\sum_{i=1}^n\lambda_i^2\mcQ_G(\tau_i,\tau_i)
={\rm tr}(A\mcQ_GA).
}
\end{proof}

Theorem \ref{thm:SS-ineq} thus reduces to a uniform lower bound
for ${\rm tr}(S\mcQ_GS)$ on any trace-free symmetric form $S$. For this to hold, we observe the following necessary condition in the case $\ell=3$, which explains the exponent in the product defining
$P$.

\begin{lemma}\label{lem:necessary-condition}
Assume $\ell=3$ and suppose $P$ is positive and $C^2$ on an open set
$U\subset L\cap\mbS^n\cong\mbS^2$.  If, for some $k,C>0$, the function $G$ defined by \eqref{eq:G-p-definition} satisfies
\eq{
{\rm tr}(S\mcQ_GS)\geq C\abs S^2
}
for all $y\in U$ and $S\in{\rm Sym}_0(T_y\mbS^n)$, then for $\widetilde P=P|_{L\cap\mbS^n}$, we have 
\eq{\label{ineq:necessary-condition-subsolution}
(\De_{\mbS^2}+2)\widetilde P^{1/2}>0\quad\text{on }U.
}
\end{lemma}

\begin{proof}
At $y\in U$ one has $s_L(y)=0$ and $G^2=kP$.  Moreover,
$\rd(s_L^2)$ and the restriction of $\na_{\mbS^n}^2(s_L^2)$ vanish on
$T_y(L\cap\mbS^n)$.  Hence, for $u\in T_y(L\cap\mbS^n)$,
\eqref{defn:mcQ-2} gives
\eq{\label{eq:QG-on-S2}
\mcQ_G(u,u)
=kP\abs u^2+\frac k2\na_{\mbS^2}^2\widetilde P(u,u)
-\frac{k}{4P}\rd P(u)^2.
}
Choose an orthonormal basis $\tau_1,\tau_2$ of
$T_y(L\cap\mbS^n)$ and consider $S\in {\rm Sym}_0(T_y\mbS^n)$ given by
\eq{
S(\tau_1)
=\tau_1,\quad
S(\tau_2)
=-\tau_2,\quad
S(\xi)
=0\text{ for }\xi\in T_y\mbS^n\setminus{\rm span}\{\tau_1,\tau_2\}.
}
It follows that $\abs{S}^2=2$.
Using \eqref{eq:QG-on-S2}, we find
\eq{
{\rm tr}(S\mcQ_GS)
&=\mcQ_G(\tau_1,\tau_1)+\mcQ_G(\tau_2,\tau_2)=k\left(2\widetilde P+\frac12\De_{\mbS^2}\widetilde P
-\frac1{4\widetilde P}\abs{\na_{\mbS^2}\widetilde P}^2\right)\\
&=k\widetilde P^{1/2}
(\De_{\mbS^2}+2)\widetilde P^{1/2},
}
\eqref{ineq:necessary-condition-subsolution} then follows.
\end{proof}

To see which product realizes this necessary condition, we choose
$o\in\mbS^2\setminus\{p_1,\ldots,p_m\}$ and let $F$ be stereographic projection from $o$ onto $\mathbb C$.  Write $z_i=F(p_i)$, then set
\eq{
v(F^{-1}(z))=
\frac{\abs{\prod_{i=1}^m(z-z_i)}^{2/m}}{1+\abs z^2},\quad\forall z\in\mathbb C,
}
with $v(o)\coloneqq\lim_{\abs{z}\ra+\infty}\frac{\abs{\prod_{i=1}^m(z-z_i)}^{\frac{2}{m}}}{1+\abs{z}^2}=1$.
It follows that
\eq{
v>0\text{ on }\mbS^2\setminus\{p_1,\cdots,p_m\},\text{ and }v(p_i)=0,\text{ for }i\in\{1,\cdots,m\}.
}
For any $y,q\in\mbS^2\setminus\{o\}$, the standard identities
\eq{
\abs{F(y)-F(q)}^2
=\frac{\abs{y-q}^2}
{(1-\langle y,o\rangle)(1-\langle q,o\rangle)},
\quad
1+\abs{F(y)}^2=\frac2{1-\langle y,o\rangle}
}
give
\eq{
v^2(y)
=\prod_{i=1}^m
\left(\frac{1-\langle y,p_i\rangle}
{1-\langle o,p_i\rangle}\right)^{2/m}.
}
In stereographic coordinates, $(\De_{\mbS^2}+2)v>0$ away from the
$p_i$. Thus, up to a normalization, the desired profile is $t^{2/m}$ away from its zero. On the other hand, the smoothness of $\phi$ at $0$ requires instead an at least linear growth near $0$. In view of this, we construct $\phi$ as follows.

\begin{lemma}\label{lem:phi}
There are $C_m,\rho_0>0$, depending only on $m$, such that for every
$T>0$ there is a smooth nondecreasing
$\phi:[0,\infty)\to[0,\infty)$ with $\phi^{-1}(0)=0$ and
\eq{\label{defn:phi}
\phi(t)=
\begin{cases}
T^{\frac2m-1}t,&0\leq t\leq T,\\
C_mt^{\frac2m},&t\geq e^{\rho_0}T.
\end{cases}
}
Moreover, for some $C=C(m)$,
\eq{\label{ineq:log-phi-log-t-derivatives}
\frac2m\leq\frac{\rd\log\phi}{\rd\log t}\leq1,
\qquad
-\frac18\left(\frac{\rd\log\phi}{\rd\log t}\right)^2
\leq\frac{\rd}{\rd\log t}
\left(\frac{\rd\log\phi}{\rd\log t}\right)\leq C.
}
\end{lemma}

\begin{proof}
Note that $m\geq \ell\geq2$.
If $m=2$ we simply let $\phi(t)=t$, and choose $\rho_0=1, \rho_1=2$.
Then $\phi(t)$ satisfies the required properties.
Assume now $m>2$, and choose a smooth nondecreasing
$\chi:\mbR\to[0,1]$ that equals $0$ on $(-\infty,0]$ and $1$ on
$[1,\infty)$, and is constant near $0$ and $1$.  For $\rho_0>0$
to be fixed, we consider the function
\eq{
\psi(s)
=1-\left(1-\frac2m\right)\chi\left(\frac{s}{\rho_0}\right),\quad s\in\mbR,
}
it follows that
\eq{\label{eq:psi-1}
\psi
=1\text{ on }(-\infty,0],\quad
\psi
=\frac2m\text{ on }[\rho_0,\infty),\quad
\frac2m\leq\psi\leq1.
}
Then we choose $\rho_0$ sufficiently large so that
\eq{\label{eq:psi-2}
\abs{\psi'(s)}\leq C(m),\qquad
\psi'(s)
=-\frac{1-\frac{2}{m}}{\rho_0}\chi'\left(\frac{s}{\rho_0}\right)
\geq-\frac18\left(\frac2m\right)^2
\geq-\frac18\psi^2(s).
}
Define the function $\phi=0$ at $0$, and on $t\in(0,\infty)$ as
\eq{
\phi(t)=T^{2/m}\widetilde\phi\left(\frac tT\right)
\quad\text{with}\quad\widetilde\phi(r)=
\exp\left(\int_0^{\log r}\psi(s)\,\rd s\right),\quad\forall r>0.
}
Note that $\widetilde\phi(r)=r$ for $0<r\leq1$ and
$\widetilde\phi(r)=C_mr^{2/m}$ for $r\geq e^{\rho_0}$, which gives
\eqref{defn:phi}.  Finally, by direct computation
\eq{
\frac{\rd\log\phi}{\rd\log t}
=\psi\left(\log\frac tT\right)\quad\text{and }\quad
\frac{\rd}{\rd\log t}
\left(\frac{\rd\log\phi}{\rd\log t}\right)
=\psi'\left(\log\frac tT\right),
}
\eqref{ineq:log-phi-log-t-derivatives} then follows from from \eqref{eq:psi-1} and \eqref{eq:psi-2}.
\end{proof}

For $t>0$, we put for simplicity
\eq{\label{defn:bm-xi-eta-kappa}
\bm\xi
=\frac{\rd\log\phi}{\rd\log t},\quad
\bm\eta
=\frac{\rd}{\rd\log t}
\left(\frac{\rd\log\phi}{\rd\log t}\right),\quad
\bm\kappa
=\bm\eta+\bm\xi^2-\bm\xi,
}
which will be used in due course.
In the linear region $0<t\leq T$, these quantities satisfy
$(\bm\xi,\bm\eta,\bm\kappa)=(1,0,0)$.

\subsection{Spherical quadratic forms and preliminary results}

Recall that $s_L(y)$ given by \eqref{defn:s_L} is the length of the projection of $y$ onto $L^\perp$.
We define for simplicity the following unit vectors: whenever the denominator is not vanishing, we set
\eq{\label{defn:y_L}
y_L
=\frac{\pi_L(y)}{\sqrt{1-s_L^2}}\in L\cap\mbS^n,\quad
y_{L^\perp}=\frac{\pi_{L^\perp}(y)}{s_L}\in L^\perp\cap\mbS^n.
}
Thus when $0<s_L(y)<1$, we can write $y=\sqrt{1-s_L^2}\,y_L+s_Ly_{L^\perp}$.
We also set the vector
\eq{\label{defn:xi_y}
\xi_y
=-s_Ly_L+\sqrt{1-s_L^2}\,y_{L^\perp},
}
which is a well-defined unit vector in $T_y\mbS^n$.

The following formulas will be useful for the later calculations in this section.

\begin{lemma}\label{lem:shared-differentials}
For each $j\in\{1,\ldots,m\}$, set
$t=1-\langle p_j,y\rangle$.
Then
\eq{\label{eq:gradient-Hessian-t}
\na_{\mbS^n}t=-p_j+(1-t)y,\qquad
\abs{\na_{\mbS^n}t}^2=2t-t^2,\qquad
\na_{\mbS^n}^2t=(1-t)g_{\mbS^n}.
}
For $u,v\in T_y\mbS^n$,
\begin{align}
\rd(s_L^2)(u)&=2\langle\pi_{L^\perp}(y),u\rangle,
\label{eq:a2-gradient}\\
\na_{\mbS^n}^2(s_L^2)(u,v)
&=2\langle\pi_{L^\perp}(u),\pi_{L^\perp}(v)\rangle
-2s_L^2\langle u,v\rangle.\label{eq:a2-hessian}
\end{align}
We use the same letter $P$ for the natural ambient extension of
\eqref{defn:P} to $L$.  Then
\begin{equation}\label{eq:P-gradient-hessian}
\begin{aligned}
\rd P(u)&=\langle DP(\pi_L(y)),\pi_L(u)\rangle,\\
\na_{\mbS^n}^2P(u,v)
&=D^2P(\pi_L(y))(\pi_L(u),\pi_L(v))
-\langle DP(\pi_L(y)),\pi_L(y)\rangle\langle u,v\rangle.
\end{aligned}
\end{equation}
\end{lemma}

\begin{proof}
The identities in \eqref{eq:gradient-Hessian-t} follow by differentiating the restriction of
$1-\langle p_j,y\rangle$ to the sphere. For any smooth ambient
function $h$, the Gauss formula yields
\eq{
\na_{\mbS^n}^2h(u,v)
=D^2h(u,v)-\langle Dh,y\rangle\langle u,v\rangle.
}
By direct computations
\eqref{eq:a2-gradient}, \eqref{eq:a2-hessian} and \eqref{eq:P-gradient-hessian} then follow.
\end{proof}

For our analysis in the near-zero (Section \eqref{sec:3-near-zeros}) and transition (Section \eqref{sec:3-transition}) regimes, one factor in the product defining $P$ is singled out, so the following lemma will be useful.

\begin{lemma}
Fix $j\in\{1,\ldots,m\}$ and $y\in\mbS^n\setminus\{p_1,\ldots,p_m\}$. Set for simplicity
\eq{\label{defn:t-f_j-mu}
t=1-\langle p_j,y\rangle,\qquad
f_j=\prod_{i\neq j}\phi(1-\langle p_i,y\rangle),\qquad
P=f_j\phi(t),\qquad
\mu=\frac{kP}{t}.
}
Then
\eq{\label{eq:single-factor-product}
\mcQ_{(f_j\phi)^{1/2}}-f_j\mcQ_{\phi^{1/2}}
=\frac\phi2\na_{\mbS^n}^2f_j
+\frac14(\rd f_j\otimes\rd\phi+\rd\phi\otimes\rd f_j)
-\frac\phi{4f_j}\rd f_j\otimes\rd f_j.
}
Moreover (recalling \eqref{defn:bm-xi-eta-kappa}),
\begin{align}
\rd(kP)&=\mu(\bm\xi\,\rd t+t\,\rd\log f_j),
\label{eq:single-factor-dkP}\\
\na_{\mbS^n}^2(kP)
&=\mu\left[\bm\xi(1-t)g_{\mbS^n}
+\frac{\bm\kappa}{t}\rd t\otimes\rd t\right]\notag\\
&\quad+\mu\bm\xi(\rd t\otimes\rd\log f_j
+\rd\log f_j\otimes\rd t)
+\mu t\frac{\na_{\mbS^n}^2f_j}{f_j}.
\label{eq:single-factor-hess-kP}
\end{align}
\end{lemma}

\begin{proof}
By \eqref{defn:mcQ-2} we compute
\begin{align*}
\mcQ_{(f_j\phi)^{1/2}}
={}&f_j\left(\phi g_{\mbS^n}+\frac12\na_{\mbS^n}^2\phi
-\frac1{4\phi}\rd\phi\otimes\rd\phi\right)
+\frac\phi2\na_{\mbS^n}^2f_j\\
&+\frac14(\rd f_j\otimes\rd\phi+\rd\phi\otimes\rd f_j)
-\frac\phi{4f_j}\rd f_j\otimes\rd f_j,
\end{align*}
which gives~\eqref{eq:single-factor-product}. Recalling \eqref{defn:bm-xi-eta-kappa}, we note
\eq{
\phi'=\frac{\bm\xi\phi}{t},\quad
\phi''=\frac{\bm\kappa\phi}{t^2},
}
then substituting in the product rules for $\rd(kf_j\phi)$ and $\na_{\mbS^n}^2(kf_j\phi)$ gives
\eqref{eq:single-factor-dkP}, \eqref{eq:single-factor-hess-kP}.
\end{proof}

The next lemma records a useful decomposition of $T_y\mbS^n$, which helps us to estimate the lower bound of \eqref{eq:A2g2-gDelta-g}.

\begin{lemma}\label{lem:mcQ_a-projection}
If $s_L(y)=0$, then
\eq{\label{eq:sL-zero-decomposition}
T_y\mbS^n=T_y(L\cap\mbS^n)\oplus L^\perp,
\qquad
\frac12\na_{\mbS^n}^2(s_L^2)=\pi_{L^\perp}.
}
If $s_L(y)>0$, then $\mcQ_{s_L}$ is the orthogonal projection onto
$L^\perp\cap\pi_{L^\perp}(y)^\perp$,
and its kernel has dimension $\ell$.  If, in addition,
$s_L(y)<1$, setting
\eq{\label{defn:E_1-E_2-E_3}
\mbE_1
=T_{y_L}(L\cap\mbS^n),\qquad
\mbE_2
=\mbR\xi_y,\qquad
\mbE_3
=L^\perp\cap y_{L^\perp}^{\perp},
}
we then have
\eq{\label{eq:sL-three-block}
T_y\mbS^n=\mbE_1\oplus\mbE_2\oplus\mbE_3,
}
and also
\begin{align}
\rd(s_L^2)
&=2s_L\sqrt{1-s_L^2}\,\xi_y^\flat,
\label{eq:d-sL-square}\\
\frac12\na_{\mbS^n}^2(s_L^2)
&=-s_L^2I_{\mbE_1\oplus\mbE_2}
+(1-s_L^2)\xi_y^\flat\otimes\xi_y^\flat
+(1-s_L^2)I_{\mbE_3}.
\label{eq:hess-sL-square}
\end{align}
In particular,
\eq{\label{eq:QsL-projection}
\mcQ_{s_L}
=\pi_{\mbE_3},\qquad
\ker\mcQ_{s_L}
=\mbE_1\oplus\mbE_2.
}
\end{lemma}

\begin{proof}
If $s_L(y)=0$, then $y\in L\cap\mbS^n$. By \eqref{eq:a2-hessian},
\eq{
\frac12\na_{\mbS^n}^2(s_L^2)(u,v)
=\langle\pi_{L^\perp}(u),\pi_{L^\perp}(v)\rangle,
}
which proves \eqref{eq:sL-zero-decomposition}.

Suppose now $s_L(y)>0$.  By \eqref{defn:mcQ-2},
\eqref{eq:a2-gradient}, and \eqref{eq:a2-hessian}, we find
\eq{
\mcQ_{s_L}(u,v)
=\langle\pi_{L^\perp}(u),\pi_{L^\perp}(v)\rangle
-\frac{\langle\pi_{L^\perp}(y),u\rangle
\langle\pi_{L^\perp}(y),v\rangle}{s_L^2},
}
which is exactly the orthogonal projection onto $L^\perp\cap\pi_{L^\perp}(y)^\perp$, with kernel of dimension
$\ell$.

If $s_L(y)<1$, then the definitions of $y_L,y_{L^\perp},\xi_y$ yield
\eqref{eq:sL-three-block}. On $\mbE_1, \mbE_2$ and $\mbE_3$, the projection
onto $L^\perp$ (i.e. $\pi_{L^\perp}$) satisfies
\eq{
\pi_{L^\perp}(v)
=&0,\quad\forall v\in\mbE_1,\\
\pi_{L^\perp}(\xi_y)
=&\sqrt{1-s_L^2}y_{L^\perp},\\
\pi_{L^\perp}(w)
=&w,\quad\forall w\in\mbE_3.
}
Substituting in
\eqref{eq:a2-gradient}--\eqref{eq:a2-hessian} then yields
\eqref{eq:d-sL-square}--\eqref{eq:hess-sL-square}. Moreover,
\eq{
\mcQ_{s_L}
={}&s_L^2g_{\mbS^n}+\frac12\na_{\mbS^n}^2(s_L^2)
-\frac1{4s_L^2}\rd(s_L^2)\otimes\rd(s_L^2)
=\pi_{\mbE_3},
}
which proves~\eqref{eq:QsL-projection}.
\end{proof}

The projection obtained above is the positive semidefinite leading term of
$\mcQ_G$, independent of the parameter $k$, as shown in the following lemma.

\begin{lemma}\label{lem:mcQ_G-lower}
For $y\notin\{p_1,\ldots,p_m\}$,
\begin{align}
\mcQ_G&=\frac12\na_{\mbS^n}^2(s_L^2)+k\mcQ_{P^{1/2}}
&&\text{if }s_L(y)=0,\label{eq:mcQ_G-a-P-1}\\
\mcQ_G&\geq\mcQ_{s_L}+k\mcQ_{P^{1/2}}
&&\text{if }s_L(y)>0.\label{eq:mcQ_G-a-P-2}
\end{align}
\end{lemma}

\begin{proof}
If $s_L(y)=0$, then by \eqref{defn:mcQ-2} and $G^2=s_L^2+kP$, we find
\eq{
\mcQ_G
={}&kPg_{\mbS^n}+\frac12\na_{\mbS^n}^2(s_L^2)
+\frac k2\na_{\mbS^n}^2P-\frac{k}{4P}\rd P\otimes\rd P=\frac12\na_{\mbS^n}^2(s_L^2)+k\mcQ_{P^{1/2}}.
}
For $s_L(y)>0$, Cauchy's inequality gives, as quadratic forms,
\eq{
\frac1{G^2}\rd(G^2)\otimes\rd(G^2)
\leq\frac1{s_L^2}\rd(s_L^2)\otimes\rd(s_L^2)
+\frac{k}{P}\rd P\otimes\rd P.
}
Substituting this in~\eqref{defn:mcQ-2} gives
\eqref{eq:mcQ_G-a-P-2}.
\end{proof}

For future purpose we compute the radial and tangential eigenvalues of
$\mcQ_{\phi(t)^{1/2}}$.

\begin{lemma}\label{lem:mcQ_P^half-q_r-q_T}
For $t=1-\langle p_j,y\rangle>0$,
\eq{\label{defn:mcQ-P-half}
\mcQ_{\phi(t)^{1/2}}
=\left(\phi+\frac{1-t}{2}\phi'\right)g_{\mbS^n}
+\left(\frac12\phi''-\frac{(\phi')^2}{4\phi}\right)
\rd t\otimes\rd t.
}
Its eigenvalue in the direction of $\na_{\mbS^n}t$ is
\eq{\label{defn:q-rad}
q_{\rm rad}
=\phi+\frac{1-t}{2}\phi'
+\frac{2t-t^2}{2}\phi''
-\frac{2t-t^2}{4\phi}(\phi')^2,
}
and its eigenvalue in every orthogonal direction is
\eq{\label{defn:q-tan}
q_{\rm tan}
=\phi+\frac{1-t}{2}\phi'.
}
\end{lemma}

\begin{proof}
By \eqref{eq:gradient-Hessian-t},
\eq{
\na_{\mbS^n}^2(\sqrt{\phi(t)})
=(\sqrt\phi)''\rd t\otimes\rd t
+(\sqrt\phi)'(1-t)g_{\mbS^n}.
}
Using \eqref{defn:mcQ-1} and
\eq{
\sqrt\phi(\sqrt\phi)'=\frac12\phi',\qquad
\sqrt\phi(\sqrt\phi)''
=\frac12\phi''-\frac{(\phi')^2}{4\phi},
}
we get \eqref{defn:mcQ-P-half}.
Evaluating \eqref{defn:mcQ-P-half} along $\na_{\mbS^n}t$ and its orthogonal complement, and using $\abs{\na_{\mbS^n}t}^2=2t-t^2$ (recalling \eqref{eq:gradient-Hessian-t}), yields \eqref{defn:q-rad}--\eqref{defn:q-tan}.
\end{proof}

To complete this subsection, we record four algebraic estimates used in due course.

\begin{lemma}\label{lem:tr-n=2}
Let $\mbE$ be a two-dimensional Euclidean space and let $B$ be a
symmetric form on $\mbE$.  Then every
$S\in{\rm Sym}_0(\mbE)$ satisfies
\eq{\label{eq:trace-SBS-2d}
{\rm tr}(SBS)
=\frac12{\rm tr}(B)\abs S^2.
}
\end{lemma}

\begin{proof}
In an orthonormal basis diagonalizing $B$, write $B=\begin{pmatrix}\lambda_1&0\\
0&\lambda_2\end{pmatrix}$ and $S=\begin{pmatrix}a_1&a_2\\
a_2&-a_1\end{pmatrix}$.
It follows that
\eq{
{\rm tr}(SBS)
=(\lambda_1+\lambda_2)(a_1^2+a_2^2)
=\frac12{\rm tr}(B)\abs S^2.
}
\end{proof}
In the 3-d case, we have the following trace free criterion:
\begin{lemma}\label{lem:three-dimensional-trace-free}
Let $\mbE$ be a three-dimensional Euclidean space and let $\bm\Psi$
be a symmetric bilinear form on $\mbE$ with eigenvalues
$\lambda_1,\lambda_2,\lambda_3$.  Set
\eq{\label{eq:three-dimensional-q}
\sigma_1
=\lambda_1+\lambda_2+\lambda_3,\qquad
\sigma_2
=\lambda_1\lambda_2+\lambda_1\lambda_3
+\lambda_2\lambda_3,\qquad
\bm q(s)
=3s^2-2\sigma_1s+\sigma_2.
}
If $\bm q(\bm b_0)>0$ and $\bm q'(\bm b_0)>0$, then ${\rm tr}\left(\widetilde S(\bm b_0I-\bm\Psi)\widetilde S\right)>0$ for every nonzero $\widetilde S\in{\rm Sym}_0(\mbE)$.
Moreover, the estimate is uniform: if $\bm b_0$ and $\bm\Psi$ vary in a compact family and $\mathbf q(\bm b_0),\mathbf q'(\bm b_0)$ are uniformly bounded away from $0$ on this compact family, then there exists $c>0$, depending only on the compact family, such that
\eq{
{\rm tr}\left(\widetilde S(\bm b_0I-\bm\Psi)\widetilde S\right)
\geq c\abs{\widetilde S}^2.
}

\end{lemma}

\begin{proof}
Choose an orthonormal basis in which
$\bm\Psi={\rm diag}(\lambda_1,\lambda_2,\lambda_3)$ and
write $\widetilde S=(s_{ij})$.  Since
$s_{11}+s_{22}+s_{33}=0$, we have
\eq{\label{eq:three-dimensional-trace-expansion}
{\rm tr}
\left(\widetilde S(\bm b_0I-\bm\Psi)\widetilde S\right)
={}&\sum_{i=1}^3(\bm b_0-\lambda_i)s_{ii}^2+\sum_{i<j}(2\bm b_0-\lambda_i-\lambda_j)s_{ij}^2.
}
For every $i<j$, direct substitution gives
\eq{\label{eq:q-pairwise-average}
\bm q\left(\frac{\lambda_i+\lambda_j}{2}\right)
=-\frac{(\lambda_i-\lambda_j)^2}{4}
\leq0.
}
Note that $\sigma_1^2-3\sigma_2\ge0$, and
since $\bm q(\bm b_0)>0$ and $\bm q'(\bm b_0)>0$, the number
$\bm b_0$ lies to the right of the largest root of $\bm q$.
\eqref{eq:q-pairwise-average} therefore gives
\eq{\label{eq:off-diagonal-positivity}
2\bm b_0-\lambda_i-\lambda_j>0\qquad \text{for every }i<j.
}

For the diagonal part, we substitute $s_{33}=-s_{11}-s_{22}$ to find
\eq{
&\sum_{i=1}^3(\bm b_0-\lambda_i)s_{ii}^2=(2\bm b_0-\lambda_1-\lambda_3)s_{11}^2
+2(\bm b_0-\lambda_3)s_{11}s_{22}
+(2\bm b_0-\lambda_2-\lambda_3)s_{22}^2.
}
The determinant of this quadratic form (with variables $s_{11}$ and $s_{22}$) is
\eq{
&(2\bm b_0-\lambda_1-\lambda_3)
(2\bm b_0-\lambda_2-\lambda_3)-(\bm b_0-\lambda_3)^2=3\bm b_0^2-2\sigma_1\bm b_0+\sigma_2
=\bm q(\bm b_0)>0,
}
and its first diagonal coefficient is positive by
\eqref{eq:off-diagonal-positivity}.  Thus both parts of
\eqref{eq:three-dimensional-trace-expansion} are positive definite, which proves that ${\rm tr}\left(\widetilde S(\bm b_0I-\bm\Psi)\widetilde S\right)>0$ for every nonzero $\widetilde S\in{\rm Sym}_0(\mbE)$.
Finally, the uniformity of the estimate on compact families follows easily by continuity. The proof is thus completed.
\end{proof}

\begin{lemma}\label{lem:algebraic-trace-estimate}
Let $n\geq2$ and $C_B,c>0$.  Let
$\mbE=\mbF_1\oplus\mbF_2$ be an orthogonal decomposition of an
$n$-dimensional Euclidean space, with $\dim\mbF_1\geq1$, and let $B$
be symmetric with $\norm B\leq C_B$.  Suppose that
\eq{\label{condi:algebraic-trace-estimate}
{\rm tr}(\widetilde S B\widetilde S)
\geq c\abs{\widetilde S}^2
}
for every $\widetilde S\in{\rm Sym}_0(\mbE,\mbF_1)$ (note that the hypothesis is vacuous when $\dim\mbF_1=1$).
Then there are $\widetilde\varepsilon,C>0$, depending only on
$n,C_B,c$, such that, for $0<\varepsilon<\widetilde\varepsilon$,
\eq{
{\rm tr}
\left(S(\pi_{\mbF_2}+\varepsilon B)S\right)
\geq C\varepsilon\abs S^2
}
for every $S\in{\rm Sym}_0(\mbE)$.
\end{lemma}
\begin{proof}
We delay the proof to Appendix~\ref{App:algebraic-lemmas}.
\end{proof}

\begin{lemma}\label{lem:schur-trace-estimate}
Let $n\geq3$ and $c,C>0$.  There are
$\widetilde\de\in(0,1)$ and $\widetilde C>0$, depending only on
$n,c,C$, with the following property.  Let
$\mbE=\mbF_1\oplus\mbF_2$, where $\dim\mbF_1=2$, and write with respect to this decomposition:
\eq{
\mcQ
=\begin{pmatrix}\mcQ_{11}&\mcQ_{12}\\
\mcQ_{21}&\mcQ_{22}\end{pmatrix},\qquad
\widehat{\mcQ}=\mcQ_{11}-\mcQ_{12}\mcQ_{22}^{-1}\mcQ_{21}.
}
Suppose that $\mcQ_{22}$ is positive definite and, for some
$\mu>0$ and $\de\in(0,\widetilde\de)$,
\eq{
\norm{\mcQ_{22}^{-1}\mcQ_{21}}\leq C\sqrt\de,
\qquad
\mcQ_{22}\geq c\frac\mu\de I_{\mbF_2},
}
and also
\eq{
\norm{\widehat\mcQ}\leq C\mu,
\qquad
{\rm tr}\widehat\mcQ\geq c\mu.
}
Then for every $S\in{\rm Sym}_0(\mbE)$, there holds
\eq{
{\rm tr}(S\mcQ S)\geq\widetilde C\mu\abs S^2.
}
\end{lemma}
\begin{proof}
We delay the proof to Appendix~\ref{App:algebraic-lemmas}.
\end{proof}

With these preliminary ingredients in hand, we now consider the near-zero,
away from zero, and transition regions separately.

\subsection{Near the zeros}\label{sec:3-near-zeros}

\begin{proposition}\label{prop:near}
Under the assumptions of Theorem \ref{thm:SS-ineq}, there exist
positive constants $T_1,\varepsilon_1,C_1$, depending only on $n,{\bf p}$, such that if $0<T\leq T_1$, $0<kT^{\frac{2}{m}-1}\leq\varepsilon_1$, and, for some
$j\in\{1,\ldots,m\}$,
\eq{
0<1-\langle p_j,y\rangle\leq T.
}
Then for any $S\in{\rm Sym}_0(T_y\mbS^n)$, we have
\eq{
{\rm tr}(S\mcQ_GS)
\geq C_1kT^{\frac{2}{m}-1}\abs{S}^2.
}
\end{proposition}

\begin{proof}
Choose $T_1=T_1({\bf p})$ sufficiently small so that the caps 
\eq{
\{y\in\mbS^n:1-\left<p_i,y\right>\leq T_1\},\quad {1\leq i\leq m},
}
are pairwise disjoint and, whenever $y$ lies in the $j$-th cap, we have
\eq{
1-\langle p_i,y\rangle>e^{\rho_0}T_1\qquad(i\neq j),
}
where $\rho_0$ is the constant from Lemma \ref{lem:phi}.
It follows that for any $y$ in the $j$-th cap,
\eq{
t=1-\langle p_j,y\rangle,\qquad
P=T^{\frac2m-1}f_jt,\quad\text{where }
f_j=\prod_{i\neq j}\phi(1-\langle p_i,y\rangle),
}
and the explicit power form \eqref{defn:phi} of all the remaining factors gives: for some $C=C({\bf p})>0$,
\eq{\label{eq:near-fj}
C^{-1}\leq f_j\leq C,\qquad
\abs{\na_{\mbS^n}f_j}+\abs{\na_{\mbS^n}^2f_j}\leq C.
}
To proceed, we put
\[
\varepsilon=kT^{\frac2m-1},\qquad
B_y=\mcQ_{(f_jt)^{1/2}},\qquad
\Pi_y=
\begin{cases}
\mcQ_{s_L},&s_L(y)>0,\\
\frac12\na_{\mbS^n}^2(s_L^2),&s_L(y)=0.
\end{cases}
\]
By Lemma~\ref{lem:mcQ_a-projection}, $\Pi_y$ is an orthogonal
projection, and \eqref{eq:mcQ_G-a-P-1}--\eqref{eq:mcQ_G-a-P-2} give
\eq{\label{eq:near-lower}
\mcQ_G\geq\Pi_y+\varepsilon B_y.
}
Applying \eqref{eq:single-factor-product} with $\phi(t)=t$ gives
\eq{
B_y-f_j\mcQ_{t^{1/2}}
=\frac t2\na_{\mbS^n}^2f_j
+\frac14\left(\rd f_j\otimes\rd t+\rd t\otimes\rd f_j\right)
-\frac{t}{4f_j}\rd f_j\otimes\rd f_j.
}
Since $\abs{\rd t}^2=2t-t^2\leq2T$, \eqref{eq:near-fj} thus gives
\eq{\label{eq:near-profile-error}
\norm{B_y-f_j\mcQ_{t^{1/2}}}\leq C({\bf p})\sqrt T.
}

Let $\mbF_1=\ker\Pi_y$ and $\mbF_2=\mbF_1^\perp$. Then
$\Pi_y=\pi_{\mbF_2}$ and
\eq{
\dim\mbF_1=
\begin{cases}
\ell,&s_L(y)>0,\\
\ell-1,&s_L(y)=0.
\end{cases}
}
We also note that $\na_{\mbS^n}t\in\mbF_1$ since in the case $s_L(y)=0$, we have $y\in L$, and hence $\na_{\mbS^n}t=-p_j+(1-t)y\in L\cap T_y\mbS^n=\mbF_1$; while in the case $s_L(y)>0$, we can use the fact that $\pi_{L^\perp}(\na_{\mbS^n}t)=(1-t)\pi_{L^\perp}(y)$ to compute
\eq{
\left<\na_{\mbS^n}t,w\right>
=(1-t)\left<\pi_{L^\perp}(y),w\right>=0,\quad\forall w\in L^\perp\cap\pi_L(y)^\perp,
}
and hence $\na_{\mbS^n}t\in{\rm ker}\Pi_y=\mbF_1$ thanks to Lemma \ref{lem:mcQ_a-projection}.

Applying Lemma~\ref{lem:mcQ_P^half-q_r-q_T} with $\phi(t)=t$, so that $\phi'=1$ and $\phi''=0$. By
\eqref{defn:q-rad} and \eqref{defn:q-tan},
\eq{
q_{\rm rad}=\frac34t,\qquad q_{\rm tan}=\frac{1+t}{2}.
}
If $\dim\mbF_1=1$, the hypothesis of
Lemma~\ref{lem:algebraic-trace-estimate} is vacuous. When $\dim\mbF_1=2$, Lemma~\ref{lem:tr-n=2} gives, for every
$\widetilde S\in{\rm Sym}_0(T_y\mbS^n,\mbF_1)$,
\eq{
{\rm tr}
\left(\widetilde S\mcQ_{t^{1/2}}\widetilde S\right)
=\frac12(q_{\rm rad}+q_{\rm tan})\abs{\widetilde S}^2
\geq\frac14\abs{\widetilde S}^2.
}
When $\dim\mbF_1=3$, choose an orthonormal eigenbasis
$e_1,e_2,e_3$, with $e_1$ radial, and write the restriction of
$\widetilde S$ as $(s_{\alpha\beta})$.  Since
$s_{11}+s_{22}+s_{33}=0$, we see
\eq{
{\rm tr}
\left(\widetilde S\mcQ_{t^{1/2}}\widetilde S\right)
={}&q_{\rm rad}s_{11}^2
+q_{\rm tan}(s_{22}^2+s_{33}^2)
+(q_{\rm rad}+q_{\rm tan})(s_{12}^2+s_{13}^2)
+2q_{\rm tan}s_{23}^2
\geq{}c\abs{\widetilde S}^2,
}
thanks to the estimates
\eq{
q_{\rm tan}
\geq\frac12,\qquad
q_{\rm rad}+q_{\rm tan}
\geq\frac12,\qquad
q_{\rm rad}+\frac12q_{\rm tan}
\geq\frac14.
}
Thus, after decreasing $T_1$, \eqref{eq:near-fj} and
\eqref{eq:near-profile-error} imply
\eq{
{\rm tr}(\widetilde S B_y\widetilde S)
\geq c\abs{\widetilde S}^2
}
in every possible dimension of $\mbF_1$. Since $\norm{B_y}\leq C(n,{\bf p})$, applying Lemma \ref{lem:algebraic-trace-estimate} and using
\eqref{eq:near-lower} gives
\eq{
{\rm tr}(S\mcQ_GS)
\geq C_1(n,{\bf p})\varepsilon\abs{S}^2
=C_1(n,{\bf p})kT^{\frac2m-1}\abs{S}^2,
}
provided $0<\varepsilon=kT^{\frac2m-1}\leq\varepsilon_1=\varepsilon_1(n,{\bf p})$ sufficiently small.
\end{proof}

\subsection{Away from the zeros}

We suppose throughout this subsection that
\eq{\label{eq:away-region}
1-\langle p_j,y\rangle\geq e^{2\rho_0}T
\qquad\text{for every }j\in\{1,\ldots,m\},
}
and hence by \eqref{defn:phi},
\eq{\label{eq:away-P}
P
=C_m^m\prod_{j=1}^m
\left(1-\langle p_j,y\rangle\right)^{\frac2m}.
}
Note that there exists a suitably small $T_2=T_2(n,{\bf p})$, such that for any $0<T\leq T_2$, \eqref{eq:away-region} is possible to hold, and we assume throughout the subsection that $T\leq T_2=T_2(n,{\bf p})$.

\subsubsection{The cases with a kernel of dimension at most two}

\begin{proposition}\label{prop:away-low}
In the cases $\ell=2\leq n$, and $n=2$, $\ell=3$, there are positive constants $k_2,C_2$, depending only on $\ell,n,T,{\bf p}$, such that $0<k\leq k_2$ and \eqref{eq:away-region} imply
\eq{
{\rm tr}(S\mcQ_GS)\geq kC_2\abs{S}^2
}
for every $S\in{\rm Sym}_0(T_y\mbS^n)$.
\end{proposition}

\begin{proof}
Set $t_j=1-\langle p_j,y\rangle$.
Using \eqref{eq:away-P}, the chain rule, and
\eqref{defn:mcQ-2}, we obtain the following logarithmic identities:
\begin{align}
\log(P^{1/2})
={}&\log(C_m^{m/2})+\frac1m\sum_{j=1}^m\log t_j,
\label{eq:away-logP}\\
P^{-1}\mcQ_{P^{1/2}}
={}&g_{\mbS^n}+\na_{\mbS^n}^2\log(P^{1/2})
+\rd\log(P^{1/2})\otimes\rd\log(P^{1/2}),              \label{eq:away-QP}\\
\na_{\mbS^n}^2\log t_j
={}&
\left(\frac1{t_j}-1\right)g_{\mbS^n}
-\rd\log t_j\otimes\rd\log t_j.                        \label{eq:away-logtj}
\end{align}
Substituting \eqref{eq:away-logP} and \eqref{eq:away-logtj} into \eqref{eq:away-QP} gives
\eq{\label{eq:away-profile}
P^{-1}\mcQ_{P^{1/2}}
=\left(\frac1m\sum_{j=1}^m\frac1{t_j}\right)g_{\mbS^n}
-\frac1m\sum_{j=1}^m\rd\log t_j\otimes\rd\log t_j
+\rd\log(P^{1/2})\otimes\rd\log(P^{1/2}).
}
Hence, for every two-dimensional subspace
$\Pi\subset T_y\mbS^n$,
\eq{
P^{-1}{\rm tr}_{\Pi}\mcQ_{P^{1/2}}
&\geq\frac1m\sum_{j=1}^m
\left(\frac2{t_j}-\abs{\na_{\mbS^n}\log t_j}^2\right)
=\frac1m\sum_{j=1}^m
\left(\frac2{t_j}-\frac{2t_j-t_j^2}{t_j^2}\right)
=1.
}
The lower bound for $t_j$ in \eqref{eq:away-region} and \eqref{eq:away-P} therefore yield
\eq{\label{eq:away-two-plane}
{\rm tr}_{\Pi}\mcQ_{P^{1/2}}
\geq c(m)T^2,\qquad
\norm{\mcQ_{P^{1/2}}}
\leq C(n,m,T),
}
where we have used the fact that $\abs{\rd\log t_j}^2=\frac{2t_j-t_j^2}{t_j^2}=\frac{2}{t_j}-1\leq\frac{2}{e^{2\rho_0}T}$.

To proceed, put
\eq{
\Pi_y
=\begin{cases}
\mcQ_{s_L},&s_L(y)>0,\\
\frac12\na_{\mbS^n}^2(s_L^2),&s_L(y)=0,
\end{cases}
\qquad
\mbF_1=\ker\Pi_y,\qquad
\mbF_2=\mbF_1^\perp.
}
By Lemma~\ref{lem:mcQ_G-lower} and the definition of $\Pi_y$,
\eq{
\Pi_y
=\pi_{\mbF_2},\qquad
\mcQ_G
\geq\Pi_y+k\mcQ_{P^{1/2}},\qquad
\dim\mbF_1
=
\begin{cases}
\ell,&s_L(y)>0,\\
\ell-1,&s_L(y)=0.
\end{cases}
}
Note that if $n=2$ and $\ell=3$, then $s_L\equiv0$. Hence in all cases, we have $\dim\mbF_1\leq2$.
By Lemma~\ref{lem:tr-n=2} and \eqref{eq:away-two-plane},
\eq{
{\rm tr}(\widetilde S\mcQ_{P^{1/2}}\widetilde S)
=\frac12{\rm tr}_{\mbF_1}(\mcQ_{P^{1/2}})
\abs{\widetilde S}^2
\geq c(m)T^2\abs{\widetilde S}^2.
}
Applying Lemma \ref{lem:algebraic-trace-estimate} with $\varepsilon=k$ then gives the required result and completes the proof.
\end{proof}

\subsubsection{The nondegenerate case when \texorpdfstring{$\ell=3$}{ell=3}}

\begin{remark}\label{rem:O-dependence}
\normalfont
Here and in all follows, the implicit constant in the notation $O(\cdot)$ may depend on $n$ and ${\bf p}$. We indicate the additional dependencies by subscripts. For example, $O_{T,\delta}(a)$ denotes a quantity whose absolute value is bounded by $C,a$, where $C$ may also depend on $T$, and $\delta$, but is independent of all other parameters under consideration.
\end{remark}

\begin{proposition}\label{prop:away-nondegenerate}
Let $\ell=3\leq n$ and $\de\in(0,1)$. There are positive constants $k_2,C_2$, depending only on $n,T,\de,{\bf p}$, such that $0<k\leq k_2$, \eqref{eq:away-region}, and $s_L(y)\geq\de$ imply
\eq{
{\rm tr}(S\mcQ_GS)
\geq kC_2\abs{S}^2
}
for every $S\in{\rm Sym}_0(T_y\mbS^n)$.
\end{proposition}

\begin{proof}
Set $t_j=1-\langle p_j,y\rangle$ and
\eq{
\mfK_\de
=\left\{y\in\mbS^n:t_j\geq e^{2\rho_0}T
\text{ for every }j,\ s_L(y)\geq\de\right\}.
}
On this compact set, $P$ is positive with uniformly bounded derivatives. We define
\eq{
\mcB_k
=\frac{\mcQ_G-\mcQ_{s_L}}{kP}.
}
Using \eqref{defn:mcQ-2}, a direct computation gives
\eq{\label{eq:away-Bk}
\mcB_k
={}&g_{\mbS^n}+\frac1{2P}\na_{\mbS^n}^2P
+\frac{\rd(s_L^2)\otimes\rd(s_L^2)}
{4s_L^2(s_L^2+kP)}\\
&-\frac{\rd(s_L^2)\otimes\rd P+\rd P\otimes\rd(s_L^2)}
{4P(s_L^2+kP)}
-\frac{k}{4P(s_L^2+kP)}\rd P\otimes\rd P.
}
Since $s_L\geq\de$, \eqref{eq:away-Bk} gives $\mcB_k=\mcB_0+O_{T,\de}(k)$ uniformly on $\mfK_\de$, where
\eq{\label{eq:away-B0}
\mcB_0={}&g_{\mbS^n}+\frac12\na_{\mbS^n}^2\log P
+\frac12\rd\log P\otimes\rd\log P\\
&-\frac1{2s_L}
\left(\rd s_L\otimes\rd\log P+\rd\log P\otimes\rd s_L\right)
+\frac1{s_L^2}\rd s_L\otimes\rd s_L.
}
In particular, $\norm{\mcB_k}\leq C(n,T,\de,{\bf p})$.

By Lemma~\ref{lem:mcQ_a-projection}, the kernel
$\mbF_1=\ker\mcQ_{s_L}$ is three-dimensional. More precisely,
\eq{
\mbF_1=T_y\mbS^3,\quad\text{where }
\mbS^3=\mbS^n\cap(L\oplus\mbR y_{L^\perp}).
}
Using the induced metric on $\mbF_1=T_y\mbS^3$ to identify vectors and covectors, we then define
\eq{
w
=\frac{\pi_{L^\perp}(y)}{s_L^2}-y,\qquad
\xi_j
=\frac{p_j-y}{t_j}+\frac{\pi_{L^\perp}(y)}{s_L^2},\qquad
\bar\xi
=\frac1m\sum_{j=1}^m\xi_j,\qquad
\bm\Phi
=\frac1m\sum_{j=1}^m\xi_j\otimes\xi_j.
}
Computing the ambient derivatives
\eq{
D\log t_j
=-\frac{p_j}{t_j},\qquad
D^2\log t_j
=-\frac{p_j\otimes p_j}{t_j^2},\qquad
Ds_L
=\frac{\pi_{L^\perp}(y)}{s_L},
}
in conjunction with the Gauss formula, give
\begin{gather}
\rd\log t_j
=w-\xi_j,\qquad
\rd s_L
=s_Lw,\qquad
\abs{\xi_j}=\frac1{s_L},                              \label{eq:away-xij}\\
\na_{\mbS^3}^2\log t_j
=\left(\frac1{t_j}-1\right)g_{\mbS^3}
-(w-\xi_j)\otimes(w-\xi_j).                           \label{eq:away-hesstj}
\end{gather}
Since $\log P=\frac{2}{m}\sum_j\log t_j$, we can substitute \eqref{eq:away-xij} and \eqref{eq:away-hesstj} into \eqref{eq:away-B0} to obtain the following decomposition of
$\mcB_0|_{\mbF_1}$:
\eq{\label{eq:away-B0-profile}
\mcB_0\big|_{\mbF_1}
={}&\frac1m\sum_{j=1}^m\frac1{t_j}g_{\mbS^3}
-\frac1m\sum_{j=1}^m(w-\xi_j)\otimes(w-\xi_j)\\
&+2(w-\bar\xi)\otimes(w-\bar\xi)
-w\otimes(w-\bar\xi)-(w-\bar\xi)\otimes w+w\otimes w\\
={}&\left(\frac{1}{m}\sum^m_{j=1}\frac{1}{t_j}\right)g_{\mbS^3}-\bm\Psi
=\bm b_0g_{\mbS^3}-\bm\Psi.
}
where
\begin{equation}\label{eq:away-b0-Psi}
\bm b_0=\frac1{s_L^2}-\langle\bar\xi,w\rangle,\qquad
\bm\Psi=\bm\Phi-2\bar\xi\otimes\bar\xi.
\end{equation}
Here, for the last equality, we have used
\eq{
\langle\xi_j,w\rangle
=\frac1{s_L^2}-\frac1{t_j},\qquad
\frac1m\sum_{j=1}^m\frac1{t_j}
=\frac1{s_L^2}-\langle\bar\xi,w\rangle.
}

To proceed, let $\lambda_1,\lambda_2,\lambda_3$ be the eigenvalues of $\bm\Psi$, and let $\sigma_1,\sigma_2$ be their first two elementary symmetric sums. From $\abs{\xi_j}=\frac{1}{s_L}$ (recalling \eqref{eq:away-xij}) and
$\bm\Phi-\bar\xi\otimes\bar\xi\geq0$, we find
\begin{align}
\sigma_1
={}&\frac1{s_L^2}-2\abs{\bar\xi}^2,             \label{eq:away-sigma1}\\
\sigma_2
={}&\sigma_2(\bm\Phi-\bar\xi\otimes\bar\xi)
-\abs{\bar\xi}^2\underbrace{{\rm tr}
(\bm\Phi-\bar\xi\otimes\bar\xi)}_{=\frac{1}{s_L^2}-\abs{\bar\xi}^2}+\left\langle
(\bm\Phi-\bar\xi\otimes\bar\xi)\bar\xi,\bar\xi
\right\rangle
\geq\abs{\bar\xi}^4-\frac{\abs{\bar\xi}^2}{s_L^2}.      \label{eq:away-sigma2}
\end{align}
Also $\abs w=\sqrt{s_L^{-2}-1}<s_L^{-1}$, and hence
$\bm b_0>(1-s_L\abs{\bar\xi})/s_L^2$.
Using \eqref{eq:three-dimensional-q}, \eqref{eq:away-sigma1}, and \eqref{eq:away-sigma2}, we thus obtain
\eq{
\bm q\left(\frac{1-s_L\abs{\bar\xi}}{s_L^2}\right)
&\geq\left(\frac{1-s_L\abs{\bar\xi}}{s_L}\right)^4,\\
\bm q'\left(\frac{1-s_L\abs{\bar\xi}}{s_L^2}\right)
&=\frac2{s_L^2}
\left(2s_L^2\abs{\bar\xi}^2-3s_L\abs{\bar\xi}+2\right)>0.
}
The vectors $\xi_j$ have the same length $s_L^{-1}$ but are not all equal, because $p\mapsto\frac{p-y}{1-\langle p,y\rangle}$ is injective on $\mbS^n\setminus\{y\}$. Thus we have $\abs{\bar\xi}<s_L^{-1}$ uniformly on $\mfK_\de$.
Since $\bm q''=6$ and $\bm b_0>\frac{1-s_L\abs{\bar\xi}}{s_L^2}$, both $\bm q(\bm b_0)$ and $\bm q'(\bm b_0)$ are uniformly positive on $\mfK_\de$. Lemma \ref{lem:three-dimensional-trace-free} therefore gives, uniformly on $\mfK_\de$, that
\eq{\label{eq:away-B0-coercive}
{\rm tr}(\widetilde S\mcB_0\widetilde S)
\geq c(n,T,\de,{\bf p})\abs{\widetilde S}^2
}
for every $\widetilde S\in{\rm Sym}_0(T_y\mbS^n,\mbF_1)$. Since $\mcB_k=\mcB_0+O_{T,\de}(k)$ uniformly on $\mfK_\de$, after decreasing $k_2$, the coercivity estimate \eqref{eq:away-B0-coercive} remains
valid for $\mcB_k$, possibly with a smaller constant $c>0$.  Finally, since
\eq{
\mcQ_G
=\mcQ_{s_L}+kP\mcB_k,
}
and $P$ is bounded above and below on $\mfK_\de$, applying
Lemma \ref{lem:algebraic-trace-estimate} with $\varepsilon=kP$ then gives the required result and completes the proof.
\end{proof}

\subsubsection{The degenerate case when
  \texorpdfstring{$\ell=3$}{ell=3}}

\begin{proposition}\label{prop:away-degenerate}
Let $\ell=3\leq n$.  There are positive constants
$k_2,\de_{\bf p},C_2$, depending only on $n,T,{\bf p}$, such that
$0<k\leq k_2$, \eqref{eq:away-region}, and
$0\leq s_L(y)\leq\sqrt{\de_{\bf p}}$ imply
\[
{\rm tr}(S\mcQ_GS)\geq kC_2\abs{S}^2
\]
for every $S\in{\rm Sym}_0(T_y\mbS^n)$.
\end{proposition}

\begin{proof}
Assume first that $s_L(y)>0$.
Since (recalling \eqref{defn:y_L})
\eq{
1-\langle p_j,y\rangle
=1-\sqrt{1-s_L^2}\langle p_j,y_L\rangle,\qquad
\abs{\pi_L(y)-y_L}
=1-\sqrt{1-s_L^2}=O(s_L^2),
}
after decreasing $\de_{\bf p}$, the point $y_L$ satisfies
\eq{
1-\langle p_j,y_L\rangle\geq e^{\rho_0}T
\qquad\text{for every }j.
}
On the resulting neighborhood of $y_L$, by \eqref{eq:away-P} and its first three derivatives, we have
\eq{\label{eq:away-P-derivatives}
\abs{D^r\log P}\leq C(n,T,{\bf p}),\qquad
\abs{D^rP}\leq C(n,T,{\bf p})P,
\qquad r=1,2,3,
}
Let $\widetilde P=P|_{L\cap\mbS^n}$.  In the pointwise comparisons
below, $P$ and its derivatives are evaluated at $y$, whereas
$\widetilde P$ and its derivatives are evaluated at $y_L$.  Since
$L\cap\mbS^n\cong\mbS^2$,
\eq{
\De_{\mbS^2}\log\widetilde P
=\frac2m\sum_{j=1}^m
\De_{\mbS^2}\log(1-\langle p_j,y_L\rangle)=-2.
}
Define the bilinear form on $T_{y_L}(L\cap\mbS^n)$ by
\eq{\label{eq:away-ByL}
\mcB_{y_L}
=g_{\mbS^2}
+\frac12\na_{\mbS^2}^2\log\widetilde P
+\frac14\rd\log\widetilde P\otimes\rd\log\widetilde P,
}
which satisfies
\eq{\label{eq:away-ByL-bounds}
{\rm tr}\mcB_{y_L}
=1+\frac14\abs{\na_{\mbS^2}\log\widetilde P}^2\geq1,\qquad
\norm{\mcB_{y_L}}
\leq C(n,T,{\bf p}).
}

With respect to the three-block decomposition \eqref{eq:sL-three-block}, put
\eq{
\mbF_1
=\mbE_1,\qquad
\mbF_2
=\mbE_2\oplus\mbE_3.
}
Thus $T_y\mbS^n=\mbF_1\oplus\mbF_2$ and $\dim\mbF_1=2$.
We now verify the four hypotheses of Lemma~\ref{lem:schur-trace-estimate}.

For $\xi_y$ given by \eqref{defn:xi_y}, and for any $w,w'\in\mbE_3$, \eqref{eq:P-gradient-hessian},
\eqref{eq:d-sL-square}, \eqref{eq:hess-sL-square}, and
\eqref{eq:away-P-derivatives} give
\eq{
\mcQ_G(w,w')
=\left(1+O_T(kP(y))\right)\langle w,w'\rangle,\qquad
\mcQ_G(w,\xi_y)=0,\\
\mcQ_G(\xi_y,\xi_y)
=\frac{(1-s_L^2)kP(y)}{G^2}+O_T(kP(y)),
}
where we have simplified in the $\xi_y$-entry as follows:
\eq{
\mcQ_G(\xi_y,\xi_y)
={}&G^2+1-2s_L^2
-\frac{s_L^2(1-s_L^2)}{G^2}+O_T(kP(y))\\
={}&\frac{(1-s_L^2)kP(y)}{G^2}+O_T(kP(y)).
}
Therefore
\eq{\label{eq:away-Q22}
(\mcQ_G)_{22}
=\begin{pmatrix}
\dfrac{(1-s_L^2)kP(y)}{G^2}+O_T(kP(y))&0\\
0&(1+O_T(kP(y)))I_{\mbE_3}
\end{pmatrix}
}
with respect to $\mbF_2=\mbE_2\oplus\mbE_3$.

Next let $v,v'\in\mbF_1$. Since $\abs{\pi_L(y)-y_L}=O(s_L^2)$, by \eqref{eq:away-P-derivatives}, Taylor's theorem, and $\pi_L(\xi_y)=-s_Ly_L$, we obtain
\begin{align}
(\rd P)_y(v)
={}&
P(y)(\rd\log\widetilde P)_{y_L}(v)
+O_T(s_L^2P(y))\abs v,
\label{eq:away-dP-Taylor}\\
(\na_{\mbS^n}^2P)_y(v,v')
={}&
P(y)\left[(\na_{\mbS^2}^2\log\widetilde P)_{y_L}(v,v')
+(\rd\log\widetilde P)_{y_L}(v)
(\rd\log\widetilde P)_{y_L}(v')\right]\notag\\
&+O_T(s_L^2P(y))\abs v\abs{v'},\label{eq:away-hessP-Taylor}\\
(\rd P)_y(\xi_y)
={}&O_T(s_LP(y)),\qquad
(\na_{\mbS^n}^2P)_y(v,\xi_y)
=O_T(s_LP(y))\abs v.                                    \label{eq:away-xi-Taylor}
\end{align}
Substituting these estimates into \eqref{defn:mcQ-2}, and taking $\rd(s_L^2)(v)=0$, $\na_{\mbS^n}^2(s_L^2)(v,\xi_y)=0$ into account, we deduce
\begin{align}
\mcQ_G(v,w)={}&0,\qquad w\in\mbE_3,\label{eq:away-mixed-E3}\\
\mcQ_G(v,\xi_y)={}&
-\frac{s_L\sqrt{1-s_L^2}\,kP(y)}{2G^2}
(\rd\log\widetilde P)_{y_L}(v)
+O_T(s_LkP(y))\abs v,\label{eq:away-mixed-xi}
\end{align}
and also
\eq{\label{eq:away-Q11}
(\mcQ_G)_{11}(v,v')={}&
kP(y)\left[\langle v,v'\rangle
+\frac12(\na_{\mbS^2}^2\log\widetilde P)_{y_L}(v,v')\right]\\
&+kP(y)\left(\frac12-\frac{kP(y)}{4G^2}\right)
(\rd\log\widetilde P)_{y_L}(v)
(\rd\log\widetilde P)_{y_L}(v')
+O_T(s_L^2kP(y))\abs v\abs{v'}.
}

Choose first $\de_{\bf p}$ sufficiently small and then
$k_2\leq c\de_{\bf p}$, both depending only on $n,T,{\bf p}$. Since $s_L^2\leq\de_{\bf p}$ and $P$ is bounded above and below, \eqref{eq:away-Q22} thus gives
\eq{\label{eq:away-Q22-lower}
(\mcQ_G)_{22}
\geq c(n,T,{\bf p})\frac{kP(y)}{\de_{\bf p}}I_{\mbF_2}.
}
Moreover, the inverse of the $\xi_y$-entry in
\eqref{eq:away-Q22} is
\eq{\label{eq:away-xi-inverse}
(\mcQ_G)_{22}^{-1}(\xi_y,\xi_y)
=\frac{G^2}{(1-s_L^2)kP(y)}+O_T\left(\frac{G^4}{kP(y)}\right).
}
\eqref{eq:away-mixed-E3}--\eqref{eq:away-xi-inverse} therefore imply
\begin{align}
\norm{(\mcQ_G)_{22}^{-1}(\mcQ_G)_{21}}
&\leq C(n,T,{\bf p})s_L\leq C(n,T,{\bf p})\sqrt{\de_{\bf p}},\label{eq:away-mixed-bound}\\
(\mcQ_G)_{12}(\mcQ_G)_{22}^{-1}(\mcQ_G)_{21}
&=\frac{s_L^2kP(y)}{4G^2}
(\rd\log\widetilde P)_{y_L}\otimes
(\rd\log\widetilde P)_{y_L}
+O_T(s_L^2kP(y)).                                       \label{eq:away-Schur-term}
\end{align}
Subtracting \eqref{eq:away-Schur-term} from \eqref{eq:away-Q11}, and using the fact that $\left(\frac12-\frac{kP(y)}{4G^2}\right)-\frac{s_L^2}{4G^2}=\frac14$, we obtain
\eq{\label{eq:away-Schur}
\widehat{\mcQ}_G
=kP(y)\mcB_{y_L}+O_T(s_L^2kP(y)).
}
After further decreasing $\de_{\bf p}$, if necessary, \eqref{eq:away-ByL-bounds} then gives
\eq{\label{eq:away-Schur-bounds}
\norm{\widehat{\mcQ}_G}\leq C(n,T,{\bf p})kP(y),\qquad
{\rm tr}_{\mbF_1}\widehat{\mcQ}_G\geq\frac12kP(y).
}
Applying Lemma \ref{lem:schur-trace-estimate} with $\mu=kP(y)$ and $\de=\de_{\bf p}$, we obtain as desired that
\eq{
{\rm tr}(S\mcQ_GS)\geq C(n,T,{\bf p})kP(y)\abs S^2
\geq kC_2(n,T,{\bf p})\abs S^2.
}

It remains to consider $s_L(y)=0$. In this case, we use
\eqref{eq:sL-zero-decomposition} and put
\eq{
\mbF_1
=T_y(L\cap\mbS^n),\qquad
\mbF_2
=L^\perp.
}
Note that in this case $y_L=y$, and we use $\mcB_y$ to denote the same bilinear form as $\mcB_{y_L}$. By \eqref{eq:mcQ_G-a-P-1} and the fact that $P$ depends only on the $L$-component, we find
\eq{
(\mcQ_G)_{12}
=(\mcQ_G)_{21}=0,\qquad
(\mcQ_G)_{22}
=I_{L^\perp}+O_T(kP(y)),\qquad
\widehat{\mcQ}_G
=(\mcQ_G)_{11}=kP(y)\mcB_y.
}
Thus \eqref{eq:away-ByL-bounds} and that $k_2\leq c\de_{\bf p}$ verify the same four Schur-complement bounds directly. The proof is thus completed after applying Lemma \ref{lem:schur-trace-estimate} again as above.
\end{proof}

Combining Proposition~\ref{prop:away-degenerate} with Proposition~\ref{prop:away-nondegenerate}, we obtain the required estimate as follows.
\begin{proposition}\label{prop:away}
Let $\ell=3\leq n$.  There is $T_2=T_2(n,{\bf p})>0$ and, for every
$T\in(0,T_2]$, positive constants $k_2,C_2$ depending only on
$\ell,n,T,{\bf p}$, such that $0<k\leq k_2$ and
\eqref{eq:away-region} imply
\eq{
{\rm tr}(S\mcQ_GS)\geq kC_2\abs S^2
}
for every $S\in{\rm Sym}_0(T_y\mbS^n)$.
\end{proposition}

\subsection{The transition region}\label{sec:3-transition}
In this final subsection, we suppose that for some $j\in\{1,\ldots,m\}$,
\eq{\label{eq:transition-region}
T
\leq t=1-\langle p_j,y\rangle
\leq e^{2\rho_0}T.
}
Though the following analysis is technical, we point out that the idea is exactly the same as in the previous two subsections.

Now we set things up. First we note that, after decreasing $T=T({\bf p})$, there is at most one such index and we can write
\eq{
P
=f_j\phi(t),\quad\text{where }
f_j
=\prod_{i\neq j}\phi(1-\langle p_i,y\rangle).
}
The construction of $\phi$ (recalling \eqref{defn:phi}) and the separation of the remaining factors give
\eq{\label{eq:transition-basic-bounds}
C^{-1}T^{\frac2m-1}\leq\frac{\phi(t)}t
\leq CT^{\frac2m-1},\qquad
C^{-1}\leq f_j\leq C,\qquad
\abs{\na f_j}+\abs{\na^2f_j}\leq C,
}
for some $C=C(n,{\bf p})>0$.
Moreover, we have
\eq{\label{esti:dt-phi-transition}
\abs{\rd t}^2=2t-t^2\leq CT,\qquad
\phi\leq CT\frac{\phi}{t},\qquad
\abs{\rd\phi}\leq C\sqrt T\frac{\phi}{t},
}
for some $C=C({\bf p})>0$.
Thus \eqref{eq:single-factor-product} implies
\eq{\label{eq:transition-product-error}
\norm{\mcQ_{(f_j\phi)^{1/2}}-f_j\mcQ_{\phi^{1/2}}}
\leq C(n,{\bf p})\sqrt T\frac{\phi}{t}
\leq C(n,{\bf p})T^{\frac2m-\frac12}.
}
We also note that for the $y$ considered in this subsection, we have $0\leq s_L(y)<1$, otherwise $s_L(y)=1$, so that $\abs{\pi_L(y)}=0$, which then implies $t=1-\left<p_j,y\right>=1$, a contradiction to \eqref{eq:transition-region}.

\begin{lemma}\label{lem:transition-zero}
There are $T_0=T_0({\bf p})>0$ and $c,C>0$, depending only on
$n,{\bf p}$, with the following property.  Suppose that
$T\in(0,T_0]$ and \eqref{eq:transition-region} holds.  Set
\eq{\label{defn:B_y-transition}
B_y
=T^{1-\frac2m}\mcQ_{(f_j\phi)^{1/2}}.
}
Then
\begin{equation}\label{eq:transition-profile}
\norm{B_y}\leq C,\qquad
{\rm tr}_{\mbE}B_y\geq c
\end{equation}
for every two-dimensional subspace
$\mbE\subset T_y\mbS^n$ containing $\na_{\mbS^n}t$.

If, in addition, $\ell\in\{2,3\}$ and $s_L(y)=0$, then there are
$k_0,C_0>0$, depending only on $n,T,{\bf p}$, such that
$0<k\leq k_0$ implies
\begin{equation}\label{eq:transition-zero}
{\rm tr}(S\mcQ_GS)
\geq C_0kT^{\frac2m-1}\abs S^2
\end{equation}
for every $S\in{\rm Sym}_0(T_y\mbS^n)$.
\end{lemma}

\begin{proof}
The radial and tangential eigenvalues in
Lemma~\ref{lem:mcQ_P^half-q_r-q_T} satisfy
\begin{align*}
q_{\rm rad}+q_{\rm tan}
={}&2\phi+(1-t)\phi'
+\frac{2t-t^2}{2}\phi''
-\frac{2t-t^2}{4\phi}(\phi')^2\\
={}&\frac{\phi}{t}
\left[\bm\eta+\frac12\bm\xi^2
+t\left(2-\frac12\bm\eta-\frac12\bm\xi
-\frac14\bm\xi^2\right)\right].
\end{align*}
By \eqref{ineq:log-phi-log-t-derivatives}, after decreasing $T_0=T_0(\bf p)$, we have
\eq{\label{eq:transition-eigenvalue}
q_{\rm rad}+q_{\rm tan}
\geq\frac{\phi}{4t}\bm\xi^2
\geq\frac1{m^2}\frac{\phi}{t},\qquad
\norm{\mcQ_{\phi^{1/2}}}\leq C(n,{\bf p})\frac{\phi}{t}.
}
Since the radial direction is $\na_{\mbS^n}t$,
\eqref{eq:transition-basic-bounds}, \eqref{eq:transition-product-error}, and \eqref{eq:transition-eigenvalue} then imply \eqref{eq:transition-profile}.

Assume now that $s_L(y)=0$. By Lemma \ref{lem:mcQ_a-projection},
\eq{
\Pi_y
=\frac12\na_{\mbS^n}^2(s_L^2)=\pi_{L^\perp},\qquad
\mbF_1
=\ker\Pi_y=T_y(L\cap\mbS^n),\qquad
\mbF_2
=L^\perp,\qquad
\dim\mbF_1=\ell-1.
}
If $\ell=2$, the space ${\rm Sym}_0(T_y\mbS^n,\mbF_1)=\{0\}$. If $\ell=3$, $\mbF_1$ is two-dimensional and contains
\eq{
\na_{\mbS^n}t=-p_j+(1-t)y.
}
Thus Lemma~\ref{lem:tr-n=2} and
\eqref{eq:transition-profile} give
\eq{
{\rm tr}(\widetilde S B_y\widetilde S)
\geq c\abs{\widetilde S}^2
}
for every
$\widetilde S\in{\rm Sym}_0(T_y\mbS^n,\mbF_1)$,
with the statement vacuous when $\ell=2$.  Finally,
\eq{
\mcQ_G
=\pi_{\mbF_2}+kT^{\frac2m-1}B_y,
}
so applying Lemma \ref{lem:algebraic-trace-estimate} yields \eqref{eq:transition-zero}.
\end{proof}

\begin{proposition}\label{prop:transition-low}
In the cases $\ell=2\leq n$, and $n=2$, $\ell=3$, there are
$T_3=T_3(\ell,n,{\bf p})>0$ and positive constants $k_3,C_3$,
depending only on $\ell,n,{\bf p}$ and $T\in(0,T_3]$, such that
$0<k\leq k_3$ and \eqref{eq:transition-region} imply
\eq{\label{eq:transition-conclusion}
{\rm tr}(S\mcQ_GS)
\geq C_3kT^{\frac2m-1}\abs S^2.
}
\end{proposition}

\begin{proof}
We first choose $T_3\leq T_0$ and note that the case $s_L(y)=0$ is exactly proved by \eqref{eq:transition-zero}. If $s_L(y)>0$, then we must have $\ell=2$. By lemma \ref{lem:mcQ_a-projection}, we set
\eq{
\mbF_1
=\ker\mcQ_{s_L}=\mbE_1\oplus\mbE_2,
\qquad
\mbF_2
=\mbE_3
=L^\perp\cap y^\perp_{L^\perp}.
}
Note that the space $\mbF_1$ is two-dimensional, and since
\eq{
\na_{\mbS^n}t
=-p_j+(1-t)\pi_L(y)+(1-t)\pi_{L^\perp}(y),
}
we have $\pi_{L^\perp}(\na_{\mbS^n}t)=(1-t)\pi_{L^\perp}(y)$, so that $\na_{\mbS^n}t\perp\mbE_3$. From this and \eqref{eq:QsL-projection} we see 
\eq{\label{eq:na-t-in-mbF_1}
\na_{\mbS^n}t\in\mbF_1.
}
Hence Lemma \ref{lem:tr-n=2} and \eqref{eq:transition-profile} give the trace estimate \eqref{condi:algebraic-trace-estimate} for $B_y$ defined by \eqref{defn:B_y-transition}.
By \eqref{eq:mcQ_G-a-P-2}, we have
\eq{
\mcQ_G
\geq\mcQ_{s_L}+kT^{\frac2m-1}B_y,
}
applying Lemma \ref{lem:algebraic-trace-estimate} then completes the proof.
\end{proof}

It then suffices to consider only the case $\ell=3\leq n$. We point out here that though the analysis below is technical, but the strategy follows essentially from the previous two subsection.
Let us first record some useful facts:

Since $p_j\in L$, we have $\pi_L(y)=\underbrace{\left<p_j,y\right>}_{1-t}p_j+\pi_{p_j^\perp}(\pi_L(y))$, so
\eq{
1=\abs{y}^2
=\abs{\pi_L(y)}^2+s_L^2
=(1-t)^2+\abs{\pi_{p_j^\perp}(\pi_L(y))}^2+s_L^2,
}
implying that $s_L^2\leq2t-t^2$.
Therefore we introduce the following parameter to measure whether $s_L$ is degenerate or not:
\eq{\label{eq:transition-theta}
\theta
=\frac{s_L^2}{2t-t^2}\in[0,1].
}
For later purposes, it is also convenient to introduce the normalized function
\eq{\label{defn:bar-mu}
\bar\mu
=\frac{G^2}{t}=\theta(2-t)+\mu,
}
where $\mu$ is defined as \eqref{defn:t-f_j-mu}. By \eqref{eq:transition-basic-bounds} we have
\eq{\label{eq:transition-mu}
c(n,{\bf p})kT^{\frac2m-1}
\leq\mu
\leq C(n,{\bf p})kT^{\frac2m-1}.
}

We continue to use Remark \ref{rem:O-dependence} in this subsection.
By virtue of the estimates \eqref{eq:transition-basic-bounds} and \eqref{esti:dt-phi-transition},
the identities
\eqref{eq:single-factor-dkP}--\eqref{eq:single-factor-hess-kP} give
\begin{align}
\rd(kP)={}&\mu(\bm\xi\,\rd t+t\,\rd\log f_j),\notag\\
\na_{\mbS^n}^2(kP)={}&
\mu\left[\bm\xi(1-t)g_{\mbS^n}
+\frac{\bm\kappa}{t}\rd t\otimes\rd t\right]
+O(\mu\sqrt t).                                      \label{eq:transition-P}
\end{align}

Whenever $\theta>0$, we have the decomposition \eqref{eq:sL-three-block}, and we choose a unit
vector $\tau\in\mbE_1$ (when $\theta=1$, any unit vector in $\mbE_1$ may be chosen) such that
\eq{\label{eq:transition-tau}
\rd t|_{\mbE_1}=\abs{\rd t|_{\mbE_1}}\tau^\flat,\qquad
\abs{\rd t|_{\mbE_1}}^2
=\frac{(1-\theta)t(2-t)}{1-s_L^2}.
}

By \eqref{eq:hess-sL-square}, \eqref{eq:transition-P}, and
\eqref{defn:mcQ-2}, a direct computation gives
\eq{\label{eq:transition-E1}
\mcQ_G|_{\mbE_1}={}&
\mu\left(t+\frac{\bm\xi(1-t)}2\right)I_{\mbE_1}
+\frac{\mu}{2t}
\left(\bm\kappa-\frac{\mu\bm\xi^2}{2\bar\mu}\right)
\left(\rd t|_{\mbE_1}\right)^{\otimes2}
+O(\mu\sqrt t),
}
and for any $Z\in\mbE_1\oplus\mbE_2$, $V,V'\in\mbE_3$,
\begin{align}
\mcQ_G(Z,V)
={}&0,\label{eq:transition-block}\\
\mcQ_G(V,V')
={}&
\left[1+kP-\frac k2
\langle DP(\pi_L(y)),\pi_L(y)\rangle\right]
\langle V,V'\rangle
=\left(1+O(\mu)\right)\langle V,V'\rangle.            \label{eq:transition-E3}
\end{align}

\subsubsection{The nondegenerate case when
  \texorpdfstring{$\ell=3$}{ell=3}}

\begin{proposition}\label{prop:transition-nondegenerate}
Let $\ell=3\leq n$ and $\bm\theta\in(0,1)$.  There are $T_3>0$,
depending only on $\ell,n,\bm\theta,{\bf p}$, and positive constants
$k_3,C_3$, depending only on $\ell,n,\bm\theta,{\bf p}$ and
$T\in(0,T_3]$, such that $0<k\leq k_3$,
\eqref{eq:transition-region}, and $\theta\in[\bm\theta,1]$ imply
\eqref{eq:transition-conclusion}.
\end{proposition}

\begin{proof}
By Lemma \ref{lem:mcQ_a-projection}, we set
\eq{
\mbF_1=\mbE_1\oplus\mbE_2,\qquad
\mbF_2=\mbE_3,\qquad
\tilde\mcB_k=\frac{\mcQ_G-\mcQ_{s_L}}{\mu},\qquad
r=\frac{(\rd t)^\sharp}{\sqrt{t(2-t)}}.
}
The vector $r$ is in fact $\frac{\na_{\mbS^n}t}{\abs{\na_{\mbS^n}t}}$ and $r\in\mbF_1$ for the same reason as in \eqref{eq:na-t-in-mbF_1}, hence by \eqref{eq:transition-tau},
\eq{
r
=\frac{\sqrt{1-\theta}}{\sqrt{1-s_L^2}}\tau
+\frac{\sqrt\theta(1-t)}{\sqrt{1-s_L^2}}\xi_y.
}
Substituting \eqref{eq:d-sL-square}--\eqref{eq:hess-sL-square} and \eqref{eq:single-factor-dkP}--\eqref{eq:single-factor-hess-kP} into \eqref{defn:mcQ-2}, using $G^2=t\bar\mu$ and $s_L^2=\theta t(2-t)$, a direct computation gives
\eq{\label{eq:transition-leading-form}
\tilde\mcB_k|_{\mbF_1}
=\frac{1}{\mu}\mcQ_G|_{\mbF_1}
={}&
\left(t+\frac{\bm\xi(1-t)}2\right)I_{\mbF_1}
+\frac{2-t}{2}
\left(\bm\kappa-\frac{\mu\bm\xi^2}{2\bar\mu}\right)
r^\flat\otimes r^\flat\\
&-\frac{\bm\xi\sqrt{1-s_L^2}\sqrt\theta(2-t)}{2\bar\mu}
\left(\xi_y^\flat\otimes r^\flat
+r^\flat\otimes\xi_y^\flat\right)
+\frac{1-s_L^2}{\bar\mu}
\xi_y^\flat\otimes\xi_y^\flat
+O_{\bm\theta}(\sqrt t),
}
where all terms involving $\rd\log f_j$ and $\na_{\mbS^n}^2f_j/f_j$ in the computation are absorbed into $O_{\bm\theta}(\sqrt t)$ after division by $\mu$, and in estimating the gradient-square term $-\frac{1}{4G^2}\rd(G^2)\otimes\rd(G^2)$ of \eqref{defn:mcQ-2}, we have used $\mu/\bar\mu\leq1$, as well as $\bar\mu\geq\theta(2-t)\geq\frac{1}{2}\bm\theta$.

Provided that $T_3$ is sufficiently small, depending only on ${\bf p}$, we have
\eq{
r
=\sqrt{1-\theta}\,\tau+\sqrt\theta\,\xi_y+O(t),\qquad
\bar\mu
=2\theta+O(t+\mu).
}
Set
\eq{
r_0
=\sqrt{1-\theta}\tau+\sqrt{\theta}\xi_y,
}
since $\theta\geq\bm\theta$, the quantities $\theta^{-1}$, $\bar\mu^{-1}$ are uniformly bounded in terms of $\bm\theta$.
Then we define $\tilde\mcB$ by replacing the explicit occurrences of $t,s_L,\mu$ in the leading
part of \eqref{eq:transition-leading-form} by zero, and replacing $r,\bar\mu$ by $r_0,2\theta$,
respectively, so that
\eq{
\tilde\mcB
=\frac{\bm\xi}{2}I_{\mbF_1}+\bm\kappa r^\flat_0\otimes r^\flat_0-\frac{\bm\xi}{2\sqrt\theta}\left(\xi_y^\flat\otimes r_0^\flat+r_0^\flat\otimes\xi_y^\flat\right)
+\frac{1}{2\theta}\xi^\flat_y\otimes\xi^\flat_y
\eqqcolon\frac{\bm\xi}{2}I_{\mbF_1}-\tilde{\bm\Psi},
}
and by \eqref{eq:transition-block}, \eqref{eq:transition-E3}, we have
\eq{\label{eq:transition-B-limit}
\norm{\tilde\mcB_k}\leq C(n,{\bf p},\bm\theta),\qquad
\norm{\tilde\mcB_k|_{\mbF_1}-\tilde\mcB}
\leq C(n,{\bf p},\bm\theta)\,(\sqrt T+\mu).
}

Computing the first two elementary symmetric functions of $\tilde{\bm\Psi}$ from \eqref{eq:transition-leading-form}, we obtain
\eq{\label{eq:transition-sigma}
\sigma_1(\tilde{\bm\Psi})
=-\bm\eta-\bm\xi^2+2\bm\xi-\frac1{2\theta},\qquad
\sigma_2(\tilde{\bm\Psi})
=\frac{1-\theta}{4\theta}
\left(2\bm\eta+\bm\xi^2-2\bm\xi\right).
}

Let $\tilde{\bm q}$ be in the form \eqref{eq:three-dimensional-q}, with $\sigma_i$ replaced by $\sigma_i(\tilde{\bm\Psi})$. Substituting \eqref{eq:transition-sigma} into \eqref{eq:three-dimensional-q} yields
\begin{align}
4\theta\tilde{\bm q}\left(\frac{\bm\xi}{2}\right)
={}&2\bm\eta+\bm\xi^2
+\theta\left(4\bm\xi\bm\eta+4\bm\xi^3-6\bm\xi^2
+2\bm\xi-2\bm\eta\right),                            \label{eq:transition-q}\\
\theta\tilde{\bm q}'\left(\frac{\bm\xi}{2}\right)
={}&1+\theta(2\bm\eta+2\bm\xi^2-\bm\xi).
\label{eq:transition-qprime}
\end{align}
The right-hand side of \eqref{eq:transition-q} is affine in $\theta$, and at $\theta=0$ it satisfies $2\bm\eta+\bm\xi^2\geq\frac34\bm\xi^2\geq\frac3{m^2}$, whereas at $\theta=1$ it satisfies
\eq{
\bm\xi(4\bm\eta+4\bm\xi^2-5\bm\xi+2)
\geq\bm\xi\left(\frac72\bm\xi^2-5\bm\xi+2\right)
>0.
}
Together with \eqref{eq:transition-qprime} and
\eqref{ineq:log-phi-log-t-derivatives}, this gives uniform bounds $\tilde{\bm q}(\bm\xi/2)\geq c(m)$ and $\tilde{\bm q}'(\bm\xi/2)\geq c(m)$ for $\theta\in[\bm\theta,1]$.
Lemma~\ref{lem:three-dimensional-trace-free} now shows that ${\rm tr}(\widetilde S\tilde\mcB\widetilde S)\geq c(m,\bm\theta)\abs{\widetilde S}^2$ for every
$\widetilde S\in{\rm Sym}_0(\mbF_1)$.
After first decreasing $T_3$, then $k_3$, depending only on the stated quantities, \eqref{eq:transition-B-limit} then transfers this estimate from $\tilde\mcB$ to $\tilde\mcB_k$. Finally, since
\eq{
\mcQ_G
=\mcQ_{s_L}+\mu\tilde\mcB_k,
}
Lemma~\ref{lem:algebraic-trace-estimate} in conjunction with \eqref{eq:transition-mu} yields \eqref{eq:transition-conclusion} as desired.
\end{proof}

\subsubsection{The degenerate case when
  \texorpdfstring{$\ell=3$}{ell=3}}

\begin{lemma}\label{lem:transition-degenerate}
Let $\ell=3\leq n$. There are $c,C,T_3>0$ and $\bm\theta\in(0,1)$, depending only on $\ell,n,{\bf p}$, and a positive constant $k_3$, depending only on $\ell,n,{\bf p}$ and $T\in(0,T_3]$, with the following property.

Suppose that $0<k\leq k_3$, \eqref{eq:transition-region} holds, and
$0<\theta\leq\bm\theta$.
Using the notations in \eqref{defn:E_1-E_2-E_3}, set
\eq{
\mbF_1
=\mbE_1,\qquad
\mbF_2
=\mbE_2\oplus\mbE_3,
}
and write $\mcQ_G=((\mcQ_G)_{ab})$ with respect to
$T_y\mbS^n=\mbF_1\oplus\mbF_2$.  Then
\eq{\label{esti:Q_G_22}
(\mcQ_G)_{22}\geq c\frac{\mu}{\bm\theta}I_{\mbF_2},
\qquad
\norm{(\mcQ_G)_{22}^{-1}(\mcQ_G)_{21}}
\leq C\sqrt{\bm\theta}.
}
Moreover, the Schur complement
\eq{
\widehat{\mcQ}_G
=(\mcQ_G)_{11}
-(\mcQ_G)_{12}(\mcQ_G)_{22}^{-1}(\mcQ_G)_{21}
}
satisfies
\eq{\label{esti:Schur-complement}
\norm{\widehat{\mcQ}_G}\leq C\mu,\qquad
{\rm tr}_{\mbF_1}\widehat{\mcQ}_G\geq c\mu.
}
\end{lemma}

\begin{proof}
After decreasing $k_3$, depending only on the stated quantities, we may assume $\mu\leq\bm\theta$.
Since $(\mcQ_G)_{11}=\mcQ_G|_{\mbF_1}$,
\eqref{eq:transition-tau}--\eqref{eq:transition-E1} then give $\norm{(\mcQ_G)_{11}}\leq C(n,{\bf p})\mu$.
Taking the trace in \eqref{eq:transition-E1} and using \eqref{eq:transition-tau}, also
\eq{
\frac{(1-\theta)(2-t)}{1-s_L^2}=2+O(\theta+t),
}
we obtain
\eq{\label{eq:transition-Q11-trace}
{\rm tr}_{\mbF_1}(\mcQ_G)_{11}
=\mu(\bm\eta+\bm\xi^2)
-\frac{\mu^2\bm\xi^2}{2\bar\mu}
+O\left(\mu(\theta+\sqrt T)\right).
}

Set for simplicity $q_\xi=\mcQ_G(\xi_y,\xi_y)$.
By \eqref{defn:mcQ-2}, \eqref{eq:d-sL-square} and
\eqref{eq:hess-sL-square},
\begin{align}
q_\xi={}&G^2+1-2s_L^2
-\frac{s_L^2(1-s_L^2)}{G^2}
+\frac12\na_{\mbS^n}^2(kP)(\xi_y,\xi_y)\notag\\
&-\frac{\rd(s_L^2)(\xi_y)\rd(kP)(\xi_y)}{2G^2}
-\frac{\rd(kP)(\xi_y)^2}{4G^2}.\label{eq:transition-qxi-full}
\end{align}
The first four terms have the following simplification using the definitions of $\mu,\bar\mu,\theta$:
\eq{
G^2+1-2s_L^2-\frac{s_L^2(1-s_L^2)}{G^2}
=\frac{\mu(1+t\mu)}{\bar\mu}.
}
By the full differential formulas \eqref{eq:single-factor-dkP}--\eqref{eq:single-factor-hess-kP}, the remaining terms in \eqref{eq:transition-qxi-full} are, respectively,
\eq{
O(\mu),\qquad
\frac{\mu}{\bar\mu}O\left(\theta+\sqrt{\theta t}\right),\qquad
\frac{\mu}{\bar\mu}O\left(\mu(\theta+t)\right).
}
Combining we have
\eq{\label{eq:transition-qxi}
q_\xi
=\frac{\mu}{\bar\mu}
\left(1+O(\theta+\mu+\sqrt T)\right),\qquad
q_\xi^{-1}
=\frac{\bar\mu}{\mu}
\left(1+O(\theta+\mu+\sqrt T)\right),
}
where the second estimate follows after decreasing $\bm\theta,T_3,k_3$, depending only the stated quantities.

To proceed, we define a one-form $\omega$ on $\mbF_1$ by
\eq{
\omega(u)
=\mcQ_G(u,\xi_y),\quad\forall u\in\mbF_1.
}
By \eqref{eq:a2-gradient}, \eqref{eq:a2-hessian},
we note $\rd(s_L^2)(u)=0$, $\na_{\mbS^n}^2(s_L^2)(u,\xi_y)=0$, and hence by \eqref{defn:mcQ-2}, we have
\eq{
\omega(u)
={}&\frac12\na_{\mbS^n}^2(kP)(u,\xi_y)-\frac{\rd(kP)(u)
\left[\rd(s_L^2)(\xi_y)+\rd(kP)(\xi_y)\right]}{4G^2}.
}
Using \eqref{eq:single-factor-dkP}--\eqref{eq:single-factor-hess-kP} and \eqref{eq:transition-tau}, we thus find
\eq{\label{eq:transition-omega}
\omega(u)=
-\frac{\mu\bm\xi s_L\sqrt{(1-\theta)t(2-t)}}{2t\bar\mu}
\langle u,\tau\rangle
+\mu O\left(\sqrt\theta+\sqrt t
+\frac{\sqrt{\theta t}}{\bar\mu}\right)\abs u.
}
Combining \eqref{eq:transition-qxi} and \eqref{eq:transition-omega}, for any $u,v\in\mbF_1$, we have
\eq{\label{eq:transition-Schur-term}
q_\xi^{-1}\omega(u)\omega(v)
=\frac{\mu\theta(2-t)\bm\xi^2}{2\bar\mu}
\langle u,\tau\rangle\langle v,\tau\rangle
+O\left(\mu(\theta+\mu+\sqrt T)\right)\abs u\abs v.
}
More precisely, the product of the leading terms is
\eq{
&\frac{\bar\mu}{\mu}
\left(\frac{\mu\bm\xi s_L\sqrt{(1-\theta)t(2-t)}}
{2t\bar\mu}\right)^2=
\frac{\mu\bm\xi^2\theta(1-\theta)(2-t)^2}{4\bar\mu}
=\frac{\mu\theta(2-t)\bm\xi^2}{2\bar\mu}
+O\left(\mu(\theta+t)\right),
}
while the cross terms and the square of the error in
\eqref{eq:transition-omega} are bounded by
\eq{
C(n,{\bf p})\mu\sqrt\theta
\left(\sqrt\theta+\sqrt t+\frac{\sqrt{\theta t}}{\bar\mu}\right)
&\leq C(n,{\bf p})\mu(\theta+\sqrt T),\\
C(n,{\bf p})\mu\bar\mu
\left(\sqrt\theta+\sqrt t+\frac{\sqrt{\theta t}}{\bar\mu}\right)^2
&\leq C(n,{\bf p})\mu(\theta+\mu+\sqrt T),
}
which proves \eqref{eq:transition-Schur-term}.

By \eqref{eq:transition-block}, the only nonzero part of $(\mcQ_G)_{21}$ is $\omega$, namely, $(\mcQ_G)_{21}u=\omega(u)\xi_y$ for any $u\in\mbF_1$. Hence
\[
\widehat{\mcQ}_G(u,v)
=(\mcQ_G)_{11}(u,v)-q_\xi^{-1}\omega(u)\omega(v).
\]
Taking the trace and using
\eqref{eq:transition-Q11-trace},
\eqref{eq:transition-Schur-term} and \eqref{defn:bar-mu},
we obtain
\eq{\label{eq:transition-Schur-trace}
{\rm tr}_{\mbF_1}\widehat{\mcQ}_G
={}&\mu(\bm\eta+\bm\xi^2)
-\frac{\mu\bm\xi^2}{2\bar\mu}
\left[\mu+\theta(2-t)\right]
+O\left(\mu(\theta+\mu+\sqrt T)\right)\\
={}&\mu\left(\bm\eta+\frac12\bm\xi^2\right)
+O\left(\mu(\theta+\mu+\sqrt T)\right).
}
Also, it is direct to check that $\norm{\widehat{\mcQ}_G}\leq C(n,{\bf p})\mu$.  Recalling \eqref{ineq:log-phi-log-t-derivatives}, we see
\eq{
\bm\eta+\frac12\bm\xi^2
\geq\frac38\bm\xi^2\geq\frac3{2m^2},
}
the two asserted Schur-complement bounds \eqref{esti:Schur-complement} then follow after decreasing the parameters $\bm\theta, T_3, k_3$, depending only on the stated quantities.

Finally, by \eqref{eq:transition-block},
\eq{
(\mcQ_G)_{22}
=q_\xi\oplus\mcQ_G|_{\mbE_3}.
}
Since $\mu\leq\bm\theta$ and
$\bar\mu=\theta(2-t)+\mu\leq3\bm\theta$, using \eqref{eq:transition-E3},
\eqref{eq:transition-qxi} we find
\eq{
(\mcQ_G)_{22}
\geq c(n,{\bf p})\frac{\mu}{\bm\theta}I_{\mbF_2}.
}
Moreover, thanks to \eqref{eq:transition-qxi} and
\eqref{eq:transition-Schur-term}, taking into account that $\theta\leq\bar\mu\leq3\bm\theta$,
\eq{
\norm{(\mcQ_G)_{22}^{-1}(\mcQ_G)_{21}}^2
=\sup_{\substack{u\in\mbF_1\\ \abs u=1}}
q_\xi^{-2}\omega(u)^2
\leq&C(n,{\bf p})\frac{\bar\mu}{\mu}\left(\frac{\mu\theta}{\bar\mu}+\mu\left(\theta+\mu+\sqrt{T}\right)\right)\\
\leq& C(n,{\bf p})\bar\mu
\leq C(n,{\bf p})\bm\theta.
}
This proves \eqref{esti:Q_G_22} and completes the proof.
\end{proof}

\begin{proposition}\label{prop:transition-degenerate}
Let $\ell=3\leq n$.  There are $T_3>0$ and
$\bm\theta\in(0,1)$, depending only on $\ell,n,{\bf p}$, and positive
constants $k_3,C_3$, depending only on
$\ell,n,{\bf p}$ and $T\in(0,T_3]$, such that
$0<k\leq k_3$, \eqref{eq:transition-region}, and
$\theta\in[0,\bm\theta]$ imply \eqref{eq:transition-conclusion}.
\end{proposition}

\begin{proof}
We choose the constants so that Lemmas~\ref{lem:transition-zero} and \ref{lem:transition-degenerate} apply.
Note that the case $s_L(y)=0$ is exactly proved by
\eqref{eq:transition-zero}. If $s_L(y)>0$, by \eqref{eq:transition-theta} we must have $\theta>0$, and we have shown in Lemma~\ref{lem:transition-degenerate} that the hypotheses of Lemma~\ref{lem:schur-trace-estimate}, with $\de=\bm\theta$, hold in this case. Applying Lemma~\ref{lem:schur-trace-estimate} and using \eqref{eq:transition-mu}, we thus deduce
\eq{
{\rm tr}(S\mcQ_GS)
\geq c\mu\abs S^2
\geq ckT^{\frac2m-1}\abs S^2.
}
This completes the proof.
\end{proof}
Combining Proposition~\ref{prop:transition-degenerate} with Proposition~\ref{prop:transition-nondegenerate}, we obtain the required estimate as follows.

\begin{proposition}\label{prop:transition}
Let $\ell=3\leq n$.  There are $T_3>0$, depending only on
$\ell,n,{\bf p}$, and positive constants $k_3,C_3$, depending only on
$\ell,n,{\bf p}$ and $T\in(0,T_3]$, such that $0<k\leq k_3$ and
\eqref{eq:transition-region} imply \eqref{eq:transition-conclusion}.
\end{proposition}

We are now ready to prove the generalized Schoen differential inequality.

\begin{proof}[Proof of Theorem~\ref{thm:SS-ineq}]
We first note that $G^2=s_L^2+kP$ constructed above satisfies all the required properties.

Now we choose $T=\frac{1}{2}\min\{T_1,T_2,T_3\}>0$, where $T_1,T_2,T_3$ are constants resulting from the precious subsections, which depend only on $\ell,n,{\bf p}$, so in turn $T=T(\ell,n,{\bf p})$.
Once $T$ is fixed, we then choose $k=\frac{1}{2}\min\{k_1,k_2,k_3\}>0$, where $k_1,k_2,k_3$ are constants resulting from previous subsections (with $T=T(\ell,n,{\bf p})$ chosen).

For any $y\notin\{p_1,\ldots,p_m\}$, either
$1-\langle p_j,y\rangle\leq T$ for some $j$, or
$T\leq1-\langle p_j,y\rangle\leq e^{2\rho_0}T$
for some $j$, or $1-\langle p_j,y\rangle\geq e^{2\rho_0}T$ for every $j$.
Thus Proposition~\ref{prop:near}, together with either
Propositions~\ref{prop:away-low} and~\ref{prop:transition-low} or
Propositions~\ref{prop:away} and~\ref{prop:transition}, gives
\eq{
{\rm tr}(S\mcQ_GS)
\geq C_0\abs S^2
}
for every $S\in{\rm Sym}_0(T_y\mbS^n)$, where $C_0=C_0(\ell,n,{\bf p})>0$. Thanks to \eqref{eq:A2g2-gDelta-g}, the proof is thus completed.

\end{proof}

\section{Branched sheeting theorem for cones with normal rank at most three}\label{Sec:4}

As in the beginning of Section \ref{sec:generalized-SS}, we fix $p_1,\dots,p_m\in\mathbb S^n$, then put ${\bf p}$, $L=L({\bf p})$, $\ell={\rm dim}L$, and assume throughout this section that $\ell=\dim L\in\{2,3\}$.
We continue to write $G,g$ for the functions $G_{\bf p},g_{\bf p}$ obtained in Theorem \ref{thm:SS-ineq}, when there is no ambiguity.

\begin{proposition}\label{prop:gradient-reduction}
Let $g$ be the tilt function in Theorem \ref{thm:SS-ineq}, namely
\eq{
    g^2(X)
    =
    1-\abs{\pi_L(\nu(X))}^2+kP(\nu(X)),
}
where $k\in(0,1)$ and $\phi$ are chosen in the proof of Theorem \ref{thm:SS-ineq}.
Then there exists a constant
$C>0$, depending only on
$\ell,n,{\bf p}$, such that at every $X\in M$ with $g(X)>0$,
\eq{
    \Abs{\na g}\le C\Abs{A}.
}
\end{proposition}

\begin{proof}
Since (recalling \eqref{defn:s_L})
\eq{
2g\abs{\na g}
=\abs{\na g^2}
=\abs{\na\left(s_L^2(\nu)+kP(\nu)\right)},
}
we just need to estimate $\na s_L^2$ and $\na P$.

First, by virtue of Lemma \ref{lem:phi}, if we put $C_\phi\coloneqq\sup_{0<t\leq2}\frac{t\abs{\phi'(t)}^2}{\phi(t)}$, then $C_\phi<\infty$ because near $t=0$, $C_\phi$ is just $T=T(\ell,n,{\bf p})$ (recalling the proof of Theorem \ref{thm:SS-ineq}), while away from $t=0$ it is clearly bounded.
Since $T=T(\ell,n,{\bf p})$, we see that $C_\phi$, as a constant determined by $\phi$, in fact depends only on $\ell,n,{\bf p}$, thanks to Lemma \ref{lem:phi}.

Then we observe that, for $t_j=1-\left<p_j,\nu\right>$, there holds
\eq{
\tau(t_j)
=-\left<D_\tau\nu,p_j\right>
=-\left<D_\tau\nu,p_j^\top\right>,\quad\forall\tau\in T_XM,
}
and hence $\abs{\na t_j}^2\leq\abs{p_j^\top}^2\abs{A}^2\leq 2t_j\abs{A}^2$ thanks to that
\eq{
\abs{p_j^\top}^2
=1-\left<p_j,\nu\right>^2
=(1-\left<p_j,\nu\right>)(1+\left<p_j,\nu\right>)
=t_j(2-t_j).
}
It follows that
\eq{
\abs{\na\phi(t_j)}
=\abs{\phi'(t_j)}\abs{\na t_j}
\leq\sqrt{2C_\phi}\sqrt{\phi(t_j)}\abs{A}.
}
Hence, by the product rule we have
\eq{
\abs{\na P}
=\Abs{\sum_{j=1}^m\left(\prod_{1\leq i\neq j\leq m}\phi(t_i)\right)\na\left(\phi(t_j)\right)}
\leq\sqrt{2C_\phi}\abs{A}\sum_{j=1}^m\sqrt{\phi(t_j)}\left(\prod_{1\leq i\neq j\leq m}\phi(t_i)\right),
}
while by the assumption that $\phi\in[0,1]$, we have for each $j$ that
\eq{
\left(\sqrt{\phi(t_j)}\left(\prod_{1\leq i\neq j\leq m}\phi(t_i)\right)\right)^2
=\phi(t_j)\left(\prod_{1\leq i\neq j\leq m}\phi^2(t_i)\right)
=P\prod_{1\leq i\neq j\leq m}\phi(t_i)
\leq P,
}
so in turn
\eq{
k\abs{\na P}
\leq mk\sqrt{2C_\phi}\sqrt{P}\abs{A}
\leq m\sqrt{2C_\phi}g\abs{A}.
}

It is thus left to estimate $\abs{\na(s_L^2(\nu))}$.
For this we recall that $s_L^2(\nu)=\abs{\pi_{L^\perp}(\nu)}^2$, and hence we have (take any orthonormal basis of $T_XM$, say $\{\tau_1,\cdots,\tau_n\}$)
\eq{
\abs{\na(s_L^2(\nu))}^2
=\sum_{i=1}^n\Abs{2\left<\pi_{L^\perp}(\nu),\pi_{L^\perp}\left(D_{\tau_i}\nu\right)\right>}^2
\leq4\abs{\pi_{L^\perp}(\nu)}^2\sum_{i=1}^n\abs{D_{\tau_i}\nu}^2
=4s_L^2(\nu)\abs{A}^2,
}
namely, $\abs{\na(s_L^2(\nu))}\leq2s_L(\nu)\abs{A}\leq2\sqrt{s_L^2(\nu)+kP(\nu)}\abs{A}=2g\abs{A}$.
Combining, we find
\eq{
2g\abs{\na g}
\leq\left(2+m\sqrt{2C_\phi}\right)g\abs{A},
}
which yields the required estimate.
\end{proof}

We can now prove the following Caccioppoli inequality, in the spirit of \cite{Bellettini25}.
\begin{proposition}\label{Prop:Caccioppoli-ineq}
Following Theorem \ref{thm:SS-ineq},  there exist positive constants $\tilde C=\tilde C(\ell,n,{\bf p})$ and
$L=L(\ell,n,{\bf p})$ with the following property:
For any stable minimal immersed hypersurface $M$ in $B_2(0)$,
for every $t\in[0,L]$ and every Lipschitz function
$\psi$ with compact support in $B_2(0)$, one has
\eq{\label{ineq:Caccioppoli}
    \int_{\{g>t\}}\left(1-\frac{t}{g}\right)\Abs{\na g}^{2}\psi^{2}
    \le
    \tilde C\int_{\{g>t\}}(g-t)^2\Abs{\na\psi}^{2}.
}
Moreover, after increasing $\tilde C$ if necessary, one also has
\eq{\label{ineq:Caccioppoli-A}
    \int_{\{g>t\}}\left(1-\frac{t}{g}\right)\Abs A^{2}\psi^{2}
    \le
    \tilde C\int_{\{g>t\}}(g-t)^2\Abs{\na\psi}^{2}.
}
\end{proposition}

\begin{proof}
We choose
$\varphi=(g-t)^+\psi$
in the stability inequality \eqref{ineq:SS81-(1.17)}. Since
$\varphi=(g-t)\psi$ on $\{g>t\}$ and $\varphi=0$ on $\{g\le t\}$, we get
\eq{
    \int_{\{g>t\}}\Abs A^2(g-t)^2\psi^2
    \le
    \int_{\{g>t\}}\Abs{\na\left((g-t)\psi\right)}^2.
}
By integration by parts, using $g-t=0$ on $\partial\{g>t\}$, we have
\eq{
    \int_{\{g>t\}}\Abs{\na\left((g-t)\psi\right)}^2
    =
    \int_{\{g>t\}}-(g-t)\De g\,\psi^2
    +
    \int_{\{g>t\}}(g-t)^2\Abs{\na\psi}^2.
}
Therefore,
\eq{\label{ineq:cacc-1}
    \int_{\{g>t\}}\Abs A^2(g-t)^2\psi^2
    \le
    \int_{\{g>t\}}-(g-t)\De g\,\psi^2
    +
    \int_{\{g>t\}}(g-t)^2\Abs{\na\psi}^2.
}
By Theorem \ref{thm:SS-ineq}, there exists $C_0=C_0(\ell,n,{\bf p})>0$ such that
\eq{
    -\De g\le\frac{g^2-C_0}{g}\Abs A^2.
}
Substituting this into \eqref{ineq:cacc-1}, we obtain
\eq{
    \int_{\{g>t\}}\Abs A^2(g-t)^2\psi^2
    \le
    \int_{\{g>t\}}\left(1-\frac{t}{g}\right)(g^2-C_0)\Abs A^2\psi^2
    \quad+
    \int_{\{g>t\}}(g-t)^2\Abs{\na\psi}^2.
}
Moving the first term on the RHS to the left gives
\eq{
    \int_{\{g>t\}}
    \left(
        (g-t)^2-\left(1-\frac{t}{g}\right)(g^2-C_0)
    \right)\Abs A^2\psi^2
    \le
    \int_{\{g>t\}}(g-t)^2\Abs{\na\psi}^2.
}
The coefficient in the square brackets can be rewritten as
\eq{
    (g-t)^2-\left(1-\frac{t}{g}\right)(g^2-C_0)
    =
    \left(1-\frac{t}{g}\right)(C_0-tg),
}
thus
\eq{\label{ineq:cacc-2}
    \int_{\{g>t\}}\left(1-\frac{t}{g}\right)(C_0-tg)\Abs A^2\psi^2
    \le
    \int_{\{g>t\}}(g-t)^2\Abs{\na\psi}^2.
}
Since $0\le P\le1$, we have
\eq{
    g^2=1-\abs{\pi_L(\nu)}^2+kP(\nu)\le1+k.
}
Set
$L=\frac{C_0}{2\sqrt{1+k}}$,
then for every $t\in[0,L]$,
\eq{
    C_0-tg\ge C_0-t\sqrt{1+k}\ge\frac{C_0}{2}\text{ on }\{g>t\}.
}
Therefore \eqref{ineq:cacc-2} yields
\eq{
    \frac{C_0}{2}
    \int_{\{g>t\}}\left(1-\frac{t}{g}\right)\Abs A^2\psi^2
    \le
    \int_{\{g>t\}}(g-t)^2\Abs{\na\psi}^2.
}
This proves \eqref{ineq:Caccioppoli-A}.
Taking Proposition
\ref{prop:gradient-reduction} into account, we deduce that \eqref{ineq:Caccioppoli} holds after increasing $\tilde C$ if necessary.
Since $C_0$ and $k$ depend only on $\ell,n$ and ${\bf p}$, the constants
$\tilde C$ and $L$ also depend only on $\ell,n$ and ${\bf p}$. The proof is complete.
\end{proof}

For a two-sided immersion $M$ with unit normal $\nu$, define
\eq{\label{eq:p-tilt-excess}
   E_{\mathbf p,R}(M)
  =R^{-n}\int_{M\cap B_R(0)}g_{\mathbf p}^{2}\,\rd\mathcal H^n,
}
where $g_{\bf p}$ is the tilt function in Theorem \ref{thm:SS-ineq}.
We have the following $\varepsilon$-regularity in terms of $E_{\mathbf p,R}(M)$.

\begin{theorem}\label{thm:epsilon-regularity}
Let $n\geq2, \ell\in\{2,3\}, \Lambda\in(0,\infty)$.
Let $M$ be a properly immersed, two-sided, stable minimal hypersurface in $B_R(0)$ with $\mcH^{n-2}\left({\rm Sing}M\right)<\infty$ and $\frac{\mcH^n(M)}{\om_nR^n}\leq\Lambda$.
Under the assumptions of Theorem \ref{thm:SS-ineq},
there exist positive constants $\varepsilon_0=\varepsilon_0(\ell,n,{\bf p}), C(\ell,n,{\bf p})$ with the following property:
if $E_{\mathbf p,R}\leq\varepsilon_0$,
then
\eq{\label{ineq:sup-g_theta}
\sup_{M\cap B_\frac{R}2(0)}g
\leq C(\ell,n,{\bf p})E_{\mathbf p,R}^\frac{1}{2}.
}
\end{theorem}

\begin{proof}
Replacing $\frac{1}{2n}$ in the proof of \cite[Theorem 4]{Bellettini25} with $\frac{\tilde C(\ell,n,{\bf p})}{2}$, and thanks to Proposition \ref{Prop:Caccioppoli-ineq}, one can follow essentially the same argument as in \cite{Bellettini25} to conclude the proof.
\end{proof}

\begin{lemma}\label{lem:small-g-imply-normal-closeness}
Under the assumptions of Theorem \ref{thm:SS-ineq}, there exists a positive constant $C=C(\ell,n,{\bf p})$ with the
following property:

If $X\in M$ satisfies
\eq{\label{assump:g<kT}
g(X)<\sqrt{k}T,
}
where $T=T(\ell,n,{\bf p})$ results from the proof of Theorem \ref{thm:SS-ineq}.
Then there exists a unique $j\in\{1,\cdots,m\}$ such that
$\left<\nu(X),p_j\right>>1-T$.
Moreover,
\eq{\label{ineq:nu-p_j-close}
\Abs{\nu(X)-p_j}\le Cg(X).
}
\end{lemma}

\begin{proof}
We argue by contradiction and suppose that for all $j\in\{1,\cdots,m\}$, we have $t_j\coloneqq1-\left<p_j,\nu(X)\right>\geq T$.
Since $\phi$ is nondecreasing with $\phi(T)=T^{\frac{2}{m}}$, it follows that
\eq{
g^2(X)
\geq kP(\nu(X))
=k\prod_{j=1}^m\phi(t_j)
\geq kT^2,
}
which contradicts \eqref{assump:g<kT}.
Hence there exists at least one $j\in\{1,\cdots,m\}$ such that $t_j<T$.
Since $T<T_1=T_1({\bf p})$, by construction of $T_1$ (recalling the proof of Proposition \ref{prop:near}) we deduce, there is exactly one $j$ such that $t_j<T$.
It follows that
\eq{
g^2(X)
\geq kP(\nu(X))
=k\phi(t_j)\prod_{1\leq i\leq m, i\neq j}\phi(t_i)
\geq kT^{\frac{2}{m}-1}t_j\left(T^{\frac{2}{m}}\right)^{m-1}
=kTt_j,
}
namely, $t_j\leq\frac{g^2(X)}{kT}$.
Note that $\abs{\nu(X)-p_j}=\sqrt{2t_j}$, and recall that $k,T$ are constants depending only on $\ell,n,{\bf p}$, we thus obtain \eqref{ineq:nu-p_j-close}.
This completes the proof.
\end{proof}

\begin{proof}[Proof of Theorem \ref{thm:epsilon-regularity-sheeting}]

For an immersion $M=\iota(\S)$, every $Y\in \S$ has a neighborhood $D_Y$ on which the immersion is an embedded disk.
If $\varepsilon=\varepsilon(n,\Lambda,\mathbf C)$ is sufficiently small, Lemma \ref{lem:small-g-imply-normal-closeness} and Theorem \ref{thm:epsilon-regularity} imply that on $D_Y$ there is a unique $j\in\{1,\dots,m\}$ such that the continuous unit normal $\nu$ satisfies
\eq{\label{ineq:<nu,p_j>>1-T}
\left<\nu,p_j\right>>1-T.
}
Since $T$ is sufficiently small, we can group the local disks $D_Y$ according to the unique pair $\{\pm p_i\}$ near which their normals lie.
Since $\pi_{P_i}$ is uniformly nonsingular on the corresponding disks, the properness and continuation argument in \cite[Proof of Theorem 5]{Bellettini25} gives a finite Lipschitz multi-valued graph over each $P_i\cap B_{R/2}(0)$, initially of some degree $d_i$.

It remains to prove that $\mathbf C$ must be a hyperplane cone and that $d_i=q_i$ for every $i$.
If either conclusion failed for arbitrarily small $\varepsilon$, then after rescaling there would be a sequence $M_k$ for which both terms in \eqref{eq:sheeting-smallness} tend to zero, and the graphical parts associated with $P_i$ (resulting from the above discussion) have uniformly bounded degrees, thanks to the uniform mass bounds $\Lambda$.
After passing to a subsequence, their varifolds converge to stationary integral varifolds supported on $P_i$.
The Constancy Theorem implies that these limits are $d_i\lvert P_i\rvert$ for integers $d_i\geq0$.
On the other hand, the varifold distance term in \eqref{eq:sheeting-smallness} implies that $\lvert M_k\rvert$ converges to $\mathbf C$.
Consequently,
$\mathbf C=\sum_{i=1}^{J}d_i\lvert P_i\rvert$, proving that $\mathbf C$ is a hyperplane cone.
Moreover, since the $P_i$ are distinct, equality of the integral varifolds gives $d_i=q_i$ for every $i$.
This proves \eqref{eq:graph-decompose-Lip} with the required multiplicities.

Finally, Theorem \ref{thm:epsilon-regularity} and \eqref{ineq:nu-p_j-close} give
\begin{equation*}
  \sup_{M\cap B_{R/2}(0)}{\rm dist}
  \bigl(\nu,\mathcal N(\mathbf C)\bigr)
  \leq C E_{\mathbf C,R}(M)^{1/2}.
\end{equation*}
The elementary relation between the normal of a graph and its slope then yields \eqref{eq:sheeting-estimate}. This completes the proof.
\end{proof}

\section{Regularity and precompactness}\label{sec:regularity-precompactness}

In this section we prove the main regularity and precompactness theorem.
To our purpose, we first introduce a special subset of the branch singularities ${\rm Sing}_bV$ for $V\in\overline{\mathscr V}(\Lambda)$ as in Definition \ref{def:intro-varifold-class}:

Let ${\rm Sing}_{b,2}V$ consist of the points
$X\in{\rm Sing}_bV$ for which there exists
$\mathbf C\in{\rm VarTan}(V,X)$ of the form as in \eqref{eq:precompactness-branch-cone}, such that $\dim\bigcap_{i=1}^{N}P_i\geq n-2$.
Since ${\rm Sing}_{b,2}V\subset{\rm Sing}_b V$, we have the inclusion ${\rm Sing}_e V\cap B_1(0)\subset
    \bigl({\rm Sing}V\cap B_1(0)\bigr)\setminus{\rm Sing}_{b,2}V$.
    
\begin{proof}[Proof of Theorem \ref{thm:l3-precompactness}]
By Allard compactness, after passing to a subsequence, $V_j\rightharpoonup V$ as varifolds, where $V$ is stationary and integral.
By the standard Ahlfors regularity of stationary varifolds and the assumption that $0\in\overline{M_j}$ for each $j\in\mbN$, we have $V\neq0$ and $0\in{\rm spt}\norm{V}$.

\noindent{\em Claim 1. The blow-up limit of the sequence cannot be a unpaired classical cone.}

Indeed, let $\widehat V_j=\lvert(\widehat M_j,\widehat\beta_j)\rvert$ be a sequence of translations and dilations of the $V_j$, and assume by contradiction that
\eq{\label{eq:hat-V_j-mfC_0}
  \widehat V_j\ra
  \mathbf C_0=\sum_{i=1}^J q_i\lvert H_i\rvert,
}
where $\{H_i\}$ are half-hyperplanes with a common boundary spine and their
multiplicities do not pair to give full hyperplanes.
Since $\widehat\beta_j$ is a positive integer constant on each connected component, $\widehat V_j$ is induced by the stable immersion with multiplicity $\widehat\beta_j$ on each corresponding component.

Since the function $G_{\mathbf C_0}$ defined as \eqref{defn:g_mfC} is even, the map $P\mapsto G_{\mathbf C_0}^2(\nu_P)$ is well-defined and continuous on $G(n,n+1)$, which vanishes on every tangent plane of $\mathbf C_0$. Therefore by the varifold convergence \eqref{eq:hat-V_j-mfC_0} and the standard disintegration of varifolds, on every fixed smaller ball, we have
\eq{
  r^{-n}\int_{\widehat M_j\cap B_r(0)}
  \widehat\beta_j G_{\mathbf C_0}^2(\nu)\,\rd\mathcal H^n
  \ra0,
}
and also the convergence of the varifold-distance term in \eqref{eq:sheeting-smallness} to zero.
Applying Theorem \ref{thm:epsilon-regularity-sheeting} to the sequence $\{\widehat M_j\}$, we conclude that $\mathbf C_0$ is paired, contradicting its definition. This proves {\em Claim 1}. We also note that the same exclusion holds for iterated blow-ups.

\noindent{\em Claim 2. For any $X\in\bigl({\rm Sing}V\cap B_1(0)\bigr)\setminus{\rm Sing}_{b,2}V$,
\eq{
\label{eq:l3-away-sb2-stratum-inclusion}
    \mathbf C\in{\rm VarTan}(V,X)
    \quad\Longrightarrow\quad
    \dim\mathcal S(\mathbf C)\le n-3.
}
}
To see this, assume by contradiction that $s=\dim\mathcal S(\mathbf C)\ge n-2$.  After rotation, we can write $\mathbf C=\mathbb R^s\times\mathbf C'$, with $\mathbf C'$ a stationary integral cone of dimension
$d=n-s\le2$ in $\mathbb R^{d+1}$.

If $d=0$, then $\mathbf C$ is an integer-multiplicity hyperplane.

If $d=1$, by the structure theorem for stationary integral one-dimensional cones, we can write
$\mathbf C'$ as a finite sum of rays, which is stationary in $\mathbb R^2$. Note that $\mathbf C$ is an unpaired classical cone, unless opposite rays occur with equal multiplicity. But by {\em Claim 1}, this is impossible. Hence $\mathbf C$ is again a finite integer sum of hyperplanes.

If $d=2$. Let $\Gamma_{\mfC'}$ be the stationary integral one-varifold on $\mathbb S^2$ induced by the link of $\mathbf C'$. By \cite{AA76}, $\Gamma_{\mfC'}$ can be written locally as an integer multiplicity geodesic network, and every tangent cone $\mfC_Y\in{\rm VarTan}(\Gamma_{\mfC'},Y)$ is a finite sum of rays, which is stationary in $\mbR^2$. If some $\mfC_Y$ has unpaired rays, then the corresponding iterated tangent cone to $\mathbf C$ is
\eq{
    \mathbb R^s\times\mathbb R Y\times\mfC_Y,
}
where $\mathbb RY=\{tY:t\in\mathbb R\}$ is the line spanned by $Y$. It is an unpaired classical cone with common spine $\mathbb R^s\times\mathbb R Y$, of dimension $s+1=n-1$, a contradiction to {\em Claim 1}. Hence all tangent rays pair in opposite directions with equal multiplicities, and it follows that each local geodesic arc should extend, with the same multiplicity, and along the
same complete geodesic. Thus $\Gamma_{\mfC'}$ must be a finite integer sum of complete geodesics, and hence $\mfC$ must be a finite integer sum of hyperplanes.

In all three cases $\mathbf C$ has the form
\eqref{eq:precompactness-branch-cone}, and its planes contain the spine of dimension at least $n-2$. Thus $X\in{\rm Sing}_{b,2}V$, contradicting the choice of $X$. This proves {\em Claim 2}.

By \eqref{eq:l3-away-sb2-stratum-inclusion} and the Federer dimension-reduction argument \cite[Appendix~A, Theorem~A.4]{Simon83}, for $n\geq3$, we conclude as desired that $\mathcal H^{n-3+\gamma}\left(\bigl({\rm Sing}V\cap B_1(0)\bigr)\setminus{\rm Sing}_{b,2}V\right)=0$ for every $\gamma>0$.

We next prove that if $n=3$, then
${\rm Sing}_eV\cap B_1(0)$ is discrete.
Fix $X\in{\rm Sing}_eV\cap B_1(0)$ and assume by contradiction that distinct points $X_i\in{\rm Sing}_eV\cap B_1(0)$ converge to $X$. Set $r_i=\lvert X_i-X\rvert$ and $\bseta_{X,r}(Z)=r^{-1}(Z-X)$. After passing to a subsequence,
\eq{
  (\bseta_{X,r_i})_\#V\to\mathbf C,
  \quad
  Y_i\coloneqq\frac{X_i-X}{r_i}\to Y\in\mathbb S^3.
}
By Ahlfors regularity of stationary varifolds, $Y\in{\rm spt}\lVert\mathbf C\rVert$.

For any $\mathbf C'\in{\rm VarTan}(\mathbf C,Y)$, since $\mathbf C$ is a cone and $Y\neq0$, we have $\mathbb RY\subset\mathcal S(\mathbf C')$, and hence $\dim\mathcal S(\mathbf C')\geq1=n-2$.
Thus the proof of {\em Claim 2} is applicable and shows that $\mathbf C'$ is a hyperplane cone.
After translation and rescaling of the original hypersurfaces $V_j=\abs{(M_j,\beta_j)}$, we obtain a new sequence $\{\abs{\widetilde M_j,\widetilde\beta_j}\in\overline{\mathscr V}(\Lambda)\}_{j}$ convergences to $\mfC'$.
Since $\mfC'$ is a hyperplane cone, both terms in \eqref{eq:sheeting-smallness} then tend to zero.
Applying Theorem \ref{thm:epsilon-regularity-sheeting} gives multi-valued graphical decompositions near $Y_i$, with Lipschitz constants tending to zero by \eqref{eq:sheeting-estimate}.
Thus $Y_i$ is either a (immersed) regular point, or a branch point for all sufficiently large $i$, contradicting that $Y_i\in{\rm Sing}_e\bigl((\bseta_{X,r_i})_\#V\bigr)$.
Hence ${\rm Sing}_eV\cap B_1(0)$ is discrete.

Finally, if $n=2$, we observe from the definition of ${\rm Sing}_{b,2}$ that ${\rm Sing}_{b,2}V={\rm Sing}_bV$.
{
Also, a contradiction argument in the same spirit as above shows that $\bigl({\rm Sing}V\cap B_1(0)\bigr)\setminus{\rm Sing}_{b,2}V
    =\emptyset$.
}
In particular, we conclude ${\rm Sing}_e V=\emptyset$.
This completes the proof.
\end{proof}

\section{Branched stable minimal cones}
\label{sec:stable-cone-example}

In this section, we prove Theorem \ref{thm:stable-cone-sharpness}. The example is constructed by combining the Kapouleas--Wiygul desingularizations of intersecting Clifford tori \cite{KapouleasWiygul2022} with Lawson's polar map \cite{Lawson1970}.

We first record some notations. Consider the two Clifford tori
\eq{
  \mathbb T_1
  =\bigl\{e^{it}(\cos s,\sin s):s,t\in\mathbb R\bigr\},\quad
  \mathbb T_2
  =\bigl\{e^{it}(\cos s,i\sin s):s,t\in\mathbb R\bigr\},
}
which meet orthogonally along the two disjoint great circles
\eq{
  \Gamma_1=\{(e^{it},0):t\in\mathbb R\},
  \quad
  \Gamma_2=\{(0,e^{it}):t\in\mathbb R\}.
}
Note that each of the four components of
$(\mathbb T_1\cup\mathbb T_2)\setminus(\Gamma_1\cup\Gamma_2)$
is isometric to the flat cylinder
\eq{\label{eq:stable-cone-limit-cylinder}
  \Omega=(0,\frac{\pi}{2})\times(\mathbb R/2\pi\mathbb Z),
  \quad
  g_0=\rd s^2+\rd t^2,
}
and $\lvert A\rvert^2=2$ on each Clifford torus.

To our purpose, we use the Kapouleas-Wiygul construction with $(k,n_1,n_2,\sigma)=(2,1,1,0)$, and their gluing theorem \cite[Theorem 7.1 and Proposition 4.17]{KapouleasWiygul2022} yields, for every sufficiently large $m$, a closed embedded minimal surface desingularizing $\mathbb T_1\cup\mathbb T_2$:
\eq{\label{eq:stable-cone-kw-surface}
  M_m\subset\mathbb S^3,
  \quad\text{with }
  {\rm genus}(M_m)=4m+1.
}
In their construction, the Scherk necks collapse to $\Gamma_1\cup\Gamma_2$, while each of the remaining four regions converge smoothly on compact subsets to the flat cylinder $\Omega$. For completeness, we record the construction together with some useful estimates in Appendix \ref{app:stable-cone-computations}. In particular, if we define
\eq{\label{eq:stable-cone-schrodinger-operator}
  \mathcal L_m
  =-\Delta_{M_m}+\frac18\lvert A_{M_m}\rvert^2,
}
and let $\lambda_m$ be its first eigenvalue, so that
\eq{\label{eq:stable-cone-rayleigh-quotient}
  \lambda_m
  =
  \inf_{0\ne u\in H^1(M_m)}
  \frac{
    \displaystyle\int_{M_m}
    \left(\lvert\na^{M_m} u\rvert^2
    +\frac18\lvert A_{M_m}\rvert^2u^2\right)\rd\mu_{M_m}}
  {\displaystyle\int_{M_m}u^2\,\rd\mu_{M_m}},
}
then the curvature concentrated in the shrinking Scherk necks will force the limit (as $\lambda\ra\infty$) of the eigenfunctions (with respect to $\lambda_m$) to have zero trace on the ends of the four cylinders, yielding the following following spectral lower bound:

\begin{proposition}\label{thm:stable-cone-spectral-lower-bound}
For the surfaces in \eqref{eq:stable-cone-kw-surface},
\eq{\label{eq:stable-cone-spectral-lower-bound}
  \liminf_{m\to\infty}\lambda_m\geq\frac{17}{4}.
}
In particular, $\lambda_m>2$ for all sufficiently large $m$.
\end{proposition}

We point out that, the constant $\frac{17}{4}$ is exactly the sum of the first Dirichlet eigenvalue $4$ of $-\Delta$ on the flat cylinder $\Omega$ and the limiting potential $\frac{1}{8}\lvert A\rvert^2=\frac{1}{4}$. For a detailed proof, see Appendix \ref{app:stable-cone-computations}.

Then we recall Lawson's polar map \cite{Lawson1970}:
Let $Y\colon M^2\to\mathbb S^3$ be a smooth oriented, minimal immersion of a closed connected surface, with global unit normal $\nu$, induced metric $g_Y$, and shape operator $S_Y$, and suppose that $Y$ is not totally geodesic.
For the polar map
\eq{\label{eq:stable-cone-polar-map}
  F=\nu\colon M\ra\mathbb S^3,
}
we have the following useful properties.
\begin{proposition}\label{prop:stable-cone-polar-identities}
On the set $U=\{\lvert A_Y\rvert>0\}$, the map $F$ is a minimal immersion and $Y$ is a unit normal to $F$. Moreover,
\eq{\label{eq:stable-cone-polar-identities}
g_F
=\frac{\lvert A_Y\rvert^2}{2}g_Y,\qquad
\lvert A_F\rvert^2=\frac{4}{\lvert A_Y\rvert^2},\qquad
\lvert A_F\rvert^2\,\rd\mu_F=2\,\rd\mu_Y.
}
The map $F$ extends across the finite branch set $Z\coloneqq M\setminus U$ as a smooth branched minimal immersion \footnote{The notion agrees with the so-called {\em generalized minimal surface} in \cite{Lawson1970}, see also \cite{Osserman86}.}.
\end{proposition}

\begin{proof}
By the Weingarten equation and minimality, we have
\eq{
  \rd F(X)
  =-\rd Y(S_YX),
  \quad
  S_Y^2
  =\frac{\lvert A_Y\rvert^2}{2}I,
}
The first identity in \eqref{eq:stable-cone-polar-identities} then follows.
As shown in \cite[Section 10]{Lawson1970}, on $U$, the map $Y$ is in fact the Gauss map of $F$. It follows that $S_F=S_Y^{-1}$,
which proves the second identity in \eqref{eq:stable-cone-polar-identities} and the minimality of $F$.
Combining the first two identities in \eqref{eq:stable-cone-polar-identities}, we obtain the third.

It is thus left to consider the set $Z=M\setminus U$. Since $Y:M^2\ra\mbS^3$ is minimal and not totally geodesic, we know its Hopf differential is holomorphic (cf. \cite[Lemma 1.2]{Lawson1970}) with finitely many zeros, which are exactly the points at which $A_Y=0$ (cf. \cite[Lemma 1.4]{Lawson1970}).
By the first identity in \eqref{eq:stable-cone-polar-identities}, we see that $Z=\{p\in M: \rd F(p)=0\}$, and hence $Z$ is indeed the branch set (see also \cite[Proposition 10.1]{Lawson1970}). Hence the smooth map $F=\nu:M\ra\mbS^3$ is a smooth branched minimal immersion with branch set $Z$.
\end{proof}

For the above branched minimal immersion $F\colon M^2\to\mathbb S^3$, we define the quadratic form
\eq{\label{eq:stable-cone-stability-quadratic-form}
  \mathcal J_F(\phi)
  =
  \int_{M\setminus Z}
  \left(
    \lvert\nabla^F\phi\rvert^2
    -\lvert A_F\rvert^2\phi^2
    +\frac14\phi^2
  \right)\rd\mu_F.
}
As a direct consequence of Proposition \ref{prop:stable-cone-polar-identities}, we find:
\begin{corollary}\label{prop:stable-cone-polar-quadratic-form}
For the polar map in \eqref{eq:stable-cone-polar-map} and every smooth function $\phi$ on $M$,
\eq{\label{eq:stable-cone-polar-quadratic-form}
  \mathcal J_F(\phi)
  =
  \int_M
  \left(
    \lvert\nabla^Y\phi\rvert^2
    +\frac18\lvert A_Y\rvert^2\phi^2
    -2\phi^2
  \right)\rd\mu_Y.
}
\end{corollary}

\begin{proof}[Proof of Theorem \ref{thm:stable-cone-sharpness}]

By Proposition \ref{thm:stable-cone-spectral-lower-bound}, we can choose a sufficiently large $m$ such that $\lambda_m>2$. Then let $\nu_m$ be a global unit normal to the minimal surface $M_m$ (given by \eqref{eq:stable-cone-kw-surface}) and set the polar map $F_m=\nu_m\colon M_m\ra\mathbb S^3$.
By Corollary \ref{prop:stable-cone-polar-quadratic-form}, in conjunction with Proposition \ref{thm:stable-cone-spectral-lower-bound}, we find
\eq{\label{eq:stable-cone-positive-quadratic-form}
  \mathcal J_{F_m}(\phi)
  \geq
  (\lambda_m-2)\int_{M_m}\phi^2\,\rd\mu_{M_m}>0
}
for every nonzero smooth $\phi$. Moreover, by Proposition \ref{prop:stable-cone-polar-identities}, the branch set $Z_m$ of $F_m:M_m\ra\mbS^3$ is finite, and indeed, we have (cf. \cite[pp. 338]{Lawson1970})
\eq{
4{\rm genus}(M_m)-4
  =16m,
}
thanks to \eqref{eq:stable-cone-kw-surface}.

To proceed, let $\mathbf C_m\subset\mbR^4$ be the cone over $F_m$, given by
\eq{\label{eq:stable-cone-parametrization}
  X_m(r,p)=rF_m(p),
  \quad\forall
  (r,p)\in(0,\infty)\times M_m,
}
which is a stationary integral cone. Moreover, by \eqref{eq:stable-cone-positive-quadratic-form} and Simons' cone-stability criterion \cite[Lemmas~6.1.5 and 6.1.6]{Simons68}, we see that $(0,\infty)\times(M_m\setminus Z_m)$ is stable minimal in $\mbR^4$. Since $Z_m$ is a finite set, by standard approximation argument (cf. \cite{SS81}) we thus find that $\mfC_m$ itself is a stable minimal cone in $\mbR^4$. Note also that $\mfC_m$ is non-flat since on the nonempty open set where $\lvert A_{M_m}\rvert>0$, the polar map has rank two and the second identity in \eqref{eq:stable-cone-polar-identities} gives $\lvert A_{F_m}\rvert>0$. Thus $\mathbf C_m$ is non-flat. Since the unique tangent cone to $\mathbf C_m$ at its vertex is $\mathbf C_m$ itself, the vertex is not a branch point in the sense of \eqref{eq:precompactness-branch-cone}. Hence we conclude $0\in{\rm Sing}_e\mathbf C_m$.

Finally, for every $n\geq3$, by a standard argument using Fubini theorem, we see that the cone $\mathbf C_m\times\mathbb R^{n-3}$ is stationary, integral, two-sided, and stable in $\mbR^{n+1}$, which satisfies $\{0\}\times\mathbb R^{n-3}\subset{\rm Sing}_e\bigl(\mathbf C_m\times\mathbb R^{n-3}\bigr)$. This completes the proof.
\end{proof}

\appendix

\section{Algebraic trace estimates}\label{App:algebraic-lemmas}

{
\begin{proof}[Proof of Lemma \ref{lem:algebraic-trace-estimate}]

For any $S\in{\rm Sym}_0(\mathbb E)$, we write
$S=
\begin{pmatrix}
A&X^*\\
X&D
\end{pmatrix}$
with respect to the decomposition $\mathbb E=\mathbb F_1\oplus\mathbb F_2$.
By the definition of ${\rm Sym}_0(\mbE,\mbF_1)$, each of its elements
vanishes on $\mbF_2$.  Define
\eq{
S_0=
\begin{pmatrix}
A-\dfrac{{\rm tr}A}{\dim\mathbb F_1}
I_{\mathbb F_1}&0\\
0&0
\end{pmatrix},
\qquad
S_\perp=S-S_0
=
\begin{pmatrix}
\dfrac{{\rm tr}A}{\dim\mathbb F_1}
I_{\mathbb F_1}&X^*\\
X&D
\end{pmatrix},
}
then \(S=S_0+S_\perp\) is the orthogonal decomposition with respect to
${\rm Sym}_0(\mbE,\mbF_1)$.

By the Cauchy-Schwarz inequality and
${\rm tr}S=0$, we find
\eq{\label{eq:algebraic-kernel-gap}
\lvert S_\perp\rvert^2
&=
2\lvert X\rvert^2+\lvert D\rvert^2
+\frac{({\rm tr}D)^2}{\dim\mathbb F_1}\leq
n\bigl(\lvert X\rvert^2+\lvert D\rvert^2\bigr)
=n{\rm tr}(S\pi_{\mathbb F_2}S).
}
On the other hand, we have
\eq{
{\rm tr}(SBS)
&=
{\rm tr}(S_0BS_0)
+2{\rm tr}(S_0BS_\perp)
+{\rm tr}(S_\perp BS_\perp)\\
&\geq
c\lvert S_0\rvert^2
-2C_B\lvert S_0\rvert\lvert S_\perp\rvert
-C_B\lvert S_\perp\rvert^2\\
&\geq
\frac{c}{2}\lvert S_0\rvert^2
-\left(C_B+\frac{2C_B^2}{c}\right)\lvert S_\perp\rvert^2,
}
where we have used the assumption on \(B\) to \(S_0\) for the first inequality.
Combining with \eqref{eq:algebraic-kernel-gap}, we thus obtain
\eq{
{\rm tr}
\left(S(\pi_{\mathbb F_2}+\varepsilon B)S\right)
&\geq
\left(\frac{1}{n}-\left(C_B+\frac{2C_B^2}{c}\right)\varepsilon\right)
\lvert S_\perp\rvert^2
+\frac{c\varepsilon}{2}\lvert S_0\rvert^2.
}
Therefore, we can choose $\tilde{\varepsilon}=\tilde{\varepsilon}(n,c,C_B)>0$ sufficiently small so that for any $\varepsilon\in(0,\tilde{\varepsilon})$,
\eq{
{\rm tr}
\left(S(\pi_{\mathbb F_2}+\varepsilon B)S\right)
\geq\min\left\{\frac{1}{2n},\frac{c}{2}\right\}
\varepsilon\lvert S\rvert^2.
}
This completes the proof.
\end{proof}
}

{
\begin{proof}[Proof of Lemma \ref{lem:schur-trace-estimate}]
For any $z=(u,v)\in\mbF_1\oplus\mbF_2$, completing the square yields
\eq{\label{eq:Schur-complement-identity}
\mcQ(z,z)
=
\widehat{\mcQ}(u,u)
+
\mcQ_{22}\left(
v+\mcQ_{22}^{-1}\mcQ_{21}u,
v+\mcQ_{22}^{-1}\mcQ_{21}u
\right).
}

For any $S\in{\rm Sym}_0(\mbE)$, we write
$S=
\begin{pmatrix}
A&X^*\\
X&D
\end{pmatrix}$
with respect to the decomposition $\mbE=\mbF_1\oplus\mbF_2$, and put for simplicity
\eq{
W
\coloneqq\pi_{\mbF_2}S+\mcQ_{22}^{-1}\mcQ_{21}\pi_{\mbF_1}S.
}
Choose an orthonormal basis $\{e_\alpha\}_{\alpha=1}^n$ such that
\eq{
\mbF_1
={\rm span}\{e_1,e_2\},\qquad
\mbF_2
={\rm span}\{e_3,\ldots,e_n\}.
}
Using \eqref{eq:Schur-complement-identity}, $S=S^\ast$, and that $\mcQ_{22}\geq c\frac{\mu}{\de}I_{\mbF_2}$, we find
\eq{\label{eq:algebraic-schur-first-bound}
{\rm tr}(S\mcQ S)
&=\sum_{\alpha=1}^n\mcQ\left(Se_\alpha,Se_\alpha\right)=\sum_{\alpha=1}^n\widehat {\mcQ}
\left(\pi_{\mbF_1}Se_\alpha,\pi_{\mbF_1}Se_\alpha\right)+\sum_{\alpha=1}^n\mcQ_{22}(We_\alpha,We_\alpha)\\
&\geq\sum_{\alpha=1}^n\widehat Q
(\pi_{\mbF_1}Se_\alpha,\pi_{\mbF_1}Se_\alpha)
+c\frac{\mu}{\de}\lvert W\rvert^2.
}
We next estimate the first term.
Let $\mathring A$ be the traceless part of $A$, by direct computation
\eq{
{\rm tr}\left(A\widehat{Q}A\right)
={\rm tr}\left(\mathring{A}\widehat{Q}\mathring{A}\right)+{\rm tr}(A){\rm tr}\left(\mathring{A}\widehat{Q}\right)+\frac{({\rm tr}(A)^2)}{4}{\rm tr}\widehat{Q},
}
and by the assumption that $\abs{\widehat{\mcQ}}\leq C\mu$, we have
\eq{\label{ineq:estimate:trA-tr(mathringA-hatQ)}
\Abs{{\rm tr}(A){\rm tr}(\mathring{A}\widehat{\mcQ})}
\leq C'\mu\abs{{\rm tr}(A)}\abs{\mathring{A}}
\leq\frac{c}{4}\mu\abs{\mathring{A}}^2+C''\mu({\rm tr}A)^2,
}
where $C',C''$ are positive constants depending only on $c,C$.
Since $\dim\mbF_1=2$, we have
$\mathring A^2=\frac12\lvert\mathring A\rvert^2
I_{\mbF_1}$, and hence (note that $\pi_{\mbF_1}S=\begin{pmatrix}
    A&X^\ast
\end{pmatrix}$ with respect to $\mbE=\mbF_1\oplus\mbF_2$)
\eq{\label{eq:algebraic-schur-Qhat-bound}
{}&\sum_{\alpha=1}^n\widehat {\mcQ}(\pi_{\mbF_1}Se_\alpha,\pi_{\mbF_1}Se_\alpha)
=
{\rm tr}(A\widehat{\mcQ}A)
+
\sum_{\alpha=3}^n\widehat {\mcQ}(X^*e_\alpha,X^*e_\alpha)\\
={}&
\frac{\lvert\mathring A\rvert^2}{2}
{\rm tr}\widehat{\mcQ}
+
{\rm tr}(A){\rm tr}(\mathring A\widehat{\mcQ})
+
\frac{\bigl({\rm tr}A\bigr)^2}{4}
{\rm tr}\widehat{\mcQ}
+
\sum_{\alpha=3}^n\widehat{\mcQ}(X^*e_\alpha,X^*e_\alpha)\\
\geq{}&
c_0\mu\lvert\mathring A\rvert^2
-C_0\mu\left(\bigl({\rm tr}A\bigr)^2+\lvert X\rvert^2\right)\ge
c_0\mu\lvert\mathring A\rvert^2
-C_1\mu\left(\lvert X\rvert^2+\lvert D\rvert^2\right)\\
={}&
c_0\mu\lvert\mathring A\rvert^2-C_1\mu\lvert \pi_{\mbF_2}S\rvert^2,
}
where $c_0,C_1>0$ depend only on $c,C$. Here we have used ${\rm tr}\widehat{\mcQ}\geq c\mu$,
$\abs{\widehat{\mcQ}}\leq C\mu$, \eqref{ineq:estimate:trA-tr(mathringA-hatQ)} for the first inequality; ${\rm tr}A=-{\rm tr}D$ and Cauchy-Schwarz for the second inequality.

It remains to estimate the last term. Since $\pi_{\mbF_2}S=W-\mcQ_{22}^{-1}\mcQ_{21}\pi_{\mbF_1}S$ and $\lvert \mcQ_{22}^{-1}\mcQ_{21}\rvert\leq C\sqrt{\de}$, we have
\eq{
\lvert \pi_{\mbF_2}S\rvert^2
&\leq 2\lvert W\rvert^2+C\de\lvert \pi_{\mbF_1}S\rvert^2=2\lvert W\rvert^2
+C\de\left(\lvert A\rvert^2+\lvert X\rvert^2\right)\leq 2\lvert W\rvert^2
+C_2\de\left(\lvert\mathring A\rvert^2
+\lvert \pi_{\mbF_2}S\rvert^2\right),
}
where $C_2=C_2(n,C)>0$. For the last inequality, we have used $\abs{A}^2=\abs{\mathring{A}}^2+\frac{({\rm tr}A)^2}{2}$, ${\rm tr}(A)=-{\rm tr}D$, $({\rm tr}D)^2\leq(n-2)\abs{D}^2$, and also $\abs{\pi_{\mbF_2}S}^2=\abs{X}^2+\abs{D}^2$.
After decreasing $\widetilde\de=\widetilde\de(n,c,C)$, we may absorb
the resulting $C_2\de\lvert \pi_{\mbF_2}S\rvert^2$ term to the left and obtain
\eq{\label{eq:algebraic-schur-V-bound}
\lvert \pi_{\mbF_2}S\rvert^2
\leq
4\lvert W\rvert^2+C_2\de\lvert\mathring A\rvert^2.
}
As a by-product of the proof, we also have
\eq{\label{eq:algebraic-schur-F1-bound}
\abs{\pi_{\mbF_1}S}^2
=\abs{A}^2+\abs{X}^2
\leq\abs{\mathring{A}}^2+C(n)\abs{\pi_{\mbF_2}S}^2.
}
Using first
\eqref{eq:algebraic-schur-first-bound},
\eqref{eq:algebraic-schur-Qhat-bound}, then further shrinking $\widetilde\de$, and finally
\eqref{eq:algebraic-schur-V-bound}, \eqref{eq:algebraic-schur-F1-bound}, we deduce
\eq{
{\rm tr}(S\mcQ S)
\geq&
\left(c_0-C_1C_2\de\right)\mu\lvert\mathring A\rvert^2+
\left(\frac{c}{\de}-4C_1\right)\mu\lvert W\rvert^2\ge
c_2\mu\left(\lvert\mathring A\rvert^2+\lvert W\rvert^2\right)\ge\tilde C\mu\abs{S}^2,
}
where $\widetilde C>0$ depends only on $n,c,C$.
This completes the proof.
\end{proof}
}

\section{The Kapouleas--Wiygul gluing construction and estimates}
\label{app:stable-cone-computations}

\subsection{The Kapouleas--Wiygul construction}
\label{app:stable-cone-kw-geometry}

We first record the Kapouleas--Wiygul gluing construction \cite{KapouleasWiygul2022} used in Section~\ref{sec:stable-cone-example}.

\subsubsection{The initial surface}

Write $(x,y)=(\rho\cos\theta,\rho\sin\theta)$ and define toral coordinates around $\Gamma_1$ by
\eq{\label{eq:stable-cone-toral-coordinates}
\Phi(\rho\cos\theta,\rho\sin\theta,z)
=e^{iz}(\cos\rho,e^{i\theta}\sin\rho),
}
where the map $\Phi:\mbR^3\ra\mbS^3$ is as in \cite[(4.1)]{KapouleasWiygul2022}. Note that the $z$-axis maps to $\Gamma_1$, while the vertical half-planes $\theta=0$ and $\theta=\pi/2$ map to $\mathbb T_1$ and $\mathbb T_2$, respectively.
Also, the map is $2\pi$-periodic in $z$.

With $k=2$ chosen in \cite[Proposition~2.6]{KapouleasWiygul2022}, the Karcher--Scherk tower \(\mathcal S_2\) is exactly the classical {\em singly periodic Scherk surface}
\eq{\label{eq:stable-cone-scherk-surface}
\Sigma_{\mathrm{Sch}}
=\{(x,y,z)\in\mathbb R^3:\sinh x\sinh y=\sin z\},
}
which has period $2\pi$ in the $z$-direction and four exponentially asymptotic planar wings.
Put as in \cite[(4.4)]{KapouleasWiygul2022}
\eq{\label{eq:stable-cone-neck-scale}
a_m=\frac{m\pi}{4}-10,
}
and let $\psi[a,b]$ be the fixed cutoff as in \cite[(1.7)]{KapouleasWiygul2022}.
For simplicity we write
\eq{
\chi_m(x)
=\psi[a_m+1,a_m](x),
}
which satisfies $\chi_m=1$ for $x\leq a_m$ and $\chi_m=0$ for $x\geq a_m+1$. Let $R_2$ be a fixed radius beyond which the four wings of $\Sigma_{\mathrm{Sch}}$ are graphical, as in \cite[Proposition~2.6]{KapouleasWiygul2022}.
Let $W_2$ denote the exponentially decaying graphing function for the positive $x$-wing as in \cite[Proposition~2.6]{KapouleasWiygul2022}, so that this wing is written as
\eq{
  \{(x,W_2(x,z),z):x\geq R_2,\ z\in\mathbb R\}.
}
Following \cite[(4.6)]{KapouleasWiygul2022}, we straighten this wing by
\eq{\label{eq:stable-cone-straightened-wing}
  (x,W_2(x,z),z)
  \longmapsto
  (x,\chi_m(x)W_2(x,z),z),
}
and the other three straightened wings are defined by the Scherk symmetries.
Truncate the straightened tower at
\eq{
x^2+y^2
=(m\pi/2)^2.
}
In the notation of \cite[(4.7)]{KapouleasWiygul2022}, the resulting surface is
\eq{
\widetilde\Sigma_m
=\widetilde{\mathcal S}_{2,m}(m\pi/2).
}
After multiplication by $\tau_m\coloneqq\frac{1}{2m}$, its boundary radius is $\pi/4$ and its period is $\pi/m$. Choosing the parameters $k=2$, $n_1=n_2=1$ in \cite{KapouleasWiygul2022},
so that
\eq{
\tau_m
=\frac{1}{2m}=(kmn_j)^{-1},
\quad
kmn_j\frac{\pi}{4}=\frac{m\pi}{2}.
}
Consequently the first neck
\eq{
\Sigma_{m,1}
=\Phi(\tau_m\widetilde\Sigma_m)
}
contains exactly $2m$ periods and replaces a neighborhood of $\Gamma_1$.

Define as in \cite[(4.14)]{KapouleasWiygul2022} the ambient isometry (which is denoted by ${\rm Rot}_{C'_1}^{\pi}$ therein)
\eq{\label{eq:stable-cone-coordinate-interchange}
  \mathcal R\colon\mathbb S^3\ra\mathbb S^3,
  \quad
  \mathcal R(z_1,z_2)=(z_2,z_1),
}
which gives the second neck
\eq{
\Sigma_{m,2}
=\mathcal R(\Sigma_{m,1}).
}
The straightened toral wings of the two pieces meet smoothly at their common radius-$\pi/4$ boundary. Hence their union is the
closed smooth initial surface
\eq{\label{eq:stable-cone-kw-initial-surface}
\mathring M_m
=\Sigma_{m,1}\cup\Sigma_{m,2}
=M(2,m,1,1,0)
}
in the notation of \cite[Definition~4.13 and (4.14)]{KapouleasWiygul2022}.
By \cite[Proposition~4.17]{KapouleasWiygul2022},
\eq{
{\rm genus}(\mathring M_m)
=2(2-1)m(1+1)+1
=4m+1.
}

Let $\mathring\nu_m$ be a unit normal to $\mathring M_m$. By \cite[Theorem 7.1]{KapouleasWiygul2022}, $\mathring M_m$ can be perturbed to yield the minimal surface $M_m$ in \eqref{eq:stable-cone-kw-surface} by a normal graph map
\eq{\label{eq:stable-cone-normal-graph-map}
  \mathcal G_m(p)
  =\exp_p^{\mathbb S^3}\bigl(w_m(p)\mathring\nu_m(p)\bigr),
  \text{ with }
  \lVert w_m\rVert_{C^2(\mathring M_m,m^2g_{\mathring M_m})}
  \leq Cm^{-3/2}.
}
Here $w_m$ is denoted by $u$ in \cite[Theorem 7.1]{KapouleasWiygul2022}.

\subsubsection{Cells, necks, and bulk regions}

To establish the estimates in the following section, we introduce for simplicity some new notations and terminologies in contrast to \cite{KapouleasWiygul2022}.

Fix a sufficiently large $b>R_2$ and let
\eq{\label{eq:stable-cone-scherk-cell}
K
=\Sigma_{\mathrm{Sch}}\cap\left\{\frac\pi2\leq z\leq\frac{5\pi}{2},\ \lvert x\rvert\leq b,\ \lvert y\rvert\leq b\right\}.
}
With the choice of $b$, the set $K$ is a compact connected surface with Lipschitz boundary.
Set the {\em wing-cut} and {\em period-cut} boundaries as
\eq{
\partial_{\mathrm{wing}}K
\coloneqq&K\cap\left(\left\{\abs{x}=b\right\}\cup\left\{\abs{y}=b\right\}\right),\\
\partial_{\mathrm{per}}K
\coloneqq&K\cap\left(\left\{z=\frac{\pi}{2}\right\}\cup\left\{z=\frac{5\pi}{2}\right\}\right).
}
Namely, $\partial_{\mathrm{wing}}K$ consists of the four arcs on $\lvert x\rvert=b$ or $\lvert y\rvert=b$, and $\partial_{\mathrm{per}}K$ consists of the two period-cut sides.
Let $g_K$ and $A_K$ denote the metric and second fundamental form induced from $\mathbb R^3$, and let $\rd\mu_K$ and $\rd\sigma_K$ denote the corresponding area and wing-boundary measures.
Define the map
\eq{
T_j(x,y,z)
=(x,y,z+2\pi(j-1)),
\quad 1\leq j\leq2m,
}
and put
\eq{
\mathring\Psi_{m,1,j}(p)
=\Phi(\tau_mT_j(p)),\quad
\mathring\Psi_{m,2,j}(p)
=\mathcal R\bigl(\Phi(\tau_mT_j(p))\bigr),\quad p\in K,
}
where the two values of the first subscript correspond to the necks around $\Gamma_1$ and $\Gamma_2$, respectively. For all sufficiently large $m$, one has $b<a_m$ and $\sqrt2b<\frac{m\pi}{2}$. Therefore, the straightening and radial truncation do not change the translated cells $\{T_j(K)\}_j$, so that the maps
$\mathring\Psi_{m,i,j}$ are well defined with image in $\mathring M_m$.

Then we define the {\em physical cells} by
\eq{\label{eq:stable-cone-physical-cells}
  \Psi_{m,i,j}=\mathcal G_m\circ\mathring\Psi_{m,i,j},
  \quad
  K_{m,i,j}=\Psi_{m,i,j}(K),
  \quad
  i\in\{1,2\},\ 1\leq j\leq2m.
}
Note that, by construction, their period-cut sides are identified cyclically.
Figure~\ref{fig:stable-cone-cell-map} illustrates one of these cell maps.
\begin{figure}[htbp]
  \centering
  \includegraphics[width=\textwidth]{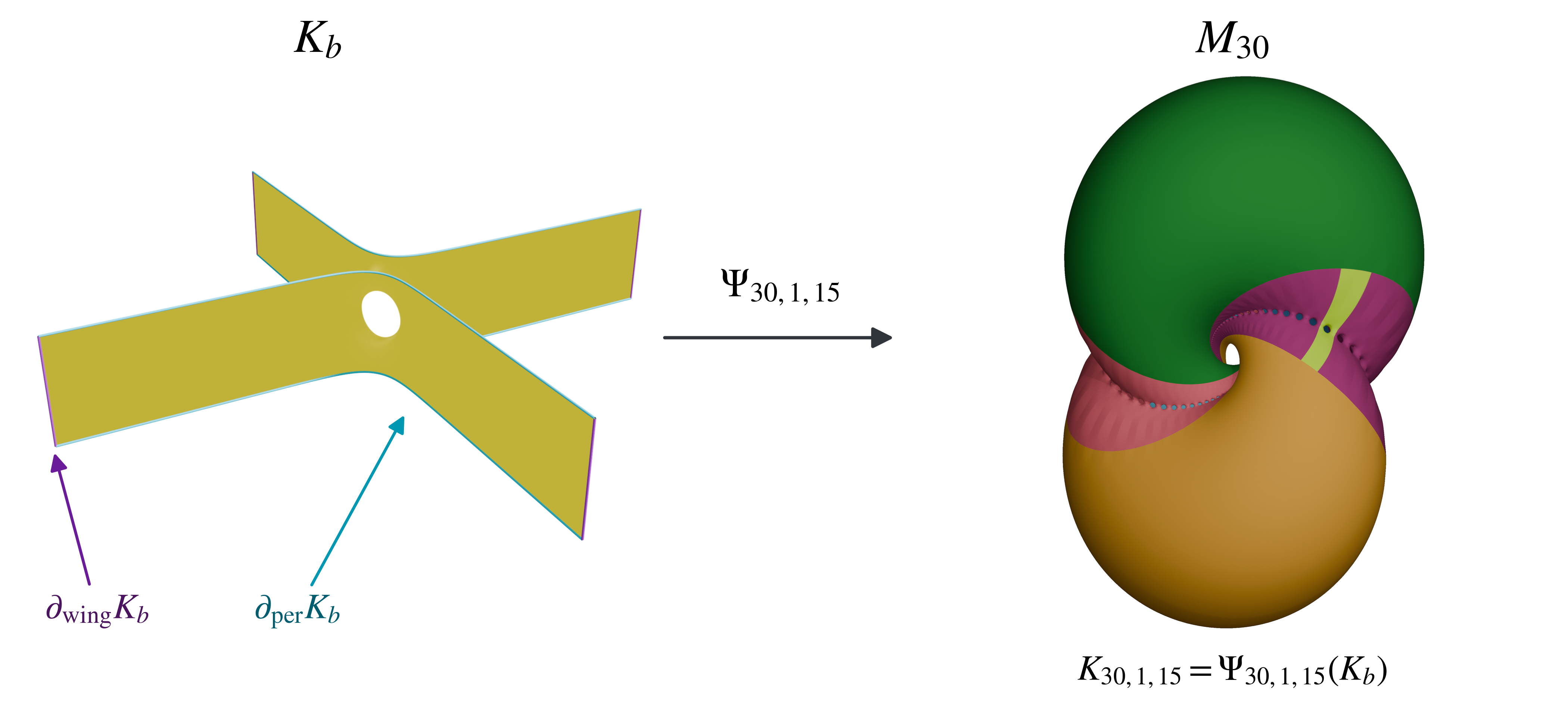}
  \caption{The map $\Psi_{30,1,15}$ from a truncated Scherk cell to the
  corresponding cell in $M_{30}$ is illustrated here. The truncated Scherk
  cell $K_b$ is shown on the left. The notation $K_b$ emphasizes its
  dependence on the truncation parameter $b$; this is the cell denoted by $K$
  in the context. On the right, we show a stereographic projection of
  $M_{30}\subset\mathbb S^3$, which is obtained by desingularizing two
  orthogonally intersecting Clifford tori. The corresponding physical cell
  $K_{30,1,15}=\Psi_{30,1,15}(K)$ is highlighted in yellow, illustrating the
  image of the Scherk cell shown on the left under $\Psi_{30,1,15}$. An
  \href{https://gaomw.com/interactive/branched-stable-example/}{interactive
  three-dimensional visualization} is also available online.}
  \label{fig:stable-cone-cell-map}
\end{figure}
Finally, put the {\em necks} and their boundary as
\eq{\label{eq:stable-cone-neck-interface}
  N_m=\bigcup_{i=1}^2\bigcup_{j=1}^{2m}K_{m,i,j},
  \quad
  \Gamma_m
  =\bigcup_{i=1}^2\bigcup_{j=1}^{2m}
  \Psi_{m,i,j}(\partial_{\mathrm{wing}}K).
}
The closures of the four components of $M_m\setminus N_m$, called {\em bulk regions}, are denoted by $B_{m,1},\dots,B_{m,4}$, which in fact correspond to the truncated portions of the four toral extended standard regions $S[T]$ in \cite[Section~5]{KapouleasWiygul2022}.
Here the cut is made at distance $\frac{b}{2m}=b\tau_m$ from the adjacent intersection circle, see \cite[Fig. 1]{KapouleasWiygul2022}.
Thus
\eq{\label{eq:stable-cone-surface-decomposition}
M_m
=N_m\cup\bigcup_{a=1}^4B_{m,a},\quad
N_m\cap\bigcup_{a=1}^4B_{m,a}
=\Gamma_m.
}

The following convergence results will be used in due course.

\begin{proposition}\label{prop:stable-cone-geometric-decomposition}
Uniformly in $i\in\{1,2\}$ and $1\leq j\leq2m$,
\eq{\label{eq:stable-cone-cell-convergence}
  \left\|
    \tau_m^{-2}\Psi_{m,i,j}^*g_{M_m}-g_K
  \right\|_{C^0(K,g_K)}
  +
  \left\|
    \tau_m^2\Psi_{m,i,j}^*\lvert A_{M_m}\rvert^2
    -\lvert A_K\rvert^2
  \right\|_{C^0(K,g_K)}
  \ra0.
}
For each $a\in\{1,\cdots,4\}$, the interior of $B_{m,a}$ has a parametrization
\eq{\label{defn:Om_m}
\Phi_{m,a}
\colon\Omega_m\ra{\rm int}B_{m,a},\quad
\Omega_m
=(b\tau_m,\frac{\pi}{2}-b\tau_m)\times(\mathbb R/2\pi\mathbb Z).
}
More precisely, let $T$ be the corresponding component of $(\mathbb T_1\cup\mathbb T_2)\setminus(\Gamma_1\cup\Gamma_2)$, and let $T_b$ be the part left after deleting the two strips of width $b/(2m)$ adjacent to its boundary circles.
The inverse of the projection $\varpi_{T,m}\colon S[T]\to T_b$ in \cite[Section~5]{KapouleasWiygul2022}, followed by the normal-graph map $\mathcal G_m$, gives $\Phi_{m,a}$ after identifying $T$ with its flat cylinder coordinates $(s,t)$. This parametrization extends to $\overline\Omega_m$ and identifies its two boundary circles with the corresponding components of $\Gamma_m$. Let $g_{m,a}=\Phi_{m,a}^*g_{M_m}$, and let $\rd\sigma_{m,a}$ be the line element induced by $g_{m,a}$ on $\partial\Omega_m$. There is a constant $C$, independent of $m$ and $a$, such that (recalling \eqref{eq:stable-cone-limit-cylinder})
\eq{\label{eq:stable-cone-bulk-equivalence}
  C^{-1}g_0\leq g_{m,a}\leq Cg_0,
  \quad
  C^{-1}\rd t\leq\rd\sigma_{m,a}\leq C\rd t.
}
On every compact subset of $\Omega$, uniformly in $a$, one has
\eq{\label{eq:stable-cone-bulk-convergence}
  g_{m,a}\ra g_0\quad\text{in }C^0,
  \quad
  \lvert A_{M_m}\rvert^2\circ\Phi_{m,a}\ra2
  \quad\text{uniformly}.
}
\end{proposition}

\begin{proof}
In our case, the period number in the notation of \cite{KapouleasWiygul2022} is $m_C=2m$, so that
$\frac{m^2}{m_C^2}g_K=\frac14g_K$.
Since $\tau_m^{-2}=4m^2$, multiplying \cite[Proposition 5.18 (i)]{KapouleasWiygul2022} by $4$ gives the metric convergence for the initial hypersurface ${\mathring{M}_m}$ in the form of \eqref{eq:stable-cone-cell-convergence}.
Thanks to \eqref{eq:stable-cone-normal-graph-map}, $\|mw_m\|_{C^2(\mathring M_m,m^2g_{\mathring M_m})}\leq Cm^{-1/2}$, it follows that on $\mathring{M}_m$, there holds
\eq{\label{ineq:initial-and-graph}
(1-\varepsilon_m)g_{\mathring{M}_m}
\leq\mathcal{G}^\ast_mg_{M_m}
\leq(1+\varepsilon_m)g_{\mathring{M}_m}\text{ with }\varepsilon_m\ra0.
}
which implies the first convergence in \eqref{eq:stable-cone-cell-convergence}. The second convergence can be proved similarly using \cite[Proposition 5.18 (ii)]{KapouleasWiygul2022}.

For each toral component $T$, the region $S[T]$ is a graph over the truncated cylinder $T_b$, and its two boundary circles are attached to planar ends of the scaled Scherk necks. By \cite[Proposition~2.6(viii)]{KapouleasWiygul2022}, these ends converge exponentially to their asymptotic planes. Combining this decay with \cite[Proposition~5.12(i)]{KapouleasWiygul2022}, we obtain, uniformly up to the moving boundary of $T_b$,
\eq{
  C^{-1}g_0
  \leq
  (\varpi_{T,m}^{-1})^*g_{\mathring M_m}
  \leq
  Cg_0.
}
The estimate \eqref{ineq:initial-and-graph} transfers this comparison to $M_m$ under the parametrizations $\Phi_{m,a}$. Restricting the metric comparison to $\partial\Omega_m$ gives the corresponding comparison of boundary line elements, and hence proves \eqref{eq:stable-cone-bulk-equivalence}. On compact subsets with a fixed positive distance from $\Gamma_1\cup\Gamma_2$, the final smooth convergence assertion of \cite[Theorem~7.1]{KapouleasWiygul2022}, together with the weighted estimates in \cite[Proposition~5.12(i),(ii)]{KapouleasWiygul2022}, yields \eqref{eq:stable-cone-bulk-convergence}.

\end{proof}

\subsection{The spectral lower bound}\label{app:stable-cone-spectral-lower-bound}

We divide the proof of Proposition \ref{thm:stable-cone-spectral-lower-bound} into the following Lemmatum.

\begin{lemma}\label{lem:stable-cone-cell-coercivity}
There is a constant $C$, depending only on $K$, such that every $v\in H^1(K)$ satisfies
\eq{\label{eq:stable-cone-model-cell-coercivity}
  \int_Kv^2\,\rd\mu_K
  +\int_{\partial_{\mathrm{wing}}K}v^2\,\rd\sigma_K
  \leq
  C\int_K
  \left(\lvert\na^{K} v\rvert^2
  +\frac18\lvert A_K\rvert^2v^2\right)\rd\mu_K.
}
\end{lemma}

\begin{proof}
For every $v\in H^1(K)$, by the trace estimate we have
\eq{
  \int_{\partial_{\mathrm{wing}}K}v^2\,\rd\sigma_K
  \leq
  C\int_K\left(\lvert\na^{K} v\rvert^2+v^2\right)\rd\mu_K.
}
Therefore it suffices to prove the estimate \eqref{eq:stable-cone-model-cell-coercivity} for $\int_Kv^2\rd\mu_K$. We argue by contradiction and assume there are $v_j\in H^1(K)$ such that
\eq{
  \int_Kv_j^2\,\rd\mu_K=1,
  \text{ but }
  \int_K
  \left(\lvert\na^{K} v_j\rvert^2
  +\frac18\lvert A_K\rvert^2v_j^2\right)\rd\mu_K
  \ra0.
}
By Rellich compactness, after passing to a subsequence, $v_j$ weakly converges in $H^1(K)$ and strongly converges in $L^2(K)$ to a function $v\in H^1(K)$. Since $\int_K\abs{\na^K v_j}^2\rd\mu_K\to0$, and $K$ is connected, we deduce that $v$ is constant. Moreover, by the strong convergence of $v_j$ in $L^2(K)$ together with $\int_Kv_j^2\rd\mu_K=1$, we see $v\neq0$.
On the other hand, the strong $L^2$ convergence and the boundedness of $\lvert A_K\rvert$ imply $v^2\int_K\lvert A_K\rvert^2\,\rd\mu_K=0$. Since the Scherk surface is nonflat, we see that the integral has to be positive, which is a contradiction. This finishes the proof.
\end{proof}

\begin{lemma}\label{lem:stable-cone-neck-coercivity}
There is a constant $C$, independent of $m$, such that every $u\in H^1(M_m)$ satisfies
\eq{\label{eq:stable-cone-neck-coercivity}
  \int_{N_m}u^2\,\rd\mu_{M_m}
  +\tau_m\int_{\Gamma_m}u^2\,\rd\sigma_{M_m}
  \leq
  C\tau_m^2
  \int_{N_m}
  \left(\lvert\nabla^{M_m} u\rvert^2
  +\frac18\lvert A_{M_m}\rvert^2u^2\right)\rd\mu_{M_m}.
}
\end{lemma}

\begin{proof}
Fix a $K_{m,i,j}$ and set
\eq{
v
=u\circ\Psi_{m,i,j}\quad\text{on }K.
}
Since $K$ is two dimension, we have
\eq{
\int_{K_{m,i,j}}u^2\,\rd\mu_{M_m}
=&\tau_m^2\int_Kv^2\rd\mu_{\tau_m^{-2}\Psi^\ast_{m,i,j}g_{M_m}},\\
\tau_m\int_{\Psi_{m,i,j}(\partial_{\mathrm{wing}}K)}
  u^2\,\rd\sigma_{M_m}
=&\tau_m^2\int_{\p_{\rm wing}K}v^2\rd\sigma_{\tau_m^{-2}\Psi^\ast_{m,i,j}g_{M_m}},\\
\int_{K_{m,i,j}}\abs{\na^{M_m} u}^2\rd\mu_{M_m}
=&\int_{K}
\left({\tau_m^{-2}\Psi^\ast_{m,i,j}g_{M_m}}\right)^{-1}(\rd v,\rd v)
\rd\mu_{\tau_m^{-2}\Psi^\ast_{m,i,j}g_{M_m}},\\
\int_{K_{m,i,j}}\abs{A_{M_m}}^2u^2\rd\mu_{M_m}
=&\int_K\tau^{2}_m\abs{A_{M_m}}^2\circ\Psi_{m,i,j}v^2\rd\mu_{\tau_m^{-2}\Psi^\ast_{m,i,j}g_{M_m}}.
}
By \eqref{eq:stable-cone-cell-convergence},
we know
\eq{
\tilde\varepsilon_m
\coloneqq\sup_{i,j}\left(\left\|
    \tau_m^{-2}\Psi_{m,i,j}^*g_{M_m}-g_K
  \right\|_{C^0(K,g_K)}
  +
  \left\|
    \tau_m^2\Psi_{m,i,j}^*\lvert A_{M_m}\rvert^2
    -\lvert A_K\rvert^2
  \right\|_{C^0(K,g_K)}\right)\ra0\text{ as }m\ra\infty.
}
Hence there exists $C=C(K)>0$, such that for all sufficiently large $m$, and for any $i,j$,
\eq{
\left(1-C\tilde\varepsilon_m\right)g_K
\leq\tau_m^{-2}\Psi^\ast_{m,i,j}g_{M_m}
\leq\left(1+C\tilde\varepsilon_m\right)g_K,\quad
\abs{A_K}^2
\leq\tau_m^2\abs{A_{M_m}}^2\circ\Psi_{m,i,j}+\tilde\varepsilon_m,
}
uniformly on $K$.
Therefore, after multiplying \eqref{eq:stable-cone-model-cell-coercivity} by \(\tau_m^2\) and combining the above estimates, one obtains
\eq{
  \int_{K_{m,i,j}}u^2\,\rd\mu_{M_m}
  +\tau_m\int_{\Psi_{m,i,j}(\partial_{\mathrm{wing}}K)}
  u^2\,\rd\sigma_{M_m}
  \leq
  C\tau_m^2
  \int_{K_{m,i,j}}
  \left(\lvert\na^{M_m} u\rvert^2
  +\frac18\lvert A_{M_m}\rvert^2u^2\right)\rd\mu_{M_m},
}
where $C=C(K)>0$.
Summing over $i\in\{1,2\}$ and $1\leq j\leq2m$, we obtain \eqref{eq:stable-cone-neck-coercivity}.
\end{proof}

\begin{lemma}\label{lem:stable-cone-strip-estimates}
There is a constant $C$, independent of $m$ and $\delta$, such that if $b\tau_m<\delta<\frac{\pi}{8}$, then for every $v\in H^1(\Omega_m,g_{m,a})$, there hold
\begin{align}
  \int_{\{b\tau_m<s<\delta\}}v^2\,\rd\mu_{g_{m,a}}
  &\leq
  C\delta\int_{\{s=b\tau_m\}}v^2\,\rd\sigma_{m,a}
  +C\delta^2\int_{\{b\tau_m<s<\delta\}}
  \lvert\na^{g_{m,a}} v\rvert^2\,\rd\mu_{g_{m,a}},
  \label{eq:stable-cone-strip-mass}\\
  \int_{\{s=\delta\}}v^2\,\rd\sigma_{m,a}
  &\leq
  C\int_{\{s=b\tau_m\}}v^2\,\rd\sigma_{m,a}
  +C\delta\int_{\{b\tau_m<s<\delta\}}
  \lvert\na^{g_{m,a}} v\rvert^2\,\rd\mu_{g_{m,a}},
  \label{eq:stable-cone-strip-trace}
\end{align}
where $\na^{g_{m,a}}$ denotes the gradient on $\Om_{m,a}$ with respect to $g_{m,a}$. The same estimates hold at the other end of the cylinder.
\end{lemma}

\begin{proof}
By the standard density argument, it suffices to prove the estimates for smooth $v$. For fixed $t$, by the fundamental theorem of calculus and Cauchy--Schwarz, we have
\eq{\label{ineq:v(s,t)-v(s_0,t)}
  \lvert v(s,t)\rvert^2
  \leq
  2\lvert v(s_0,t)\rvert^2
  +2(s-s_0)\int_{s_0}^s\lvert\partial_sv(\xi,t)\rvert^2\,\rd\xi.
}
Integrating first in $t$ and then in $s\in(s_0,\delta)$ gives
\eq{
  \int_{s_0}^{\delta}\int v^2\,\rd t\,\rd s
  \leq
  2\delta\int v^2(s_0,t)\,\rd t
  +2\delta^2\int_{s_0}^{\delta}\int
  \lvert\partial_sv\rvert^2\,\rd t\,\rd s.
}
Letting $s=\delta$ in \eqref{ineq:v(s,t)-v(s_0,t)} and integrating in $t$ gives
\eq{
  \int v^2(\delta,t)\,\rd t
  \leq
  2\int v^2(s_0,t)\,\rd t
  +2\delta\int_{s_0}^{\delta}\int
  \lvert\partial_sv\rvert^2\,\rd t\,\rd s.
}
Thanks to \eqref{eq:stable-cone-bulk-equivalence}, these estimates can be transformed from the flat metric $g_0$ to $g_{m,a}$, so that \eqref{eq:stable-cone-strip-mass} and \eqref{eq:stable-cone-strip-trace} follows.
Replacing $s$ by $\frac{\pi}{2}-s$ proves the estimates at the other end, which completes the proof.
\end{proof}

Recalling \eqref{eq:stable-cone-limit-cylinder} and \eqref{defn:Om_m}. Fix any nonzero $f\in C_c^\infty(\Omega)$, we know that $\spt(f)\subset\Om_m$ for all sufficiently large $m$. Define a function $\tilde f_m$ on $M_m$ by letting $\tilde f_m(\Phi_{m,a}(s,t))\coloneqq f(s,t)$ for some $a\in\{1,\cdots,4\}$, then extending $\tilde f_m$ by zero to $M_m\setminus B_{m,a}$. Clearly, $\tilde f_m\in H^1(M_m)$.
The corresponding Rayleigh quotient of $\tilde f_m$ is
\eq{
&\frac{\int_{M_m}\left(\abs{\na^{M_m}\tilde f_m}^2+\frac{1}{8}\abs{A_{M_m}}^2\tilde f_m^2\right)\rd\mu_{M_m}}{\int_{M_m}\tilde f^2_m\rd\mu_{M_m}}\\
=&\frac{\int_{\Om_m}\left(\abs{\na^{g_{m,a}}f}^2+\frac{1}{8}\left(\abs{A_{M_m}}^2\circ\Phi_{m,a}\right)f^2\right)\sqrt{{\rm det}g_{m,a}}\rd s\rd t}{\int_{\Om_m}f^2\sqrt{{\rm det}g_{m,a}}\rd s\rd t}.
}
By \eqref{eq:stable-cone-bulk-convergence} we see, for all sufficiently large $m$, $\lambda_m$ defined as \eqref{eq:stable-cone-rayleigh-quotient} satisfies the uniform bound $\lambda_m\leq C_f$ with $C_f>0$ depending on $f$ but independent of $m$.

Now we consider a subsequence (still indexed by $m$) whose limit is exactly $\liminf_{m\ra\infty}\lambda_m$, and denote by $u_m>0$ the corresponding first eigenfunction of this subsequence, normalized by $\int_{M_m}u_m^2\,\rd\mu_{M_m}=1$.
Then
\eq{\label{eq:stable-cone-eigenfunction-energy}
  \int_{M_m}
  \left(\lvert\na^{M_m} u_m\rvert^2
  +\frac18\lvert A_{M_m}\rvert^2u_m^2\right)\rd\mu_{M_m}
  =\lambda_m\leq C_f.
}
Set for each $a\in\{1,\cdots,4\}$ the function
\eq{
v_m^{(a)}
=u_m\circ\Phi_{m,a}\quad\text{on }\Omega_m.
}
Choose $\de_k\searrow0$ with $\de_k<\frac{\pi}{8}$ and set an exhaustion of $\Om$:
\eq{\label{defn:K_k}
K_k
=[\delta_k,\frac{\pi}{2}-\delta_k]\times(\mathbb R/2\pi\mathbb Z),\quad
K_1\subset K_2\subset\cdots,\quad\bigcup_kK_k=\Om.
}
Note that for fixed $k$, one has $K_k\subset\Omega_m$ for all large $m$. By \eqref{eq:stable-cone-bulk-equivalence} and \eqref{eq:stable-cone-eigenfunction-energy}, there holds
\eq{\label{ineq:v_m^a-W^1,2-upper-bound}
  \sum_{a=1}^4\lVert v_m^{(a)}\rVert_{H^1(K_k)}^2
  \leq C,
}
where $C$ is independent of $k$ and $m$. By Rellich compactness and a diagonal argument, there exist functions $v^{(a)}\in H^1_{\mathrm{loc}}(\Omega)$ such that, on every $K_k$,
\eq{\label{eq:stable-cone-local-eigenfunction-convergence}
  v_m^{(a)}\rightharpoonup v^{(a)}
  \quad\text{weakly in }H^1(K_k),
  \qquad
  v_m^{(a)}\ra v^{(a)}
  \quad\text{strongly in }L^2(K_k).
}

\begin{lemma}\label{lem:stable-cone-no-loss}
The limit functions belong to $H_0^1(\Omega)$ and satisfy
\eq{\label{eq:stable-cone-limit-mass}
  \sum_{a=1}^4\int_\Omega\bigl(v^{(a)}\bigr)^2\,\rd s\,\rd t=1.
}
\end{lemma}

\begin{proof}
By \eqref{eq:stable-cone-neck-coercivity} and \eqref{eq:stable-cone-eigenfunction-energy}, we have
\eq{\label{eq:stable-cone-neck-vanishing}
  \int_{N_m}u_m^2\,\rd\mu_{M_m}=O(\tau_m^2),
  \quad
  \int_{\Gamma_m}u_m^2\,\rd\sigma_{M_m}=O(\tau_m).
}
It follows that for a fixed $\de$,
\eq{
\de\sum^4_{a=1}\left(\int_{\{s=b\tau_m\}}(v_m^{(a)})^2\rd\sigma_{m,a}+\int_{\{s=\frac{\pi}{2}-b\tau_m\}}(v_m^{(a)})^2\rd\sigma_{m,a}\right)
=\de\int_{\Gamma_m}u^2_m\rd\sigma_{M_m}
=O(\de\tau_m),
}
which turns to $0$ as $m\ra\infty$ (recalling $\tau_m=\frac{1}{2m}$). Using \eqref{eq:stable-cone-eigenfunction-energy}, we also find
\eq{
&\de^2\sum^4_{a=1}\left(\int_{\{b\tau_m<s<\de\}}\abs{\na^{g_{m,a}} v_m^{(a)}}^2\rd\mu_{g_{m,a}}+\int_{\{\frac{\pi}{2}-\de<s<\frac{\pi}{2}-b\tau_m\}}\abs{\na^{g_{m,a}} v_m^{(a)}}^2\rd\mu_{g_{m,a}}\right)\\
\leq&\de^2\int_{M_m}\abs{\na^{M_m} u_m}^2\rd\mu_{M_m}
\leq C_f\de^2.
}
Applying \eqref{eq:stable-cone-strip-mass} at both ends of every bulk cylinder, then summing up.
Using the above estimates we obtain
\eq{\label{ineq:limsup-bdry-strip}
  \limsup_{m\to\infty}
  \sum_{a=1}^4
  \int_{\{s<\delta\}\cup\{s>\frac{\pi}{2}-\delta\}}
  \bigl(v_m^{(a)}\bigr)^2\,\rd\mu_{g_{m,a}}
  \leq C\delta^2.
}
By \eqref{eq:stable-cone-surface-decomposition}, we can write
\eq{
1
=&\int_{M_m}u_m^2\rd\mu_{M_m}\\
=&\int_{N_m}u_m^2\rd\mu_{M_m}+\sum_{a=1}^4\int_{\{\de<s<\frac{\pi}{2}-\de\}}(v_m^{(a)})^2\rd\mu_{g_{m,a}}
+\sum_{a=1}^4\int_{\{s<\de\}\cup\{s>\frac{\pi}{2}-\de\}}(v_m^{(a)})^2\rd\mu_{g_{m,a}},
}
thus by \eqref{eq:stable-cone-neck-vanishing} and \eqref{ineq:limsup-bdry-strip},
\eq{
1-C\de^2+O(\tau_m^2)
\leq\sum_{a=1}^4\int_{\{\de<s<\frac{\pi}{2}-\de\}}(v_m^{(a)})^2\rd\mu_{g_{m,a}}
\leq1.
}
By virtue of the convergences \eqref{eq:stable-cone-local-eigenfunction-convergence}, \eqref{eq:stable-cone-bulk-convergence}, we can first let $m\ra\infty$ then $\de\searrow0$, which proves \eqref{eq:stable-cone-limit-mass}.
For each $K_k$, by the lower semicontinuity (thanks to the weak convergence in \eqref{eq:stable-cone-local-eigenfunction-convergence}), in conjunction with the estimate \eqref{ineq:v_m^a-W^1,2-upper-bound}, we find
\eq{
\sum_{a=1}^4\int_{K_k}\left(\abs{v^{(a)}}^2+\abs{D v^{(a)}}^2\right)\rd s\rd t
\leq C,
}
where $C$ is independent of $k$ and $m$.
Since $\{K_k\}_{k\in\mbN}$ is an exhaustion of $\Om$, by monotonicity convergence theorem we thus conclude $v^{(a)}\in H^1(\Om)$ for all $a\in\{1,\cdots,4\}$.

Finally, we show that $v^{(a)}$ has zero boundary trace. To this end we fix $\delta\in(0,\frac{\pi}{8})$. Similar to the proof of \eqref{ineq:limsup-bdry-strip}, we can use \eqref{eq:stable-cone-strip-trace}, in conjunction with \eqref{eq:stable-cone-neck-vanishing}, \eqref{eq:stable-cone-eigenfunction-energy} and \eqref{eq:stable-cone-bulk-equivalence}, to deduce
\eq{\label{ineq:limsup-boundary-mass}
  \limsup_{m\to\infty}
  \int_{\{s=\delta\}\cup\{s=\frac{\pi}{2}-\de\}}\bigl(v_m^{(a)}\bigr)^2\,\rd t
  \leq C\delta.
}
Consider the cylinder
\eq{
U_\de
\coloneqq
\left(\frac{\de}{2},\de\right)
\times(\mbR/2\pi\mbZ).
}
Then $\overline{U_\de}\Subset\Om$, and
$U_\de\subset\Om_m$ for all sufficiently large $m$.
The boundary of $U_\de$ has two components, and restriction of the trace operator to the component
$\{s=\de\}\times S^1$ gives a bounded linear map
\eq{
T_\de:H^1(U_\de)\to L^2(S^1),
\quad
T_\de w=w(\de,\cdot).
}
Hence by \eqref{eq:stable-cone-local-eigenfunction-convergence}, we have
\eq{
v^{(a)}_m(\de,\cdot)\rightharpoonup v^{(a)}(\de,\cdot)\quad\text{ weakly in }L^2(S^1),
}
and thanks to \eqref{ineq:limsup-boundary-mass},
\eq{\label{ineq:v^a-boundary-slide-estimate}
\int_{\{s=\de\}}(v^{(a)})^2\rd t
\leq\liminf_{m\ra\infty}\int_{\{s=\de\}}(v^{(a)}_m)^2\rd t
\leq C\de.
}
By the above, we can view $v^{(a)}(\cdot,\cdot)$ as a $L^2(\mbR/2\pi\mbZ)$-valued function on $\de\in(0,\frac{\pi}{2})$, satisfying $v^{(a)}\in H^1\left((0,\frac{\pi}{2});L^2(\mbR/2\pi\mbZ)\right)\hookrightarrow
  C^0\left([0,\frac{\pi}{2}];L^2(\mbR/2\pi\mbZ)\right)$.
Thanks to \eqref{ineq:v^a-boundary-slide-estimate}, we thus conclude that $v^{(a)}(0,\cdot)=0$.
Similarly, one sees that $v^{(a)}(\frac{\pi}{2},\cdot)=0$, and hence $v^{(a)}\in H^1_0(\Om)$.
This completes the proof.
\end{proof}

\begin{proof}[Proof of Proposition \ref{thm:stable-cone-spectral-lower-bound}]

For each $k\in\mbN$, we recall the notation $K_k$ given by \eqref{defn:K_k}.
For a fixed $k$, when $m$ is sufficiently large, we have
\eq{
\lambda_m
=&\int_{M_m}\left(\abs{\na^{M_m} u_m}^2+\frac{1}{8}\abs{A_{M_m}}^2u_m^2\right)\rd\mu_{M_m}\\
\geq&\sum_{a=1}^4\int_{\Phi_{m,a}(K_k)}\left(\abs{\na^{M_m} u_m}^2+\frac{1}{8}\abs{A_{M_m}}^2u_m^2\right)\rd\mu_{M_m}.
}
By \eqref{eq:stable-cone-bulk-convergence} and \eqref{ineq:v_m^a-W^1,2-upper-bound}, there exists a sequence $\hat\varepsilon_{m}=\hat\varepsilon_m(k)$, which $\ra0$ as $m\ra\infty$, such that
\eq{
\lambda_m
\geq\sum_{a=1}^4\int_{K_k}\left(\abs{Dv_m^{(a)}}^2+\frac{1}{4}(v^{(a)}_m)^2\right)\rd s\rd t
-C\hat\varepsilon_m,
}
where $C>0$ is independent of $m$ and $k$.
By the weak convergence in \eqref{eq:stable-cone-local-eigenfunction-convergence} and the lower semi-continuity of $H^1$-norm, we find
\eq{
\liminf_{m\ra\infty}\lambda_m
\geq\sum_{a=1}^4\int^{\frac{\pi}{2}-\de_k}_{\de_k}\int_{\mbR/2\pi\mbZ}\left(\abs{D v^{(a)}}^2+\frac{1}{4}(v^{(a)})^2\right)\rd t\rd s.
}
Letting $k\ra\infty$, by the monotonicity convergence theorem, we get
\eq{
\liminf_{m\ra\infty}\lambda_m
\geq\sum_{a=1}^4\int_\Om\left(\abs{D v^{(a)}}^2+\frac{1}{4}(v^{(a)})^2\right)\rd s\rd t.
}
Since the one-dimensional Poincar\'e inequality in the $s$-variable yields
\begin{equation}\label{eq:stable-cone-cylinder-poincare}
  \int_\Om\abs{D v}^2\rd s\rd t
  \geq4\int_\Om v^2\rd s\rd t,
  \quad\forall v\in H_0^1(\Om),
\end{equation}
with equality for $v(s,t)=\sin(2s)$.
By virtue of the fact that $v^{(a)}\in H^1_0(\Om)$, in conjunction with \eqref{eq:stable-cone-limit-mass}, we thus conclude
\eq{
\liminf_{m\ra\infty}\lambda_m
\geq\sum^4_{a=1}\left(4\int_\Om(v^{(a)})^2\rd s\rd t+\frac{1}{4}\int_\Om (v^{(a)})^2\rd s\rd t\right)
=\frac{17}{4}
}
as desired.
This completes the proof.
\end{proof}

\providecommand{\bysame}{\leavevmode\hbox to3em{\hrulefill}\thinspace}
\providecommand{\MR}{\relax\ifhmode\unskip\space\fi MR }
\providecommand{\MRhref}[2]{\href{http://www.ams.org/mathscinet-getitem?mr=#1}{#2}
}
\providecommand{\href}[2]{#2}

\end{document}